\documentclass[reqno, a4paper]{amsart}

\usepackage{fixltx2e}
\usepackage[T1]{fontenc}
\usepackage[ngerman,british]{babel}
\addto\extrasbritish{\languageshorthands{ngerman}}
\useshorthands{"}
\usepackage{graphicx}
\usepackage{amsfonts}
\usepackage{amssymb}
\usepackage{amsthm}
\usepackage{amsmath}
\usepackage[all]{xy}
\usepackage{hyperref}
\usepackage{mathbbol}
\usepackage{calrsfs}
\let\smallhorn=\wedge
\usepackage{mathabx}
\usepackage{alphalph}

\renewcommand{\square}{\raisebox{.4mm}{\,\ensuremath{\boxvoid}}}

\theoremstyle{plain}
\newtheorem{theorem}[subsection]{Theorem}
\newtheorem{lemma}[subsection]{Lemma}
\newtheorem{proposition}[subsection]{Proposition}
\newtheorem{corollary}[subsection]{Corollary}

\theoremstyle{definition}
\newtheorem{definition}[subsection]{Definition}
\newtheorem{remark}[subsection]{Remark}
\newtheorem{notation}[subsection]{Notation}
\newtheorem{example}[subsection]{Example}

\newcommand{\comp}{\raisebox{0.2mm}{\ensuremath{\scriptstyle{\circ}}}}
\newcommand{\defn}{\textbf}
\newcommand{\del}{\partial}
\newcommand{\noproof}{\hfill \qed}
\newcommand{\To}{\Rightarrow}
\newcommand{\downmapsto}{\rotatebox{270}{$\mapsto$}}
\newcommand{\horn}{\raisebox{.145mm}{\scalebox{1.3}[1.15]{$\smallhorn$}}\!}
\newcommand{\cycle}{\triangle}
\newcommand{\kuub}{\raisebox{.4mm}{\ensuremath{\boxvoid}}}
\newcommand{\squaredot}{\raisebox{.4mm}{\,\ensuremath{\boxdot}}}
\newcommand{\lb}{\langle}
\newcommand{\rb}{\rangle}
\newcommand{\lind}{\lgroup}
\newcommand{\rind}{\rgroup}

\DeclareMathOperator{\Centr}{Centr}
\DeclareMathOperator{\ext}{Ext}
\DeclareMathOperator{\cod}{cod}
\DeclareMathOperator{\Coker}{coker}
\DeclareMathOperator{\dom}{dom}
\renewcommand{\H}{\mathrm{H}}
\DeclareMathOperator{\Hom}{hom}
\DeclareMathOperator{\Ker}{ker}
\DeclareMathOperator{\op}{op}
\DeclareMathOperator{\pr}{pr}
\DeclareMathOperator{\Eq}{Eq}
\newcommand{\R}[1]{\Eq(#1)}
\newcommand{\RR}[1]{\Eq^2(#1)}
\newcommand{\Rthree}[1]{\Eq^3(#1)}
\newcommand{\Rn}[1]{\Eq^n(#1)}
\DeclareMathOperator{\Kernel}{Ker}
\newcommand{\K}[1]{\Kernel(#1)}
\newcommand{\Ktwo}[1]{\Kernel^2(#1)}
\newcommand{\Kn}[1]{\Kernel^n(#1)}
\newcommand{\Knminus}[1]{\Kernel^{n-1}(#1)}
\DeclareMathOperator{\Tors}{Tors}

\newcommand{\X}{\ensuremath{\mathcal{X}}}
\newcommand{\B}{\ensuremath{\mathcal{B}}}

\newcommand{\E}{\ensuremath{\mathcal{E}}}
\newcommand{\F}{\ensuremath{\mathcal{F}}}

\newcommand{\dd}{\ensuremath{\mathbb{d}}}
\newcommand{\ff}{\ensuremath{\mathbb{f}}}
\newcommand{\GG}{\ensuremath{\mathbb{G}}}
\newcommand{\KK}{\ensuremath{\mathbb{K}}}
\newcommand{\one}{\ensuremath{\mathbb{1}}}
\renewcommand{\tt}{\ensuremath{\mathbb{t}}}
\newcommand{\TT}{\ensuremath{\mathbb{T}}}
\newcommand{\xx}{\ensuremath{\mathbb{x}}}
\newcommand{\XX}{\ensuremath{\mathbb{X}}}
\newcommand{\yy}{\ensuremath{\mathbb{y}}}
\newcommand{\YY}{\ensuremath{\mathbb{Y}}}

\renewcommand{\a}{\ensuremath{\mathsf{a}}}
\newcommand{\ab}{\ensuremath{\mathsf{ab}}}
\newcommand{\Ab}{\ensuremath{\mathsf{Ab}}}
\newcommand{\Arr}{\ensuremath{\mathsf{Arr}}}
\newcommand{\arr}{\ensuremath{\mathsf{arr}}}
\newcommand{\Arrn}{\ensuremath{\mathsf{Arr}^{n}}}
\newcommand{\Arrnn}{\ensuremath{\mathsf{Arr}^{n+1}}}

\newcommand{\CExt}{\ensuremath{\mathsf{CExt}}}
\newcommand{\cosk}{\ensuremath{\mathsf{cosk}}}
\newcommand{\Cosk}{\ensuremath{\mathsf{Cosk}}}
\newcommand{\D}{\ensuremath{\mathsf{D}}}
\newcommand{\Dd}{\ensuremath{\mathsf{d}}}
\newcommand{\Diag}[1]{\ensuremath{3^{#1}\text{-}\mathsf{Diag}}}
\newcommand{\Dis}{\ensuremath{\mathsf{Dis}}}
\newcommand{\Ext}{\ensuremath{\mathsf{Ext}}}
\newcommand{\Extn}{\ensuremath{\mathsf{Ext}^{n}}}
\newcommand{\Gp}{\ensuremath{\mathsf{Gp}}}
\newcommand{\Gpd}{\ensuremath{\mathsf{Gpd}}}
\newcommand{\Fun}{\ensuremath{\mathsf{Fun}}}
\renewcommand{\L}{\ensuremath{\mathsf{L}}}

\newcommand{\Nil}{\ensuremath{\mathsf{Nil}}}
\newcommand{\nil}{\ensuremath{\mathsf{nil}}}
\newcommand{\Ran}{\ensuremath{\mathsf{Ran}}}
\renewcommand{\S}{\ensuremath{\mathsf{s}}}
\newcommand{\Set}{\ensuremath{\mathsf{Set}}}
\newcommand{\simparr}{\ensuremath{\mathsf{s}}}
\newcommand{\SimpArr}{\ensuremath{\mathsf{SArr}}}
\newcommand{\SimpCExt}{\ensuremath{\mathsf{SCExt}}}
\newcommand{\trunc}{\ensuremath{\mathsf{tr}}}

\def\pullback{
 \ar@{-}[]+R+<6pt,-1pt>;[]+RD+<6pt,-6pt>%
 \ar@{-}[]+D+<1pt,-6pt>;[]+RD+<6pt,-6pt>}

\def\pullbackdots{%
 \ar@{.}[]+R+<6pt,-1pt>;[]+RD+<6pt,-6pt>%
 \ar@{.}[]+D+<1pt,-6pt>;[]+RD+<6pt,-6pt>}
 
\def\pushout{%
 \ar@{-}[]+L+<-6pt,1pt>;[]+LU+<-6pt,6pt>%
 \ar@{-}[]+U+<-1pt,6pt>;[]+LU+<-6pt,6pt>}

\def\splitpullback{%
 \ar@{-}[]+R+<6pt,-.51ex>;[]+RD+<6pt,-6pt>%
 \ar@{-}[]+D+<.51ex,-6pt>;[]+RD+<6pt,-6pt>}

\def\skewpullback{%
 \ar@{-}[]+R+<6pt,-1pt>;[]+RD+<-7pt,-8pt>%
 \ar@{-}[]+D+<-12pt,-8pt>;[]+RD+<-7pt,-8pt>}

\newdir{>>}{{}*!/3.5pt/:(1,-.2)@^{>}*!/3.5pt/:(1,+.2)@_{>}*!/7pt/:(1,-.2)@^{>}*!/7pt/:(1,+.2)@_{>}}
\newdir{ >>}{{}*!/8pt/@{|}*!/3.5pt/:(1,-.2)@^{>}*!/3.5pt/:(1,+.2)@_{>}}
\newdir{ |>}{{}*!/-3.5pt/@{|}*!/-8pt/:(1,-.2)@^{>}*!/-8pt/:(1,+.2)@_{>}}
\newdir{ >}{{}*!/-8pt/@{>}}
\newdir{>}{{}*:(1,-.2)@^{>}*:(1,+.2)@_{>}}
\newdir{<}{{}*:(1,+.2)@^{<}*:(1,-.2)@_{<}}

\begin{document}

\title[Higher central extensions and cohomology]{Higher central extensions\\ and cohomology}

\author{Diana Rodelo}
\author{Tim Van~der Linden}

\thanks{Diana Rodelo was supported by the Centro de Matem\'{a}tica da
Universidade de Coimbra (CMUC), funded by the European Regional
Development Fund through the program COMPETE and by the Portuguese
Government through the FCT--Funda\c{c}\~{a}o para a Ci\^{e}ncia e
a Tecnologia under the project PEst-C/MAT/UI0324/2013 and grants
number PTDC/MAT/120222/2010 and SFRH/BPD/69661/2010}
\thanks{Tim Van der Linden is a Research Associate of the Fonds de la Recherche Scientifique--FNRS. He was supported by Centro de Matem\'atica da Universidade de Coimbra (CMUC) and by FCT--Funda\c c\~ao para a Ci\^encia e a Tecnologia (under grant number SFRH/BPD/38797/2007) and wishes to thank the Janelidze family for its warm hospitality during his stay in Cape Town}

\email{drodelo@ualg.pt}
\email{tim.vanderlinden@uclouvain.be}

\address{Departamento de Matem\'atica, Faculdade de Ci\^{e}ncias e Tecnologia, Universidade do Algarve, Campus de
Gambelas, 8005--139 Faro, Portugal}
\address{Centro de Matem\'atica, Universidade de Coimbra, 3001--454 Coimbra, Portugal}
\address{Institut de Recherche en Math\'ematique et Physique, Universit\'e catholique de Louvain, chemin du cyclotron~2 bte~L7.01.02, 1348 Louvain-la-Neuve,
Belgium}

\keywords{Cohomology, categorical Galois theory, semi-abelian category, higher central extension, torsor}

\subjclass[2010]{18G50, 18G60, 18G15, 20J, 55N}

\begin{abstract}
We establish a Galois-theoretic interpretation of cohomology in semi-abelian categories: cohomology with trivial coefficients classifies central extensions, also in arbitrarily high degrees. This allows us to obtain a duality, in a certain sense, between ``internal'' homology and ``external'' cohomology in semi-abelian categories. These results depend on a geometric viewpoint of the concept of a higher central extension, as well as the algebraic one in terms of commutators.
\end{abstract}

\maketitle

\tableofcontents

\section*{Introduction}
This article exposes a hidden duality between ``internal'' homology and ``external'' cohomology for certain group-like structures: we prove that cohomology with trivial coefficients classifies (higher) central extensions. Together with the work in low dimensions and with several closely related results in homology theory, this reveals a deep connection between Galois theory and cohomology, and a close link with homology which has been invisible so far.

The context in which we work is sufficiently general to cover cohomology of, say, groups, crossed modules, Lie algebras and non-unitary rings, as well as the Yoneda $\ext$ groups in the abelian case, and many new examples can easily be added to the list. In fact, almost any semi-abelian category would do, as long as it satisfies a certain commutator condition which occurs naturally in this setting---see below.

This interpretation of cohomology is part of a bigger programme which intends to understand homological algebra in a non-abelian environment from the viewpoint of (categorical) Galois theory. Related results include, for instance, higher Hopf formulae for homology in semi-abelian categories~\cite{EGVdL}, higher-dimensional torsion theories~\cite{Everaert-Gran-TT}, a theory of satellites for homology without projectives~\cite{GVdL2}, and higher-dimensional commutator theory based on a notion of higher centrality~\cite{EverVdL4, EverVdLRCT}.

\subsection*{Higher centrality}
The key novelty in the present approach to (co)homology of non-abelian algebraic objects is the concept of \emph{higher centrality}. It allows us to express in an abstract but simple way the commutator conditions which we have to deal with.

Following the ideas of Janelidze~\cite{Janelidze:Double, Janelidze:Hopf-talk}, the formal theory of (not necessarily central) \emph{higher (cubic) extensions} was first developed in~\cite{EGVdL} in order to provide a general setting for the Brown--Ellis--Hopf formulae~\cite{Brown-Ellis, Donadze-Inassaridze-Porter}. The notion of \emph{centrality} in the sense of categorical Galois theory~\cite{Borceux-Janelidze, Janelidze:Pure, Janelidze-Kelly} depends on a Galois structure, and accordingly, centrality of higher extensions is defined using a tower of Galois structures.

Let us make this explicit with a concrete example. Consider the category~$\Gp$ of all groups and its (reflective) subcategory $\Nil_{2}$ determined by all groups of nilpotency class at most $2$. The induced reflector $\nil_{2}\colon{\Gp\to \Nil_{2}}$, left adjoint to the inclusion functor, takes a group $G$ and sends it to its $2$-nilpotent quotient~$G/[[G,G],G]$. This situation---$\Gp$ being a variety of algebras over $\Set$, and~$\Nil_{2}$ a subvariety of it---admits a canonical homology theory: Barr--Beck comonadic homology~\cite{Barr-Beck} with coefficients in the reflector $\nil_{2}$. Now for any group $Z$, the induced third homology group $\H_{3}(Z,\nil_{2})$ of~$Z$ may be expressed by a Hopf formula, namely the quotient~\cite[Theorem~9.3]{EGVdL}
\[
\frac{K_{0}\cap K_{1}\cap [[X,X],X]}{[[K_{0}\cap K_{1},X],X][[K_{0}, K_{1}],X][[K_{0},X],K_{1}][[X,K_{0}],K_{1}][[X,X],K_{0}\cap K_{1}]}.
\]
Here the objects $K_{0}=\K{c}$ and $K_{1}=\K{d}$ are the kernels of $c$ and $d$, for any \emph{two-cubic presentation}
\begin{equation}\label{Double-Extension-Intro}
\vcenter{\xymatrix{X \ar@{ >>}[r]^-{c} \ar@{ >>}[d]_-{d} & C \ar@{ >>}[d] \\
D \ar@{ >>}[r] & Z}}
\end{equation}
of $Z$. This means that the objects $C$, $D$ and $X$ are projective (= free) groups, and furthermore this commutative square is a \emph{two-cubic extension} of $Z$: all its arrows, as well as the induced arrow to the pullback $\lind d,c\rind\colon{X\to D\times_{Z}C}$, are surjections. The denominator in the formula is a generalised commutator: a two-cubic extension of groups such as~\eqref{Double-Extension-Intro} is central (with respect to $\Nil_{2}$) precisely when this denominator is zero. The concept of \emph{centrality} of two-cubic extensions is given by the Galois structure $\Gamma_{1}$ in the ``tower'' consisting of
\[
\Gamma_{0}=(\Gp,\Nil_{2},\E,\F,\nil_{2},\subseteq)
\] 
and
\[
\Gamma_{1}=(\Ext(\Gp),\CExt_{\Nil_{2}}(\Gp),\E^{1},\F^{1},(\nil_{2})_{1},\subseteq),
\] 
where $\E$, $\F$ are the classes of surjections and $\E^{1}$, $\F^{1}$ are the classes of two-cubic extensions in $\Gp$ and in $\Nil_{2}$, respectively. Here $\Gamma_{1}$ is induced by $\Gamma_{0}$ through its one-cubic central extensions, which are the objects of the full reflective subcategory $\CExt_{\Nil_{2}}(\Gp)$ with reflector $(\nil_{2})_{1}$ of the category $\Ext(\Gp)$ of one-cubic extensions in $\Gp$. 

It is not hard to construct a two-cubic presentation of an object, certainly not in the varietal case, since a truncation of any simplicial projective resolution will~do. As is apparent from the formula, the main difficulty in making it explicit lies in characterising the (two-cubic) central extensions corresponding to the functor which is being derived (in this case,~$\nil_{2}$). Higher cubic central extensions are defined by induction; let us explain how this is done for lowest degrees (more details can be found in the following sections and in the articles~\cite{EverHopf, EGoeVdL, EGVdL}, amongst others).

A \defn{semi-abelian} category~\cite{Janelidze-Marki-Tholen, Borceux-Bourn} is pointed, Barr-exact~\cite{Barr} and Bourn-pro\-to\-mod\-ular~\cite{Bourn1991} with binary sums. Let $\X$ be a semi-abelian category and $\B$ a \defn{Birkhoff subcategory}~\cite{Janelidze-Kelly} of $\X$---full, reflective and closed under subobjects and regular quotients, so that a Birkhoff subcategory of a variety is nothing but a subvariety. Let
\begin{equation}\label{Adjunction-1}
\vcenter{\xymatrix{{\X} \ar@<1ex>[r]^-{I} \ar@{}[r]|-{\perp} & \B \ar@<1ex>[l]^-{\supset}}}
\end{equation}
denote the induced adjunction, with $I$ the reflector and $\eta\colon{1_{\X}\To I}$ the unit.

In this context, a \defn{cubic extension} $f\colon{X\to Z}$ is defined as a regular epimorphism, and an \defn{extension} $(k,f)$ as a short exact sequence
\[
\xymatrix{0 \ar[r] & A \ar@{{ |>}->}[r]^-{k} & X \ar@{ >>}[r]^-{f} & Z \ar[r] & 0.}
\]
Together with the classes of cubic extensions in $\X$ and in $\B$, the adjunction~\eqref{Adjunction-1} forms a \emph{Galois structure} in the sense of Janelidze~\cite{Borceux-Janelidze, Janelidze:Pure}. Central (cubic) extensions are defined with respect to such a Galois structure, as follows. A cubic extension~$f$ is called \defn{trivial} when the induced naturality square
\[
\vcenter{\xymatrix{X \ar@{ >>}[r]^-{f} \ar@{ >>}[d]_-{\eta_{X}} & Z \ar@{ >>}[d]^{\eta_{Z}} \\ IX \ar@{ >>}[r]_-{If} & IZ}}
\]
is a pullback; $f\colon{X\to Z}$ is \defn{central} when either of the kernel pair projections $\pr_{0}$, $\pr_{1}\colon \R{f}=X\times_Z X\to X$ is trivial~\cite{Janelidze-Kelly}. An extension $(k,f)$ is said to be \defn{trivial} or \defn{central} whenever so is the underlying cubic extension $f$.

\begin{figure}
$\vcenter{\xymatrix@!0@=3.5em{& 0 \ar@{.>}[d] & 0 \ar@{->}[d] & 0 \ar@{->}[d]\\
0 \ar@{->}[r] & A  \ar@{{ |>}->}[r] \ar@{{ |>}.>}[d] & \cdot \ar@{{ |>}->}[d] \ar@{-{ >>}}[r] & \cdot \ar@{{ |>}->}[d] \ar@{->}[r] & 0\\
0 \ar@{->}[r] & \cdot \ar@{.>}[rd]  \ar@{{ |>}->}[r] \ar@{-{ >>}}[d] & F_{2} \ar@{-{ >>}}[r]_-{f_{0}} \ar@{-{ >>}}[d]^-{f_{1}}  & \cdot \ar@{-{ >>}}[d] \ar@{->}[r] & 0\\
0 \ar@{->}[r] & \cdot \ar@{{ |>}->}[r] \ar[d] & \cdot \ar@{.{ >>}}[r] \ar@{->}[d] & Z \ar@{->}[d] \ar@{.>}[r] & 0\\
& 0 & 0 & 0 }}$
\caption{A two-fold extension is a $3\times 3$-diagram: all rows and columns are short exact sequences.}\label{3x3 diag}
\end{figure}
It turns out that the cubic central extensions relative to $\B$ determine a reflective subcategory~$\CExt_\B(\X)$ of the category $\Ext(\X)$ of cubic extensions in $\X$, so we have an adjunction
\[
\xymatrix{{\Ext(\X)} \ar@<1ex>[r]^-{I_{1}} \ar@{}[r]|-{\perp} & \CExt_{\B}(\X).\ar@<1ex>[l]^-{\supset}}
\]
Together with classes of two-cubic extensions, defined as in the case of groups above, this adjunction forms a Galois structure, and thus we acquire the notion of \emph{two-cubic central extension} with respect to $\B$. This construction may be repeated ad infinitum, so that notions of \emph{$n$-cubic extension} (special $n$-dimensional cubes in $\X$) and \emph{$n$-cubic central extension} are obtained. An \defn{$n$-fold (central) extension} will be a special diagram of short exact sequences: a $3^{n}$-diagram, which is essentially an $n$-cubic extension with chosen kernels. For instance, a two-fold extension is a short exact sequence of short exact sequences---a $3\times 3$-diagram as in Figure~\ref{3x3 diag}, where the bottom right square is the underlying two-cubic extension. The dotted arrows in this diagram form a \emph{Yoneda extension}~\cite{Yoneda-Exact-Sequences} from $A$ to $Z$, which in the abelian case allows to reconstruct the entire $3\times 3$-diagram up to equivalence. In general, though, the diagram may not be thus reduced without loss of information. See Figure~\ref{Figure-Direction} below for a picture in dimension three.

Of course, whether or not an $n$-cubic extension is central with respect to some chosen Birkhoff subcategory depends on this subcategory more than anything else. In many cases (like, for instance, the case of groups vs.\ $2$-nilpotent groups) there are explicit descriptions of the central extensions in some, or in all, degrees (see, for instance, \cite{CKLVdL, Everaert-Gran-TT, Everaert-Gran-nGroupoids, EverVdL3}). Knowing, in a given case, what the central extensions are, gives a complete description of the corresponding homology objects as higher Hopf formulae: this is the content of Theorem~8.1 in~\cite{EGVdL}. 
In this article we only consider cubic extensions which are central with respect to the Birkhoff subcategory ${\B=\Ab(\X)}$ of all \emph{abelian} objects in $\X$, the objects which admit an internal abelian group structure; that is to say, they are central with respect to abelianisation. The reason for this constraint is that we only treat cohomology with trivial coefficients---coefficients in trivial modules, which are precisely the internal abelian group objects. In the non-trivial case, where the theory involves Birkhoff subcategories of~$\Ab(\X)$, the situation becomes significantly more complicated, and forms the subject of current work-in-progress.

The Hopf formulae now take the following shape~\cite{EGVdL}: 
\begin{equation}\label{General-Hopf}
\H_{n+1}(Z,\Ab(\X))\cong\dfrac{\bigcap_{i\in n}\K{f_{i}}\cap \lb F_{n}\rb}{L_{n}[F]}
\end{equation}
for any $n$-presentation~$F$ of $Z$. Here $F_{n}$ is the ``initial object'' of the $n$-cubic extension $F$ and the $f_{i}$ are the ``initial arrows'' (see the solid part of Figure~\ref{Figure-Direction} for a picture in degree three). The brackets~$\lb-\rb$ in the formula give the \emph{zero-dimensional commutator} of $F_{n}$ determined by its abelianisation: for any object $X$ of $\X$ there is a short exact sequence
\begin{equation}\label{Angular Bracket}
\xymatrix@1{0 \ar[r] & \lb X\rb \ar@{{ |>}->}[r] & X \ar@{-{ >>}}[r]^-{\eta_{X}} & \ab X \ar[r] & 0,}
\end{equation}
so $\lb X\rb=[X,X]$, the Huq commutator~\cite{BG,Huq} of $X$ with itself, giving a functor $\lb - \rb\colon {\X\to \X}$. The object in the denominator of the Hopf formula is the smallest normal subobject of~$F_{n}$ which, when divided out, makes $F$ central; in other words, an $n$-cubic extension~$F$ is central if and only if $L_{n}[F]=0$. In many cases (see Section~\ref{Section-Commutator-Assumption}) this ``abstract higher-dimensional commutator'' may be computed as a join of binary Huq commutators~\cite{EGVdL, RVdL3}.

On the other hand, the use of higher (central) extensions is not at all limited to homology and Hopf formulae. The concept of higher (cubic) extension is quite interesting in its own right~\cite{EGoeVdL} while centrality may, for instance, be used to model more exotic commutator theories~\cite{EverVdL4, EverVdLRCT}. The present article is meant to clarify the connection with cohomology and obtain a higher-dimensional counterpart of the the low-dimensional work which has been done in this context~\cite{Bourn1999, Bourn-Janelidze:Torsors, Gran-VdL, RVdL}.

\subsection*{Cohomology and centrality}
The current development starts with the long-established interpretation of the second cohomology group $\H^{2}(Z,A)$ of a group $Z$ with coefficients in an abelian group~$A$ in terms of central extensions of $Z$ by $A$ (see for instance~\cite{MacLane:Homology}). A central extension $(k,f)$ of $Z$ by $A$ is a short exact sequence of groups
\begin{equation}\label{1-fold extension}
\xymatrix{0 \ar[r] & A \ar@{{ |>}->}[r]^-{k} & X \ar@{ >>}[r]^-{f} & Z \ar[r] & 0,}
\end{equation}
so $k=\ker f$ and $f=\Coker k$, such that the commutator $[A,X]$ is trivial:
\[
\text{$axa^{-1}x^{-1}=1$ for all $a\in A$ and~${x\in X}$.} 
\]
Two extensions $(k,f)\colon {A\to X\to Z}$ and $(k',f')\colon{A\to X'\to Z}$ of $Z$ by $A$ are said to be equivalent if and only if there exists a group (iso)morphism $i\colon{X\to X'}$ satisfying ${f'\comp i=f}$ and~$i\comp k=k'$. The corresponding equivalence classes, together with the so-called Baer sum, form an abelian group $\Centr^{1}(Z,A)$, and this group is isomorphic to~$\H^{2}(Z,A)$. 

In this article we generalise the isomorphism $\Centr^{1}(Z,A)\cong \H^{2}(Z,A)$ in two ways: first of all, we also consider higher cohomology groups; secondly, we replace the concrete context of groups by an abstract context of a semi-abelian category satisfying an additional axiom which holds in many important examples of semi-abelian categories---in particular, it holds in the category of groups and in any abelian category.

It was proved in~\cite{Gran-VdL}, see also~\cite{Bourn:Baer-Sums-2007} and~\cite{Bourn-Janelidze:Torsors}, that the classical interpretation of group cohomology via central extensions may be extended from the context of groups to semi-abelian categories. Here the concept of centrality is the one coming from Galois theory, using the Birkhoff subcategory of all abelian objects~\cite{Bourn-Gran}---that is, we use centrality relative to abelianisation---and the cohomology is comonadic cohomology~\cite{Barr-Beck}. Thus the well-known similar results for Lie algebras over a field, commutative algebras, non-unitary rings, (pre)crossed modules, etc.\ could be included in a general theory, and new examples could be studied. 

When $A$ is a $Z$-module with a non-trivial action, of course the above concept of central extension does no longer suffice to capture cohomology with coefficients in~$A$. Nevertheless, there is good hope that something can be done in general and in higher degrees which extends both the present work and the results in~\cite{Bourn-Janelidze:Torsors}. In the current paper we limit ourselves to the case of trivial module coefficients essentially for the sake of simplicity. In contrast with this potential extension of the theory, when $A$ is not even an abelian object---so when we enter the realm of true non-abelian cohomology---it is not clear at all how the current approach could form the basis of a new theory.

The next step, an interpretation of the third cohomology group in similar terms, turned out to be quite hard to take. The reason is that one needs a theory of higher central extensions for this---which until recently was unavailable. The problem was finally solved in~\cite{RVdL}, where the characterisation of two-cubic central extensions given in~\cite{Janelidze:Double, Gran-Rossi} is extended to semi-abelian categories and an isomorphism
\[
\H^{3}(Z,A)\cong \Centr^{2}(Z,A)
\]
is constructed. The abelian group $\Centr^{2}(Z,A)$ consists of equivalence classes of two-fold central extensions of an object~$Z$ by an abelian object~$A$ as in Figure~\ref{3x3 diag}, equipped with a canonical addition induced by the internal group structure of~$A$. It must be mentioned that the cohomology theory used in~\cite{RVdL}---the \emph{directions approach}, using internal $n$-fold groupoids, introduced by Bourn in~\cite{Bourn1999, Bourn2002b, Bourn:Baer-Sums-2007} and further worked out by Bourn and Rodelo in~\cite{Bourn-Rodelo, Rodelo-Directions}---is less classical than the one we shall be using here, or at least is not obviously related to it in higher degrees. In this paper we use a different interpretation of cohomology which is based on higher torsors~\cite{Duskin, Duskin-Torsors, Glenn}.

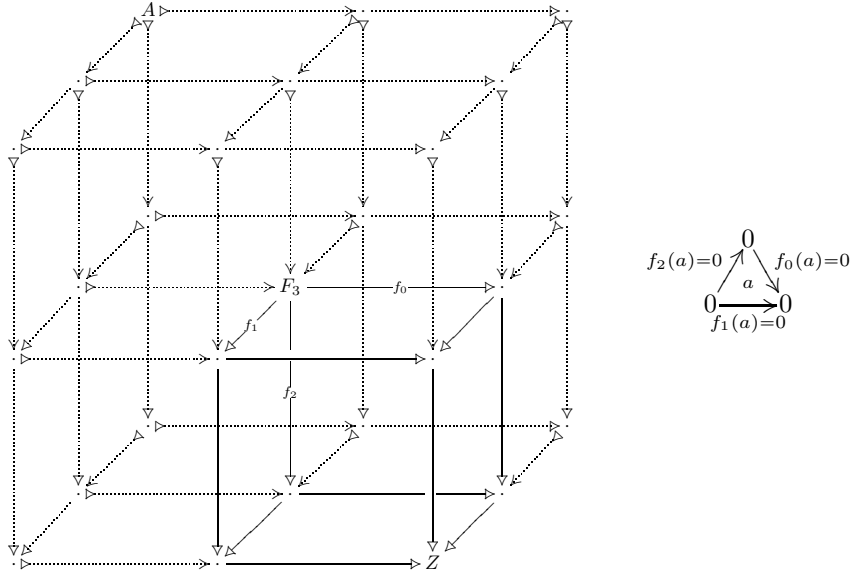
\begin{figure}
\resizebox{.6\textwidth}{!}{
$\vcenter{\xymatrix{&& A \ar@{{ |>}.>}[ld] \ar@{{ |>}.>}[rrr] \ar@{{ |>}.>}[ddd] &&& \cdot \ar@{{ |>}.>}[ld] \ar@{{ |>}.>}[ddd] \ar@{.{ >>}}[rrr] &&& \cdot \ar@{{ |>}.>}[ld] \ar@{{ |>}.>}[ddd]\\
&{\cdot} \ar@{.{ >>}}[ld] \ar@{{ |>}.>}[rrr] \ar@{{ |>}.>}[ddd] &&& \cdot \ar@{.{ >>}}[ld] \ar@{{ |>}.>}[ddd] \ar@{.{ >>}}[rrr] &&& \cdot \ar@{.{ >>}}[ld] \ar@{{ |>}.>}[ddd]\\
{\cdot} \ar@{{ |>}.>}[rrr] \ar@{{ |>}.>}[ddd] &&& \cdot \ar@{{ |>}.>}[ddd] \ar@{.{ >>}}[rrr] &&& \cdot \ar@{{ |>}.>}[ddd]\\
&&{\cdot} \ar@{{ |>}.>}[ld] \ar@{{ |>}.>}[rrr] \ar@{.{ >>}}[ddd] &&& \cdot \ar@{{ |>}.>}[ld] \ar@{.{ >>}}[ddd] \ar@{.{ >>}}[rrr] &&& \cdot \ar@{{ |>}.>}[ld] \ar@{.{ >>}}[ddd]\\
&{\cdot} \ar@{.{ >>}}[ld] \ar@{{ |>}.>}[rrr] \ar@{.{ >>}}[ddd] &&& F_{3} \ar@{-{ >>}}[ld]|-{f_{1}} \ar@{-{ >>}}[ddd]|(.33){\hole}|-{f_{2}} \ar@{-{ >>}}[rrr]|-{f_{0}} &&& \cdot \ar@{-{ >>}}[ld] \ar@{-{ >>}}[ddd]\\
{\cdot} \ar@{{ |>}.>}[rrr] \ar@{.{ >>}}[ddd] &&& \cdot \ar@{-{ >>}}[ddd] \ar@{-{ >>}}[rrr] &&& \cdot \ar@{-{ >>}}[ddd]\\
&&{\cdot} \ar@{{ |>}.>}[ld] \ar@{{ |>}.>}[rrr] &&& \cdot \ar@{{ |>}.>}[ld] \ar@{.{ >>}}[rrr] &&& \cdot \ar@{{ |>}.>}[ld] \\
&{\cdot} \ar@{.{ >>}}[ld] \ar@{{ |>}.>}[rrr] &&& \cdot \ar@{-{ >>}}[ld] \ar@{-{ >>}}[rrr]|(.66){\hole} &&& \cdot \ar@{-{ >>}}[ld]\\
{\cdot} \ar@{{ |>}.>}[rrr] &&& \cdot \ar@{-{ >>}}[rrr] &&& Z}}$}
\qquad
$\vcenter{\xymatrix@1@!0@R=2.4495em@C=1.4142em{& 0 \ar@{}[d]|(.67){a} \ar[rd]^-{f_{0}(a)=0}\\
0 \ar[ru]^-{f_{2}(a)=0} \ar[rr]_-{f_{1}(a)=0} && 0}}$
\caption{A $3$-fold (central) extension of $Z$ with its direction~$A$. The solid arrows are the underlying $3$-cubic extension.}\label{Figure-Direction}
\end{figure}

A key ingredient here is the concept of \defn{direction} of a higher (central) extension, which is the initial object of this extension $E$, when $E$ is considered as a diagram of short exact sequences in $\X$: the objects $A$ in Figure~\ref{3x3 diag} and Figure~\ref{Figure-Direction}. From the point of view of the $n$-cubic extension $F$ underlying $E$, the direction is an intersection of (chosen) kernels. For a one-fold extension such as~\eqref{1-fold extension}, the direction is the kernel $A=\K{f}$ of the underlying one-cubic extension $f\colon{X\to Z}$, while for a two-cubic extension as in Figure~\ref{3x3 diag} it is the intersection of the kernels~${\K{f_{0}}\cap \K{f_{1}}}$. Considering the underlying two-cubic extension $F$ as an arrow in the category of arrows in $\X$, the kernel of $F$ is a one-cubic extension in $\X$, whose kernel is isomorphic to the direction of $E$; so we write it as~$\Ktwo{F}$. In higher degrees a similar (inductive) analysis makes sense: given an $n$-fold extension $E$ with underlying $n$-cubic extension~$F$, its direction is $\Kn{F}=\bigcap_{i\in n}\K{f_{i}}$, which is an abelian object of~$\X$ when~$F$ is central (compare with the Hopf formula~\eqref{General-Hopf}).

Figure~\ref{Figure-Direction} gives a picture in degree $3$. The different but equivalent ways in which the direction may be obtained as a kernel come from the several ways in which a three-cubic extension may be considered as an arrow between two-cubic extensions, etc. An element of $F_{3}$ should be viewed as a (directed) triangle with faces given by $f_{0}$, $f_{1}$ and $f_{2}$, and such a triangle $a$ is in the direction~$A$ if and only if all its faces (edges) are zero.

We write $\CExt^{n}_{Z}(\X)$ for the category of $n$-fold central extensions over $Z$. Thus, for each $n\geq 1$ and any object $Z$ in $\X$, we obtain a functor
\[
\D_{(n,Z)}\colon\CExt^{n}_{Z}(\X)\to \Ab(\X)
\]
that sends an $n$-fold central extension $E$ of $Z$ to its direction $A$. Given any object~$Z$ in $\X$ and any abelian object $A$, an \defn{$n$-fold central extension of $Z$ by $A$} is an $n$-fold central extension $F$ of~$Z$ with direction~$A$, an object of the fibre $\D_{(n,Z)}^{-1}A$. Taking connected components gives us the (possibly large) set
\[
\Centr^{n}(Z,A)=\pi_{0}(\D_{(n,Z)}^{-1}A)
\]
which admits a canonical abelian group structure, since, as established in Proposition~\ref{Proposition-Centr^{n}(Z,-)}, the assignment $A\mapsto \Centr^{n}(Z,A)$ gives rise to a product-preserving functor $\Ab(\X) \to \Set$.

Now the question remains whether these groups have any cohomological meaning. The main body of this article explains that, indeed, they have: we shall prove that, under the commutator condition (CC), they agree with the interpretation of comonadic cohomology in terms of higher torsors.

\subsection*{Cohomology via higher torsors}
One could say that Duskin and Glenn's \emph{higher torsors}~\cite{Duskin, Duskin-Torsors, Glenn} are to central extensions what truncations of simplicial resolutions are to cubic extensions, or what groupoids are to pregroupoids:
\[
\frac{\text{torsor}}{\text{central extension}}=\frac{\text{truncation of simplicial resolution}}{\text{cubic extension}}=\frac{\text{groupoid}}{\text{pregroupoid}}.
\]
In a groupoid
\[
\xymatrix{G_{1} \ar@<-1ex>[r]_-{\del_{0}} \ar@<1ex>[r]^-{\del_{1}} &
G_{0} \ar[l]|-{\sigma_{0}}}
\]
there are identities (given by $\sigma_{0}$) and a composition $m$
\[
\vcenter{\xymatrix@1@!0@R=2.4495em@C=1.4142em{& {\cdot} \ar[rd]^-{\beta}\\
{\cdot} \ar[ru]^-{\alpha} \ar@{.>}[rr]_-{\gamma} && {\cdot}}}
\qquad\qquad
m(\beta,\alpha)=\gamma
\]
which is associative, admits inverses and is compatible with the identities; there is only one object of objects, $G_{0}$. On the other hand, a pregroupoid~\cite{Kock-Pregroupoids}
\[
\xymatrix@!0@=3em{& G_{1} \ar[dl]_-{\del_{0}} \ar[dr]^-{\del_{1}} \\
G_{0} && G'_{0}}
\]
has two objects of objects, $G_{0}$ and $G_{0}'$. Consequently, it has no identities, and instead of a composition, there is an (associative) Mal'tsev operation $p$
\[
\vcenter{\xymatrix@1@!0@=2em{& {\cdot}\\
{\cdot} \ar[ru]^-{\gamma} \ar@{.>}[rd]_-{\delta} && {\cdot} \ar[lu]_-{\beta} \ar[ld]^-{\alpha}\\
&{\cdot}}}
\qquad\qquad
p(\alpha,\beta,\gamma)=\delta
\]
satisfying $p(\alpha,\alpha,\gamma)=\gamma$ and $p(\alpha,\gamma,\gamma)=\alpha$. In the present context, associativity is automatic.

In dimension $3$ now, truncating a simplicial object $\XX$ at degree $2$, we obtain a diagram as on the left
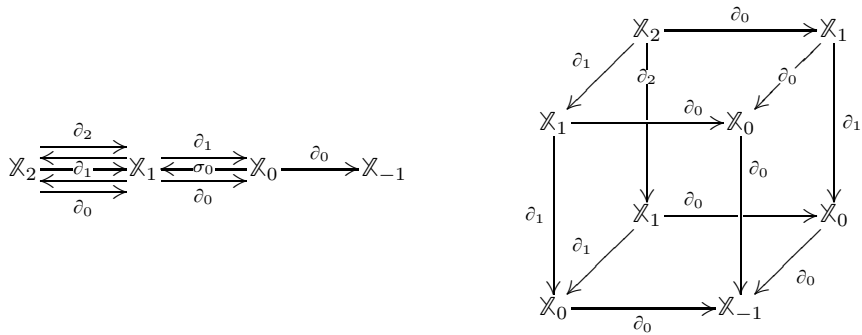
\[
\vcenter{\xymatrix@1@!0@=45pt{\XX_{2} \ar[r]|-{\del_{1}} \ar@<2ex>[r]^-{\del_{2}} \ar@<-2ex>[r]_-{\del_{0}} & \XX_{1} \ar@<-1ex>[l] \ar@<1ex>[l] \ar@<1ex>[r]^-{\del_{1}} \ar@<-1ex>[r]_-{\del_{0}} & \XX_{0} \ar[r]^-{\del_{0}} \ar[l]|-{\sigma_{0}} & \XX_{-1}}}
\qquad\qquad
\vcenter{\xymatrix@1@!0@=35pt{
& \XX_{2} \ar[rr]^-{\del_{0}} \ar[dd]|(.25){\del_{2}}|-{\hole} \ar[ld]_-{\del_{1}} && \XX_{1} \ar[dd]^-{\del_{1}} \ar[ld]|-{\del_{0}} \\
\XX_{1} \ar[rr]^(.75){\del_{0}} \ar[dd]_-{\del_{1}} && \XX_{0} \ar[dd]^(.25){\del_{0}} \\
& \XX_{1} \ar[ld]_-{\del_{1}} \ar[rr]^(.25){\del_{0}}|-{\hole} && \XX_{0} \ar[ld]^-{\del_{0}} \\
\XX_{0} \ar[rr]_-{\del_{0}} && \XX_{-1}}}
\]
which may be ``unfolded'' to a commutative cube as on the right.
The extension property of this cube corresponds to acyclicity of the given simplicial object $\XX$ (its being a resolution) up to degree $2$. Note that this cube is special, because certain objects in it occur several times; on the other hand, the cube does not capture the degeneracies present in the simplicial object. Groupoids (and torsors) live in the \emph{simplicial} world, whereas pregroupoids belong to the \emph{cubical} world of $n$-cubic (central) extensions. As we shall explain in Subsection~\ref{Maltsev}, higher central extensions may be considered as \emph{higher-dimensional pregroupoids} in some precise sense.

Given an object $Z$ and an abelian object $A$ in a semi-abelian category $\X$, we consider the augmented simplicial object $\KK(Z,A,n)$ determined by
\[
\resizebox{\textwidth}{!}{
$
\vcenter{\xymatrix@R=20pt@C=45pt{\scriptstyle{n+1} & \scriptstyle{n} & \scriptstyle{n-1} & \scriptstyle{n-2} \ar@{}[r]|-{\cdots} & \scriptstyle{0} & \scriptstyle{-1}\\
A^{n+1} \times Z \ar@<2.33ex>[r]^-{\del_{n+1}\times 1_{Z}} \ar@<1.16ex>[r]|-{\pr_{n}\times 1_{Z}} \ar@<-2.33ex>[r]_-{\pr_{0}\times 1_{Z}}^-{\vdots} &
 A \times Z \ar@<1.75ex>[r]^-{\pr_{Z}} \ar@<-1.75ex>[r]_-{\pr_{Z}}^-{\vdots} &
 Z \ar@{=}@<1.75ex>[r] \ar@{=}@<-1.75ex>[r]^-{\vdots} & Z \ar@{}[r]|-{\cdots} & Z \ar@{=}[r] & Z}}
$}
\]
with $\del_{n+1}=(-1)^{n}\sum^{n}_{i=0}(-1)^{i}\pr_{i}$. An \defn{$n$-torsor of $Z$ by $A$} is an augmented simplicial object $\TT$ equipped with a simplicial morphism $\tt\colon {\TT\to \KK(Z,A,n)}$ such that
\begin{enumerate}
\item[(T1)] $\tt$ is a fibration which is exact from degree $n$ on;
\item[(T2)] $\TT\cong \Cosk_{n-1}\TT$, the $(n-1)$-coskeleton of $\TT$;
\item[(T3)] $\TT$ is a resolution.
\end{enumerate}
Axiom (T2) means that $\TT$ does not contain information above level $n-1$, which together with (T3) amounts to the $(n-1)$-truncation $T$ of $\TT$, considered as an~$n$-cube, being an $n$-cubic extension. The fibration property in (T1) is (almost) automatic, while the exactness tells us that, for all $i$,
\begin{equation}\label{Torsor-Iso}
\cycle(\TT,n)\cong A\times \horn^{i}(\TT,n).
\end{equation}
Here $A=\bigcap_{i\in n}\K{\del_{i}}$ is the direction of $T$, the object $\cycle(\TT,n)$ consists of all $n$"~cycles in $\TT$ and $\horn^{i}(\TT,n)$ is the object of $(n,i)$-horns in $\TT$. In degree two, for instance, we obtain the following picture:
\begin{equation}\label{Simplicial-Dimension-Two}
\begin{matrix}\cycle(\TT,2) & \cong & A & \times & \horn^{1}(\TT,2) \\
\vcenter{\xymatrix@1@!0@R=2.4495em@C=1.4142em{& {\cdot} \ar[rd]^-{\beta}\\
{\cdot} \ar[ru]^-{\alpha} \ar[rr]_-{\gamma} && {\cdot}}} &&
\vcenter{\xymatrix@1@!0@R=2.4495em@C=1.4142em{{0} \ar[rr]_-{a} && {0}}} &&
\vcenter{\xymatrix@1@!0@R=2.4495em@C=1.4142em{& {\cdot} \ar[rd]^-{\beta}\\
{\cdot} \ar[ru]^-{\alpha} && {\cdot}}}
\end{matrix}
\end{equation}
Given $a$, $\alpha$ and $\beta$, there is a unique arrow $\gamma=\mu^{1}(a,\beta,\alpha)$ such that the projection $\pr_{A}(\beta, \gamma, \alpha)$ on $A$ gives back $a$. In some sense $a=0$ if and only if the triangle on the left ``commutes'', and taking $\gamma=\mu^{1}(0,\beta,\alpha)=m^{1}(\beta,\alpha)$ as a composite of $\beta$ and~$\alpha$ really does define a groupoid structure $m^{1}$ on $T$.

Let $\S^{+}(\X)$ denote the category of augmented simplicial objects in $\X$. The full subcategory of the slice $\S^{+}(\X)/\KK(Z,A,n)$ determined by the $n$-torsors of~$Z$ by~$A$ is written $\Tors^{n}(Z,A)$. Taking connected components we obtain the set
\[
\Tors^{n}[Z,A]=\pi_{0}\Tors^{n}(Z,A)
\]
of equivalence classes of $n$-torsors of $Z$ by $A$. It is, in fact, an abelian group~\cite{Duskin-Torsors}.

Duskin explains in~\cite{Duskin, Duskin-Torsors} that the group $\Tors^{n}[Z,A]$ may be considered as a cohomology group $\H^{n+1}(Z,A)$ of $Z$ with coefficients in the trivial module $A$, and that under certain conditions this cohomology coincides with other known cohomology theories. For instance, when $\X$ is monadic over $\Set$, we obtain Barr--Beck cohomology~\cite[Theorem 5.2]{Duskin-Torsors}. 

If $\GG$ is the comonad induced by the forgetful/free adjunction of $\X$ to $\Set$, if~$Z$ is an object of $\X$ and $A$ an abelian object, then for any natural number $n$,
\[
\H^{n+1} (Z,A)_{\GG}=\H^{n}\Hom (\ab \GG Z,A)
\]
is the $(n+1)$-th cohomology group of $Z$ with coefficients in $A$, relative to the comonad $\GG$~\cite{Barr-Beck}. This defines a functor $\H^{n+1} (-,A)\colon \X \to \Ab $, for any~${n\geq 0}$. As mentioned in the previous paragraph, Duskin obtains an isomorphism
\[
\H^{n+1} (Z,A)_{\GG}\cong \H^{n+1} (Z,A),
\]
where the latter cohomology group is $\Tors^{n}[Z,A]$ by definition.

\subsection*{The geometry of higher central extensions}
In Section~\ref{Section-Geometry} we analyse higher central extensions from a geometrical point of view so that we can compare them with higher torsors. We work towards Proposition~\ref{Proposition-Central-then-Torsor} which says that an augmented simplicial object carries a (unique) structure of $n$-torsor as soon as its underlying $n$-fold arrow is an $n$-cubic central extension. This result is based on Theorem~\ref{Theorem-Higher-Centrality} which gives a new characterisation of higher central extensions: an $n$-cubic extension $F$ is central if and only if its direction $A$ is abelian and there is a canonical isomorphism
\begin{equation}\label{Iso-Box}
\bigboxvoid_{i\in n}\R{f_{i}}\cong A\times \bigboxdot_{i\in n}^{I}\R{f_{i}}
\end{equation}
for any (hence, all) $I\subseteq n$. (Compare with the isomorphism~\eqref{Torsor-Iso}.) The precise definition of the objects $\bigboxvoid_{i\in n}\R{f_{i}}$ and $\bigboxdot_{i\in n}^{I}\R{f_{i}}$ will be presented in Section~\ref{Section-Geometry}, but we can already explain the meaning of this characterisation in some low-dimensional cases and give the main idea.

When $n=1$ this characterisation of centrality becomes the well-known result (see~\cite{Bourn-Gran, Bourn-Gran-Maltsev}) that an extension $(k,f)\colon{A\to X\to Z}$ is central if and only if the kernel (direction) $A$ of $f$ is abelian and the kernel pair of $f$ may be decomposed into a product as $\R{f}\cong A\times X$.

When $F$ is the two-cubic extension~\eqref{Double-Extension-Intro} the isomorphism becomes
\[
\R{d}\square \R{c}\cong A\times (\R{d}\times_{X} \R{c}),
\]
where the direction $A$ is given by $A=\K{d}\cap \K{c}$. As explained in Subsection~\ref{Degree two}, this isomorphism can be obtained as a consequence of the analysis of two-cubic extensions carried out in~\cite{RVdL}. Recall~\cite{Janelidze-Pedicchio, Borceux-Bourn} that $\R{d}\square \R{c}$ contains \emph{diamonds} (as on the left)
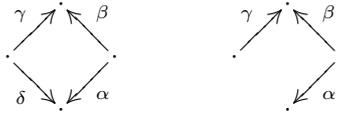
\[
\xymatrix@1@!0@=2em{& {\cdot}\\
{\cdot} \ar[ru]^-{\gamma} \ar[rd]_-{\delta} && {\cdot} \ar[lu]_-{\beta} \ar[ld]^-{\alpha}\\
&{\cdot}}
\qquad\qquad
\xymatrix@1@!0@=2em{& {\cdot}\\
{\cdot} \ar[ru]^-{\gamma} && {\cdot} \ar[lu]_-{\beta} \ar[ld]^-{\alpha}\\
&{\cdot}}
\]
in $X$, so that the object $\R{d}\times_{X} \R{c}$, which is an instance of the pullback~\eqref{Pullback-Smith} on page~\pageref{Pullback-Smith}, contains \emph{diamonds with one face missing} (as on the right above) and
\[
\pi\colon {\R{d}\square \R{c}\to \R{d}\times_{X} \R{c}}
\]
is the projection which forgets $\delta$. The analogy with~\eqref{Simplicial-Dimension-Two} is clear and not accidental: the missing $\delta$ corresponds to a unique element $a$ of~$A$; on the other hand, given any diamond (including $\delta$), the corresponding element $a$ of~$A$ measures how far the diamond is from being ``commutative'' (in which case one may think of~$\delta$ as a composite $\alpha\beta^{-1}\gamma$). Note that instead of forgetting~$\delta$, we could have chosen to forget $\alpha$, $\beta$ or $\gamma$; each of those choices determines a different pullback~${\R{d}\times_{X} \R{c}}$ which, for the sake of clarity, could be written as $\R{d}\squaredot^{I} \R{c}$ where the index~${I\subseteq 2}$ determines the chosen projection (indeed there are four options).

In general, given an $n$-cubic extension $F$, the object $\bigboxvoid_{i\in n}\R{f_{i}}$ contains what we call ``$n$-dimensional diamonds'' in $F_{n}$ (see Figure~\ref{Figure-Diamond} on page~\pageref{Figure-Diamond} for an illustration in dimension $3$) and $\bigboxdot_{i\in n}^{I}\R{f_{i}}$ contains ``$n$-dimensional diamonds'' with one face (determined by the index $I\subseteq n$) missing. The cubic extension $F$ is central when its direction $A$ is abelian and the canonical projection
\[
\pi^{I}\colon\bigboxvoid_{i\in n}\R{f_{i}}\to \bigboxdot_{i\in n}^{I}\R{f_{i}}
\]
induces the isomorphism~\eqref{Iso-Box}; this means that a missing face in any $n$-fold diamond in $F_{n}$ is completely determined by an element in $A$. We also obtain an explicit formula for the splitting $\pr_{A}\colon {\bigboxvoid_{i\in n}\R{f_{i}}\to A}$ of the kernel of $\pi^{I}$, the projection on $A$ which gives us a ``measure of commutativity'' for $n$-fold diamonds: Proposition~\ref{Centrality-Sum} states that the lifting of
\[
\sum_{J\subseteq n}(-1)^{|J|}\eta_{F_{n}}\comp \pr_{J}
\]
over $A$, where $\pr_{J}\colon{\bigboxvoid_{i\in n}\R{f_{i}}\to F_{n}}$ sends a diamond to its $J$-face, is $\pr_{A}$.

Using this geometrical interpretation of centrality we can compare torsors and central extensions. Any $n$-cycle may be ``completed'' into an $n$-fold diamond by adding well-chosen degeneracies, and thus restricting the isomorphism~\eqref{Iso-Box} to an isomorphism~\eqref{Torsor-Iso} we may prove that any augmented simplicial object of which the underlying $n$-fold arrow is a central extension is in fact an $n$-torsor. The converse, however, needs more, since in general it is not clear how an isomorphism on the simplicial level may be extended to an isomorphism on the level of higher-dimensional diamonds. For this implication we pass via an interpretation of centrality in terms of commutators.

\subsection*{The commutator condition}
In order to complete the equivalence between torsors and higher central extensions, we shall assume that centrality may be characterised in terms of binary Huq commutators. We call this assumption the \defn{commutator condition (CC)} on higher central extensions~\cite{RVdL3}: it holds when, for all $n\geq 1$, an $n$-cubic extension $F$ is central if and only~if
\[
\displaystyle\Bigl[\bigcap_{i\in I}\K{f_{i}},\bigcap_{i\in n\setminus I}\K{f_{i}}\Bigr]=0
\]
for all $I\subseteq n$. Following~\cite{RVdL3}, an $n$-cubic extension which satisfies this commutator condition is called \defn{H-central}. If we name the concept of centrality coming from Galois theory \defn{categorical centrality}, then (CC) says that \emph{H-central and categorically central extensions are the same}.

The condition (CC) amounts to asking that the Hopf formula for higher homology~\eqref{General-Hopf} becomes a quotient of binary Huq commutators: its denominator~$L_{n}[F]$ is then equal to the join $\bigcup_{I\subseteq n}\bigl[\bigcap_{i\in I}\K{f_{i}},\bigcap_{i\in n\setminus I}\K{f_{i}}\bigr]$, so that
\[
\H_{n+1}(Z,\Ab(\X))\cong\dfrac{\bigcap_{i\in n}\K{f_{i}}\cap [F_{n},F_{n}]}{\bigcup_{I\subseteq n}\bigl[\bigcap_{i\in I}\K{f_{i}},\bigcap_{i\in n\setminus I}\K{f_{i}}\bigr]}
\]
for any $n$-cubic presentation $F$ of any object $Z$ and for any $n\geq 1$. We shall, however, focus on the cohomological meaning of this condition rather than on Hopf formulae.

It is certain that many categories satisfy (CC), although thus far no explicit characterisation is known; in the article~\cite{EGVdL} the categories of groups, Lie algebras and non-unitary rings are given as examples, and it is not difficult to add new examples to the list by using the technique explained there. A~wide range of (generally non-trivial) examples are those semi-abelian categories with a protoadditive abelianisation functor~\cite{Everaert-Gran-nGroupoids, Everaert-Gran-TT}, of which two extreme special cases are all semi-abelian arithmetical categories, such as the categories of von Neumann regular rings, Boolean rings and Heyting semilattices (where the cohomology theory becomes trivial) on the one hand, and all abelian categories (where, via a version of the Dold--Kan correspondence~\cite{Dold-Puppe}, the theory gives us the Yoneda $\ext$ groups) on the other. More recently it was shown in~\cite{RVdL3} that all semi-abelian categories with the \emph{Smith is Huq}~\cite{MFVdL} property satisfy (CC), while the categories of loops and of commutative loops do not. So, \emph{action representative} semi-abelian categories~\cite{BJK2,Borceux-Bourn-SEC}, \emph{action accessible} categories~\cite{BJ07} (which makes all \emph{categories of interest}~\cite{Orzech} examples~\cite{Montoli}), \emph{strongly semi-abelian}~\cite{Bourn2004} and \emph{Moore} categories~\cite{Gerstenhaber, Rodelo:Moore} are all examples of categories with the \emph{Smith is Huq} property, thus satisfy (CC). For instance, so do the categories of associative and non-associative algebras and of (pre)crossed modules, and all \emph{varieties of groups} in the sense of~\cite{Neumann}. In any case, every semi-abelian category satisfies (CC) for~${n=1}$ (see~\cite{Gran-Alg-Cent, Gran-VdL}).

Proposition~\ref{Proposition-Torsor-then-Central} now tells us that in a semi-abelian category with (CC), the $n$"~cubic extension underlying an $n$-torsor is always central. Hence for a truncated augmented simplicial object, the two concepts are equivalent (Theorem~\ref{Theorem-Torsor-Equivalence}): indeed, given a simplicial object $\TT$, when it exists, a morphism $\tt$ making~$(\TT,\tt)$ a torsor is necessarily unique; furthermore, any morphism of $n$-fold central extensions of $Z$ by $A$ which restricts to the truncation of a simplicial morphism, extends uniquely to a morphism of $n$-torsors of $Z$ by $A$---see Section~\ref{Section-Torsors-and-Centrality}, in particular Proposition~\ref{Proposition-Full}.

\subsection*{The main theorem}
Proposition~\ref{Proposition-Simplification-of-Central-Extension} tells us that, as soon as enough projectives exist,  any $n$-fold central extension of an object $Z$ by an abelian object $A$ is connected to an $n$-fold central extension of~$Z$ by~$A$ of which the underlying $n$-cube is a truncation of an augmented simplicial object. Thus under~(CC), any $n$-cubic central extension is connected to the simplicial-object part of an $n$-torsor. Since, furthermore, this process is compatible with directions, and we acquire an isomorphism
\[
\Tors^{n}[Z,A]=\pi_{0}\Tors^{n}(Z,A)\cong \pi_{0}(\D^{-1}_{(n,Z)}A)=\Centr^{n}(Z,A),
\]
natural in $A$. As a consequence we obtain the main result of this article, Theorem~\ref{Main-Theorem}: if $Z$ is an object and~$A$ an abelian object in a semi-abelian category with enough projectives satisfying the commutator condition (CC), then for every $n\geq 1$ we have that
\[
\H^{n+1}(Z,A)\cong\Centr^{n}(Z,A)
\]
as abelian groups. This establishes the result conjectured in~\cite{RVdL}, though for a different definition of cohomology, and has several other interesting implications. For instance, from~\cite{Duskin} it follows that there is a long exact sequence for~$\Centr^{n}(Z,-)$.

\subsection*{``Duality'' between ``internal'' homology and ``external'' cohomology}
We call a (co)homology theory \defn{internal} when the (co)homology object is an actual (abelian) object in the ground semi-abelian category, and \defn{external} if it is an (abelian) group and hence, in general, is an object outside the ground category. For instance, the approach to homology in terms of higher Hopf formulae discussed above is internal, while the approach to cohomology via higher central extensions is external. An example of internal \emph{co}homology is developed in Gray's Ph.D.\ thesis~\cite{GrayPhD}.

Combined with the main result of the article~\cite{GVdL2}, our present  interpretation of external cohomology gives an answer to the following somewhat naive question:
\begin{quote}
\emph{In which sense are internal homology\\ and external cohomology ``dual'' to each other?}
\end{quote}
The word ``dual'' here should not be read in its formal categorical sense, but similarly to the way we read ``dual of a vector space''. It is true that there is a kind of ``duality'', or at least a strong symmetry, in the definitions of homology and cohomology when one uses, for instance, the comonadic Barr--Beck approach. Nevertheless, so far there was no meaningful connection at all between the \emph{interpretations} of internal homology (using Hopf formulae, say) and external cohomology (many different approaches here), at least not for non-abelian algebraic objects. Following~\cite{Tim-Missing-Link}, we claim that the hidden connection is the concept of \emph{direction for higher central extensions} and the analysis of both homology (internal, relative to abelianisation) and cohomology (external, with trivial coefficients) in these terms.

The theory of \emph{satellites}~\cite{GVdL2, Janelidze-Satellites} makes it possible to replace Hopf formulae for homology with (possibly large) limits, so that homology objects may also be computed in contexts where not enough projective objects are available. The results in~\cite{GVdL2} are again based on higher central extensions in semi-abelian categories, and the article's Corollary~4.10 tells us that, for any Birkhoff subcategory~$\B$ of a semi-abelian category $\X$, for any object $Z$ of $\X$ and any integer~${n\geq 1}$, the homology object $\H_{n+1}(Z,\B)$ is the limit of the diagram $\D_{(n,Z)}\colon{\CExt^{n}_{Z}(\X)\to \B}$.

That is to say, in the case of abelianisation, all the internal homological and external cohomological information on an object $Z$ at a given level $n$ is contained in one and the same functor 
\[
\D_{(n,Z)}\colon\CExt^{n}_{Z}(\X)\to \Ab(\X)\colon E\mapsto \D_{(n,Z)}E=\bigcap_{i\in n}\K{f_{i}}=\Kn{F}
\]
(where $F=E|_{2^{n}}$) in two ``opposite'' ways,
\[
\H_{n+1}(Z,\Ab(\X))=\lim \D_{(n,Z)}\qquad\text{and}\qquad \H^{n+1}(Z,A)=\pi_{0}(\D^{-1}_{(n,Z)}A);
\]
homology is a limit of $\D_{(n,Z)}$ while cohomology consists of connected components of a fibre of $\D_{(n,Z)}$. So on the one end we have the limit of all possible directions and, on the other, all classes of all central extensions with one given and fixed direction---again, see Figure~\ref{Figure-Direction}. We consider this ``duality'' (Theorem~\ref{Duality-Theorem}) to be a major point of the present article. In the article~\cite{PVdL1} it is analysed from the point of view of the Yoneda lemma, which deals with precisely this kind of contrast or ``duality'' between ``internal'' and ``external''.

\subsection*{Structure of the text}
In Section~\ref{Section-Preliminaries} we recall some basic definitions and results which we need later on: semi-abelian categories, simplicial objects, higher extensions and higher torsors. Section~\ref{Section-Central-Extensions} contains all the theory needed to introduce the groups~$\Centr^{n}(Z,A)$. In Section~\ref{Section-Geometry} we give a geometric interpretation of the concept of higher central extension (Theorem~\ref{Theorem-Higher-Centrality} and Proposition~\ref{Centrality-Sum}), used in the next section where we analyse torsors in terms of this geometry. The most important result here is Proposition~\ref{Proposition-Central-then-Torsor} which says that a truncation of an augmented simplicial object, considered as a higher extension, is a torsor as soon as it is a central extension. The other implication in the equivalence between torsors and central extensions is obtained in Section~\ref{Section-Commutator-Assumption} (Proposition~\ref{Proposition-Torsor-then-Central} and Theorem~\ref{Theorem-Torsor-Equivalence}). However, to make it work, we have to strengthen the context of semi-abelian categories with the additional commutator condition~(CC). The short last Section~\ref{Section-Main-Theorem} explains how to suitably transform an $n$-cubic central extension (which need not be a truncation of anything simplicial) into an $n$-cubic central extension underlying a torsor, so that we may conclude with Theorem~\ref{Main-Theorem}, the isomorphism $\H^{n+1}(Z,A)\cong\Centr^{n}(Z,A)$ for all $n\geq 1$, and Theorem~\ref{Duality-Theorem}, the ``duality'' between internal homology and external cohomology.

\pagebreak
\section{Preliminaries}\label{Section-Preliminaries}
We sketch the context in which we shall be working: homological and semi-abelian categories for all general results, with the approach to external cohomology in Barr-exact categories due to Duskin~\cite{Duskin, Duskin-Torsors} and Glenn~\cite{Glenn}. We also recall the definition of higher extensions and the relation with simplicial resolutions~\cite{EGVdL, EGoeVdL}.

\subsection{Barr-exact, homological and semi-abelian categories}
For the sake of clarity, the results in this article will be presented in the context of semi-abelian categories. Although this is an extremely convenient environment to work in, it is probably not the most general context in which the theory may be developed. Nevertheless, we believe that in this first approach it is better not to cloud our results in technical subtleties concerning the surrounding category, but rather to focus on their intrinsic meaning and their correctness. The only disadvantage this added transparency may possibly have is the potential loss of some more elaborate examples; such examples can always be recovered later on.

We recall the main definitions and properties of Barr-exact~\cite{Barr}, homological~\cite{Borceux-Bourn} and semi-abelian categories~\cite{Janelidze-Marki-Tholen}.

Recall that a \defn{regular epimorphism} is the coequaliser of some pair of morphisms. A~finitely complete category endowed with a pullback-stable (regular epimorphism, monomorphism)-factorisation system is called \defn{regular}. Regular categories provide a natural context for working with relations. We denote the kernel relation (=~kernel pair) of a morphism $f$, the pullback of $f$ along itself, by $(\R{f},\pr_{0},\pr_{1})$ or~$(\R{f},f_{0},f_{1})$, depending on the situation. A regular category is said to be \defn{Barr-exact} when every equivalence relation is the kernel pair of some morphism~\cite{Barr}.

A \defn{pointed} category (that is, with a \defn{zero object}, an initial object that is also terminal) that admits pullbacks is called \defn{Bourn-protomodular}~\cite{Bourn1991} when the Split Short Five Lemma holds. Moreover, if the pointed category is regular, then protomodularity is equivalent to the (Regular) Short Five Lemma: given a commutative diagram
\begin{equation}\label{Short-Five-Lemma}
\vcenter{\xymatrix{0 \ar[r] & \K{f'} \ar[d]_{k} \ar@{{ |>}->}[r]^-{\ker f'} & X' \ar[d]_-{x} \ar@{-{ >>}}[r]^-{f'} & Y' \ar[d]^-{y} \ar[r] & 0\\
0 \ar[r] & \K{f} \ar@{{ |>}->}[r]_-{\ker f} & X \ar@{-{ >>}}[r]_-{f} & Y \ar[r] & 0}}
\end{equation}
with regular epimorphisms $f$, $f'$ and their kernels, if $k$ and $y$ are isomorphisms then also $x$ is an isomorphism. We usually denote the kernel of a morphism~$f$ by $(\K{f}, \Ker f)$. A pointed, regular and protomodular category is called \defn{homological}~\cite{Borceux-Bourn}. This is a context where many of the basic diagram lemmas of homological algebra hold. In particular, here the notion of \emph{(short) exact sequence} has its full meaning: a regular (= normal) epimorphism with its kernel.

In order for commutator theory to work flawlessly, the context should be finitely cocomplete and Mal'tsev. By definition, a \defn{Mal'tsev} category~\cite{Carboni-Lambek-Pedicchio,CPP} is finitely complete and such that every reflexive relation is necessarily an equivalence relation. It is well known that any finitely complete protomodular category is necessarily a Mal'tsev category~\cite{Bourn1996}.

Joining all these conditions brings us to the notion of a \defn{semi-abelian} category which can be defined as a pointed, Barr-exact and protomodular category that admits binary coproducts. This definition unifies many older approaches towards a suitable categorical context for the study of homological properties of non-abelian categories such as the categories of groups, Lie algebras, etc. In the founding article~\cite{Janelidze-Marki-Tholen} which introduces the concept, it is explained how this solves the problem of finding the right axioms to be added to Barr-exactness in order that the resulting context is equivalent with the contexts obtained in terms of ``old-style axioms'' such as, for instance, the one introduced in~\cite{Huq}.

Examples of semi-abelian categories include all varieties of $\Omega$-groups~\cite{Higgins}, such as groups and non-unitary rings, precrossed and crossed modules, and categories of non-unitary algebras such as associative algebras and Leibniz and Lie $n$-algebras; then there are non-unitary $C^{*}$-algebras and loops; also any abelian category is an example, as is the dual of the category of pointed objects in any elementary topos.

\begin{lemma}\cite{Bourn-Janelidze:Semidirect, Bourn2001}\label{Lemma-Iso-Pullback}
In a semi-abelian category, given a commutative diagram with short exact rows such as~\eqref{Short-Five-Lemma} above, $k$ is an isomorphism if and only if the right-hand square is a pullback.\noproof
\end{lemma}

\begin{lemma}\cite[Theorem 4.9]{Bourn-Gran-Normal-Sections}\label{Lemma-Split-Kernel-Product}
In a semi-abelian category, given a short exact sequence
\[
\vcenter{\xymatrix{0 \ar[r] & \K{f} \ar@<-.5ex>@{{ |>}->}[r]_-{\ker f} & X \ar@{-{ >>}}[r]^-{f} \ar@<-.5ex>@{-{ >>}}[l]_-{p} & Y \ar[r] & 0}}
\]
in which the kernel of $f$ is split by a morphism $p$, the object $X$ is a product of which the projections are $p$ and $f$.
\end{lemma}
\begin{proof}
Applying Lemma~\ref{Lemma-Iso-Pullback} to the diagram
\[
\vcenter{\xymatrix{0 \ar[r] & \K{f} \ar@{=}[d] \ar@{{ |>}->}[r]^-{\ker f} & X \ar[d]_-{p} \ar@{-{ >>}}[r]^-{f} & Y \ar[d] \ar[r] & 0\\
0 \ar[r] & \K{f} \ar@{=}[r] & \K{f} \ar@{-{ >>}}[r] & 0}}
\]
shows that its right hand square is a pullback.
\end{proof}

\subsection{The Huq commutator and the Smith/Pedicchio commutator}\label{Commutators}
We work in a semi-abelian category $\X$. A coterminal pair
\begin{equation*}\label{Cospan}
\vcenter{\xymatrix{K \ar[r]^-{k} & X & L \ar[l]_-{l}}}
\end{equation*}
of morphisms in $\X$ \defn{(Huq-)commutes}~\cite{BG,Huq} when there is a (necessarily unique) morphism $\varphi_{k,l}$ such that the diagram
\[
\xymatrix@!0@=3em{ & K \ar[ld]_{\lind 1_{K},0\rind} \ar[rd]^-{k} \\
K\times L \ar@{.>}[rr]|-{\varphi_{k,l}} && X\\
& L \ar[lu]^{\lind 0,1_{L}\rind} \ar[ru]_-{l}}
\]
is commutative. We shall only consider the case where $k$ and $l$ are normal monomorphisms (kernels). The \defn{Huq commutator $[k,l]\colon {[K,L]\to X}$ of~$k$ and~$l$} is the smallest normal subobject of $X$ which should be divided out to make $k$ and~$l$ commute, so that they do commute if and only if $[K,L]=0$. It may be obtained through the colimit $Q$ of the outer square above, as the kernel of the (normal epi)morphism ${X\to Q}$. The commutator $[K,L]$ becomes the ordinary commutator of normal subgroups~$K$ and~$L$ in the case of groups, the ideal generated by $KL+LK$ in the case of non-unitary rings, the Lie bracket in the case of Lie algebras, and so on.

Consider a pair of equivalence relations
\begin{equation*}\label{Category-RG}
\xymatrix@!0@=6em{R \ar@<1.5ex>[r]^-{r_{0}} \ar@<-1.5ex>[r]_-{r_{1}} & X \ar[l]|-{\lind1_{X},1_{X}\rind} \ar[r]|-{\lind1_{X},1_{X}\rind} & S \ar@<1.5ex>[l]^-{s_{0}} \ar@<-1.5ex>[l]_-{s_{1}}}
\end{equation*}
on a common object $X$ and consider the induced pullback of $r_{1}$ and $s_{0}$:
\begin{equation}\label{Pullback-Smith}
\vcenter{\xymatrix@!0@=4em{R\times_{X}S \pullback \ar[r]^-{\pi_{S}} \ar[d]_-{\pi_{R}} & S \ar[d]^-{s_{0}} \\
R \ar[r]_-{r_{1}} & X}}
\end{equation}
The pair $(R,S)$ \defn{(Smith/Pedicchio-)commutes}~\cite{Smith, Pedicchio, BG} when there is a (necessarily unique) morphism $\theta$ such that the diagram
\[
\xymatrix@!0@=3em{ & R \ar[ld]_{\lind 1_{R},\lind1_{X},1_{X}\rind r_{1}\rind} \ar[rd]^-{r_{0}} \\
R\times_{X}S \ar@{.>}[rr]|-{\theta} && X\\
& S \ar[lu]^{\lind\lind1_{X},1_{X}\rind s_{0},1_{S}\rind} \ar[ru]_-{s_{1}}}
\]
is commutative. As for the Huq commutator, the \defn{Smith/Pedicchio commutator} is the smallest equivalence relation $[R,S]$ on $X$ which, divided out of~$X$, makes~$R$ and $S$ commute. It can be obtained through a colimit, similarly to the situation above. Thus $R$ and $S$ commute if and only if $[R,S]=\Delta_{X}$, where $\Delta_{X}$ is the smallest equivalence relation on $X$. We say that $R$ is a \defn{central} equivalence relation when it commutes with $\nabla_X$, the largest equivalence relation on $X$, so that~$[R,\nabla_X]=\Delta_X$.

\subsection{Abelian objects, Beck modules}\label{Subsection-Abelian}
In a semi-abelian category $\X$, an object~$A$ is said to be \defn{abelian} when~$[A,A]=0$. The abelian objects of~$\X$ determine a full and reflective subcategory which is denoted~$\Ab(\X)$. Given any object $X$ of~$\X$, we shall write~${\lb X\rb=[X,X]}$, so that we obtain a short exact sequence
\[
\xymatrix@1{0 \ar[r] & \lb X\rb \ar@{{ |>}->}[r] & X \ar@{-{ >>}}[r]^-{\eta_{X}} & \ab X=X/[X,X] \ar[r] & 0}
\]
where $\eta_{X}$ is the $X$-component of the unit $\eta$ of the adjunction
\[
\xymatrix{{\X} \ar@<1ex>[r]^-{\ab} \ar@{}[r]|-{\perp} & \Ab(\X). \ar@<1ex>[l]^-{\supset}}
\]
An object in a semi-abelian category is abelian precisely when it admits a (necessarily unique) internal abelian group structure. In fact, $\Ab(\X)$ may be viewed as the abelian category of internal abelian groups in $\X$. For instance, an abelian object in the category of groups is an abelian group, and an abelian associative algebra over a field is a vector space (equipped with a trivial multiplication).

Given an object $Z$ of $\X$, a \defn{$Z$-module} or \defn{Beck module over $Z$} is an abelian group in the slice category~$\X/Z$. Thus a $Z$-module $(f,m,s)$ consists of a morphism~$f\colon {X\to Z}$ in $\X$, equipped with a multiplication $m$ and a unit~$s$ as in the diagrams
\[
\vcenter{\xymatrix@!0@R=3em@C=2em{\R{f} \ar[rr]^-{m} \ar[rd] && X \ar[ld]^-{f} \\ & Z}}
\qquad
\qquad
\vcenter{\xymatrix@!0@R=3em@C=2em{Z \ar[rr]^-{s} \ar@{=}[rd] && X \ar[ld]^-{f} \\ & Z}}
\]
satisfying the usual axioms. In particular we obtain a split short exact sequence
\[
\xymatrix{0 \ar[r] & A \ar@{{ |>}->}[r]^-{\Ker f} & X \ar@<-.5ex>@{ >>}[r]_-{f} & Z \ar[r] \ar@<-.5ex>[l]_-{s} & 0}
\]
where $A$ is an abelian object in $\X$ and $f$ is split by $s$. Furthermore, the morphism~$f$ satisfies $[\R{f },\R{f}]=\Delta_{X}$. Conversely, given the splitting $s$ of~$f$, this latter condition makes it possible to recover the multiplication $m$. Hence, for split epimorphisms in a semi-abelian category, ``being a Beck module'' is a property; the entire module structure is contained in the splitting. Using the equivalence between split epimorphisms and internal actions~\cite{Bourn-Janelidze:Semidirect}, we can replace $X$ with a semi-direct product~${(A,\xi)\rtimes Z}$. By the above, modules are ``abelian actions''. For simplicity, we denote a $Z$-module by its induced~$Z$-algebra $(A,\xi)$.

For us, the most important case arises when the $Z$-module structure on~$A$ is the trivial one, denoted $(A,\tau)$: then $A$ is just an abelian object, the semidirect product~${(A,\tau)\rtimes Z}$ is $A\times Z$ and $f$ is the product projection $\pr_{Z}\colon {A\times Z\to Z}$.

\subsection{Connected components}\label{Connected-components}
In a category $\X$, two objects are \defn{connected} when there exists a (finite) zigzag of morphisms between them. This defines an equivalence relation between the objects of $\X$, of which the equivalence classes form the set $\pi_{0}(\X)$ of \defn{connected components} of $\X$.

In general $\pi_{0}(\X)$ may not be a small set, and even in the two situations where we shall use this construction (Subsection~\ref{Torsors} and Definition~\ref{Definition-Centr}) it will a priori not be clear whether or not the result is not a proper class. In fact, even when it \emph{is} a proper class, this has no significant effect at all on the theory we develop, so we decided not to go into this question any further. Additionally, in the monadic case the smallness of the cohomology groups follows from the interpretation in terms of Barr--Beck cohomology.

\subsection{A lemma on double split epimorphisms}
By a result in~\cite{Bourn1996}, a finitely complete category is \defn{naturally Mal'tsev}~\cite{Johnstone:Maltsev} when, given a split epimorphism of split epimorphisms as in
\begin{equation}\label{Double-Split-Epi}
\vcenter{\xymatrix@=3em{A_{1} \ar@<-.5ex>[r]_-{f_{1}} \ar@<-.5ex>[d]_-{a} & B_1 \ar@<-.5ex>[d]_-{b} \ar@<-.5ex>[l]_-{\overline{f_{1}}}\\
A_{0} \ar@<-.5ex>[u]_-{\overline{a}} \ar@<-.5ex>[r]_-{f_{0}} & B_{0} \ar@<-.5ex>[u]_-{\overline{b}} \ar@<-.5ex>[l]_-{\overline{f_{0}}}}}
\end{equation}
(all squares commute), if the square is a (down-right) pullback of split epimorphisms, then it is an (up-left) pushout of split monomorphisms. As a consequence we obtain the following lemma (see also~\cite{MFVdL2} and~\cite{EGoeVdL}).

\begin{lemma}\label{Lemma-Naturally-Maltsev}
In a naturally Mal'tsev category, given a double split epimorphism such as~\eqref{Double-Split-Epi}, the universally induced comparison morphism
\[
\lind a,f_{1}\rind\colon A_{1}\to A_{0}\times_{B_{0}}B_{1}
\]
to the pullback of $f_{0}$ and $b$ is a split epimorphism, with a unique splitting
\[
\nu\colon {A_{0}\times_{B_{0}}B_{1}\to A_{1}}
\]
such that $\overline{a}=\nu \lind1_{A_{0}},\overline{b}f_{0}\rind$ and $\overline{f_{1}}=\nu\lind \overline{f_{0}}b,1_{B_{1}}\rind$.
\noproof
\end{lemma}

It is well known that every additive category is naturally Mal'tsev. In particular, for any semi-abelian category $\X$, the above lemma is valid in the abelian category~$\Ab(\X)$.

\subsection{The von Neumann construction of the finite ordinals}
We shall write ${0=\emptyset}$ and $n=\{0,\dots, n-1\}$ for $n\geq 1$. We also write $2^{n}$ for the power-set of $n$, considered as a category of which an object is a subset of $n$, and an arrow ${I\to J}$ is an inclusion $I\subseteq J$.

\subsection{Higher arrows}\label{HDA}
Let $\X$ be any category. The category $\Arrn(\X)$ consists of \defn{$n$-fold arrows} in~$\X$: $\Arr^{0}(\X)=\X$, while $\Arr^{1}(\X)=\Arr(\X)$ is the category of arrows in~$\X$ and $\Arrnn(\X)=\Arr(\Arrn(\X))$.

The category of arrows in $\X$ is the functor category ${\Fun(2^{\op},\X)=\X^{2^{\op}}}$. Similarly, any $n$-fold arrow $F$ in $\X$ may be viewed as an ``$n$-fold cube with chosen directions'', a functor $F\colon{(2^{n})^{\op}\to \X}$, and any morphism of $n$-fold arrows as a natural transformation between such functors. If $F$ is an $n$-fold arrow and~$I$ and~$J$ are subsets of $n$ such that $I\subseteq J$, we shall write $F_{I}=F(I)$ for the value of~$F$ in~$I$ and~$f^{J}_{I}\colon{F_{J}\to F_{I}}$ for the value of~$F$ in the morphism induced by the inclusion $I\subseteq J$. When $I=J\setminus \{i\}$ we write $f_{i}\colon{F_{J}\to F_{I}}$ for~$f^{J}_{I}$.

An $n$-fold arrow given as a functor $F\colon{(2^{n})^{\op}\to \X}$ can be seen as an arrow between $(n-1)$-fold arrows $F\colon \dom F \to \cod F$, where its domain $\dom F$ is determined by the restriction of $F$ to all $I\subseteq n$ which contain $n-1$, and its codomain $\cod F$ by the restriction of $F$ to all $I\subseteq n$ which do not contain~$n-1$. Thus, if~${n\geq 2}$, we may see $F$ as a commutative square
\begin{equation}\label{Double-Extension}
\vcenter{\xymatrix{X \ar[r]^-{c} \ar[d]_-{d} & C \ar[d]^-{g}\\
D \ar[r]_-{f} & Z}}
\end{equation}
in $\Arr^{n-2}(\X)$ or, equivalently, a morphism $(c,f)\colon d\to g$ of~$\Arr^{n-1}(\X)$.

Given an $n$-fold arrow $F\colon{(2^{n})^{\op}\to \X}$, we can always consider the restriction of this diagram to the subcategory $2^{n}\setminus \{n\}$; it is the $n$-fold cube~$F$ without its ``initial object'' $F_{n}$. When it exists, write $(\L F,(\pr_{i})_{i\in n})$ for the limit of this diagram, and
\[
l_{F}=\lind f_{0},\dots,f_{n-1}\rind\colon{F_{n}\to \L F}
\]
for the universally induced comparison morphism.

\begin{lemma}\label{Lemma-L-pullbacks}
Given $n\geq 2$, if $F$ is an $n$-fold arrow considered as a square~\eqref{Double-Extension} of $(n-1)$-fold arrows, then $\L F$ may be obtained as $\L G$, where the $(n-1)$-fold arrow $G$ is $\lind d,c\rind\colon X\to D\times_{Z}C$, induced by the pullback of $f$ and $g$. Furthermore, $l_{F}=l_{G}$.
\end{lemma}
\begin{proof}
This is part of the proof of Proposition~1.16 in~\cite{EGoeVdL}.
\end{proof}

\subsection{Higher cubic extensions}\label{Extensions}
Let $\X$ be a semi-abelian category. A~\defn{zero-cubic extension} in~$\X$ is an object of $\X$ and a \defn{one-cubic extension} is a regular epimorphism in $\X$. For $n\geq 2$, an \defn{$n$-cubic extension} is a commutative square~\eqref{Double-Extension} in~$\Arr^{n-2}(\X)$ such that the morphisms $c$, $d$, $f$, $g$ and the universally induced comparison morphism $\lind d,c\rind\colon{X\to D\times_Z C}$ to the pullback of $f$ with $g$ are $(n-1)$-cubic extensions. The $n$-cubic extensions determine a full subcategory~$\Extn(\X)$ of~$\Arrn(\X)$, and $\Ext(\X)=\Ext^{1}(\X)$.

\begin{proposition}\cite{EGoeVdL}\label{Limit-Characterisation-Extensions}
Given any $n$-fold arrow $F$ in a regular category, the following are equivalent:
\begin{enumerate}
\item $F$ is an $n$-cubic extension;
\item for all $\emptyset\neq I\subseteq n$, the morphism ${F_{I}\to\lim_{J\subsetneq I} F_{J}}$ is a regular epimorphism.
\end{enumerate}
In particular, the induced comparison $l_{F}=\lind f_{0},\dots,f_{n-1}\rind\colon{F_{n}\to \L F}$ is regular epimorphic. \noproof
\end{proposition}

In a Mal'tsev category, a double split epimorphism such as~\eqref{Double-Split-Epi} above is always a two-cubic extension. That is to say, the induced comparison morphism $\lind a,f_{1}\rind$ may not be a split epimorphism as in Lemma~\ref{Lemma-Naturally-Maltsev}, but it will certainly be a regular epimorphism. More generally, any split epimorphism between one-cubic extensions is a two-cubic extension, as follows from~\cite[Theorem~5.7]{Carboni-Kelly-Pedicchio}.

\subsection{Extensions as diagrams of short exact sequences}\label{3^n diagrams}
In what follows we view higher extensions slightly differently: as diagrams of short exact sequences, such as the one displayed in Figure~\ref{3x3 diag} on page~\pageref{3x3 diag} and in Figure~\ref{Figure-Direction} on page~\pageref{Figure-Direction}.

Consider the ordinal $3$ as a category $0\to1\to 2$ and, for $n\geq 1$, its $n$-th power $3^{n}=3\times \cdots \times 3$. The category $ 3^n$ has initial object $ i_n= (0,\dots,0)$ and terminal object $ t_n= (2,\dots ,2)$. Moreover, it has an embedding
\[
\alpha_{e,i}\colon 3\to 3^n\colon k\mapsto (e_{1},\dots,e_{i-1},k,e_{i+1},\dots, e_{n})
\]
parallel to the $i$-th coordinate axis, for each object $e$ of $3^{n}$.

Now, given objects $Z$ and $A$ in $\X$, an \defn{$n$-fold extension} (\defn{under $A$ and over~$Z$}, or \defn{of $Z$ by $A$}) in $\X$ is a functor $E\colon (3^n)^{\op}\to \X$ which sends $i_n$ to $Z$ and $t_n$ to~$A$, and such that each composite
\[
\xymatrix@R=5ex@C=3em{
3^{\op} \ar[r]^-{\alpha_{e,i}^{\op}} &
 (3^n)^{\op} \ar[r]^-{E} &
 \X
}
\]
is a short exact sequence.

For example, a one-fold extension under $A$ and over $Z$ is just a short exact sequence $ A=E_2 \to E_1 \to E_0=Z$. A two-fold (or \defn{double}) extension under $A$ and over $Z$ is a \defn{$3\times 3$-diagram}~\cite{Bourn2001} as in Figure~\ref{3x3 diag}, in which each row and column is short exact:
\[
\xymatrix@R=5ex@C=3em{
A=E_{2,2} \ar@{{ |>}->}[r] \ar@{{ |>}->}[d] &
 E_{1,2} \ar@{-{ >>}}[r] \ar@{{ |>}->}[d] &
 E_{0,2} \ar@{{ |>}->}[d] \\
E_{2,1} \ar@{{ |>}->}[r] \ar@{-{ >>}}[d] &
 E_{1,1} \ar@{-{ >>}}[r] \ar@{-{ >>}}[d] &
 E_{0,1} \ar@{-{ >>}}[d] \\
E_{2,0} \ar@{{ |>}->}[r] &
 E_{1,0} \ar@{-{ >>}}[r] &
 Z=E_{0,0}
}
\]
Figure~\ref{Figure-Direction} displays a $3$-fold extension as a \defn{$3\times 3\times 3$-diagram}. In general, an $n$-fold extension is the same thing as a \defn{$3^{n}$-diagram} in $\X$. We write $\Diag{n}(\X)$ for the category of $n$-fold extensions in $\X$, considered as a full subcategory of $\Fun((3^{n})^{\op},\X)$. The natural embedding $2^{n}\to 3^{n}$ induces a forgetful functor
\[
(-)|_{2^{n}}\colon\Diag{n}(\X)\to \Ext^{n}(\X)\colon E\mapsto F=E|_{2^{n}}.
\]
We shall, however, always think of an $n$-cubic extension as being part of some $3^{n}$-diagram.

\begin{proposition}\label{3x3 vs ext}
For any $n\geq 1$, the forgetful functor $\Diag{n}(\X)\to \Ext^{n}(\X)$ is an equivalence of categories.
\end{proposition}
\begin{proof}
By induction on $n$, we prove that an $n$-fold arrow underlies a $3^{n}$-diagram if and only if it is an $n$-cubic extension. This then shows that the above functor is well defined and (essentially) surjective. 

The case $n=1$ is clear. Suppose now that $n\geq 1$. Then a $3^{n+1}$-diagram, being a short exact sequence of $3^{n}$-diagrams, corresponds to a short exact sequence of $n$-cubic extensions by the induction hypothesis. By Proposition~3.9 in~\cite{EGVdL}, the exactness of this sequence is equivalent to its cokernel piece being an $(n+1)$-cubic extension. 

Moreover, the functor is fully faithful, because any morphism between the $n$-cubic extensions underlying two given $3^{n}$-diagrams extends uniquely to their chosen kernels.
\end{proof} 

Depending on the situation, we may prove categorical properties of $\Diag{n}(\X)$ for $\Ext^{n}(\X)$ and vice versa. 

\subsection{Augmented simplicial objects}
Recall that the \defn{augmented simplicial category} $\Delta^{+}$ has finite ordinals $n\geq 0$ for objects and order preserving functions for morphisms. The category $\S^{+}(\X)$ of \defn{augmented simplicial objects} and augmented simplicial morphisms in a category $\X$ is the functor category $\Fun ((\Delta^{+})^{\op },\X)$. An augmented simplicial object $\XX\colon {(\Delta^{+})^{\op}\to \X}$ is usually considered as a sequence of objects $(\XX_{n})_{n\geq -1}$, with \defn{face operators}~$\del_{i}\colon {\XX_{n}\to \XX_{n-1}}$ and \defn{degeneracy operators} $\sigma_{i}\colon {\XX_{n}\to \XX_{n+1}}$ for ${n\geq i\geq 0}$, subject to the simplicial identities
\[
\begin{aligned}
\del_{i}\comp \del_{j} &=\del_{j-1}\comp \del_{i}\quad \text{if $i<j$}\\
\sigma_{i}\comp \sigma_{j} &= \sigma_{j+1}\comp \sigma_{i}\quad \text{if $i\leq j$}
\end{aligned}
\qquad\text{and}\qquad
\del_{i}\comp\sigma_{j}=\begin{cases}\sigma_{j-1}\comp \del_{i} & \text{if $i<j$} \\
1 & \text{if $i=j$ or $i=j+1$}\\
\sigma_{j}\comp \del_{i-1} & \text{if $i>j+1$.}\end{cases}
\]

\begin{remark}
All simplicial objects we shall be considering in this text will come equipped with some augmentation, even when we occasionally drop the word ``augmented''.
\end{remark}

\subsection{Truncations and coskeleta}\label{Truncations}
For $n\geq 0$, let $\Delta^{+}_{n}$ denote the full subcategory of~$\Delta^{+}$ determined by the ordinals $i\leq n$. The functor category
\[
\SimpArr^{n}(\X)=\Fun ((\Delta^{+}_{n})^{\op },\X)
\]
is the category of \defn{$(n-1)$-truncated simplicial objects} in $\X$. Indeed, as soon as~$\X$ is finitely complete, there is the adjunction
\[
\xymatrix@C=50pt{{\S^{+}(\X)} \ar@<1ex>[r]^-{\trunc_{n-1}} \ar@{}[r]|-{\perp} & \SimpArr^{n}(\X), \ar@<1ex>[l]^-{\cosk_{n-1}}}
\]
where the truncation functor $\trunc_{n-1}$ is given by composition of a simplicial object with the inclusion $\Delta^{+}_{n}\subseteq \Delta^{+}$, and its right adjoint $\cosk_{n-1}$ by right Kan extension along this functor. More explicitly, a coskeleton of an $(n-1)$-truncated simplicial object may be computed using iterated simplicial kernels (see the next subsection).

Clearly, $\trunc_{n-1}\cosk_{n-1}=1_{\SimpArr^{n}(\X)}$. Conversely, a coskeleton of an $(n-1)$-truncated simplicial object contains no information above simplicial degree $n-1$; given any simplicial object $\XX$, we can remove all higher-dimensional information by applying the functor $\Cosk_{n-1}=\cosk_{n-1}\trunc_{n-1}\colon{\S^{+}(\X)}\to {\S^{+}(\X)}$ to it.

Any $(n-1)$-truncated simplicial object may be considered as an $n$-fold arrow, through composition with the functor
\[
\a_{n}\colon 2^{n}\to \Delta^{+}_{n}
\]
which maps a set $I\subseteq n$ to the associated ordinal $|I|$, and an inclusion $I\subseteq J$ to the corresponding order-preserving map ${|I|\to |J|}$. This defines a faithful functor
\[
\arr_{n}=\Fun(-,\a_{n})\colon{\SimpArr^{n}(\X)\to \Arr^{n}(\X)}.
\]
(An ${(n-1)}$-truncated simplicial object has the additional structure of the degeneracies: a morphism of $n$-fold arrows between two given $(n-1)$-truncated simplicial objects need not commute with the degeneracy operators, and furthermore its components at two given sets of the same size need not coincide.) Hence, if $X$ denotes the~$n$-fold arrow underlying the $(n-1)$-truncation of a simplicial object~$\XX$, then~${X_{I}=\XX(|I|)=\XX_{|I|-1}}$ and, in particular, $X_{n}=\XX_{n-1}$. Note how the difference in font style allows to distinguish between the absolute degree $n$ and the simplicial degree $n-1$. 

In presence of enough projectives, we may now characterise higher extensions as follows. 

\pagebreak
\begin{proposition}\label{Projective-Characterisation-Extensions}
Given any $n$-fold arrow $F$ in a regular category with enough projectives, the following are equivalent:
\begin{enumerate}
\item $F$ is an $n$-cubic extension;
\item for any ${(n-1)}$-truncated degreewise projective simplicial object $\XX$, any collection of arrows ${(X_{J}\to F_{J})_{|J|\leq i}}$ satisfying the conditions of a morphism of $n$-fold arrows up to absolute degree $i\in n$ extends to an actual morphism of $n$-fold arrows ${X\to F}$.  
\end{enumerate}
\end{proposition}
\begin{proof}
We first use induction on $i$ to prove that (i) implies (ii). Suppose a collection  ${(X_{J}\to F_{J})_{|J|\leq i}}$ like in the statement is given and let $I\subseteq n$ be such that $|I|=i+1$. Then we obtain the needed morphism ${X_{I}\to F_{I}}$ as the dotted lifting in the diagram
\[
\xymatrix{&&F_{I} \ar@{-{ >>}}[d]\\
X_{I} \ar@{.>}[rru] \ar[r] & \lim_{J\subsetneq I} X_{J} \ar[r] & \lim_{J\subsetneq I} F_{J}}
\]
---which exists because $X_{I}$ is projective by assumption, while the right hand side vertical arrow is a regular epimorphism by Proposition~\ref{Limit-Characterisation-Extensions}.

To see that (ii) implies (i) we again use Proposition~\ref{Limit-Characterisation-Extensions}. This time we have to show that any morphism ${P\to \lim_{J\subsetneq I} F_{J}}$ for $P$ projective lifts to a morphism ${P\to F_{I}}$. We simply use the $(n-1)$-truncation of the constant simplicial object $P$, and extend the collection of arrows induced by the given arrow to a morphism of $n$-fold arrows to obtain the needed lifting.
\end{proof}

\subsection{Simplicial kernels}
Let
\[
(f_i\colon {X\to Y})_{i\in n}
\]
be a sequence of $n$ morphisms in a finitely complete category $\X$. A \defn{simplicial kernel} of~$(f_0,\ldots,f_{n-1})$ is a sequence
\[
(k_i\colon {K\to X})_{i\in n+1}
\]
of $n+1$ morphisms in $\X$ satisfying $f_ik_j=f_{j-1}k_i$ for $0\leq i<j\leq n$, which is universal with respect to this property. It may be computed as a limit in~$\X$.

We need the following lemma, which is probably well known:

\begin{lemma}\label{Lemma coskeleton pullback}
Let $(f_i\colon {X\to Y})_{i\in n}$ and $(f'_i\colon {X'\to Y'})_{i\in n}$ be two sequences of $n$ morphisms in a finitely complete category $\X$, and consider morphisms $\chi\colon {X\to X'}$ and $\upsilon\colon {Y\to Y'}$ for which all squares in any diagram 
\[
\vcenter{\xymatrix{
 X \ar[r]^-{f_{i}} \ar[d]_-{\chi} & Y \ar[d]^-{\upsilon}\\
 X' \ar[r]_-{f'_{i}} & Y'}}
\]
are pullbacks. Then all of the induced squares between the respective simplicial kernels of $(f_{i})_{i\in n}$ and $(f'_{i})_{i\in n}$ are pullbacks as well.
\end{lemma}
\begin{proof}
It suffices to give a formal proof in $\Set$. Consider $x_{m}$ in $X$ and $(x'_{0},\dots,x'_{n})$ in the simplicial kernel $K'$ of $(f'_{i})_{i\in n}$ such that $\chi(x_{m})=x'_{m}$. If $(x_{0},\dots,x_{n})$ is an element of the simplicial kernel $K$ of $(f_{i})_{i\in n}$, then $x_{j}$ necessarily satisfies
\[
\upsilon(f_{m}(x_{j}))=\upsilon(f_{j-1}(x_{m}))=f'_{j-1}(\chi(x_{m}))=f'_{j-1}(x'_{m})=f'_{m}(x'_{j})
\]
in case $m<j$, and 
\[
\upsilon(f_{m-1}(x_{j}))=\upsilon(f_{j}(x_{m}))=f'_{j}(\chi(x_{m}))=f'_{j}(x'_{m})=f'_{m-1}(x'_{j})
\]
when $m>j$. This completely determines $(x_{0},\dots,x_{n})$ via the pullback property which we assume to hold. It is also clear that any tuple $(x_{0},\dots,x_{n})$ thus obtained is indeed an element of $K$. In fact, it turns out to be precisely the needed unique element for which $\chi(x_{0},\dots,x_{n})=(x'_{0},\dots,x'_{n})$ and $k_{m}(x_{0},\dots,x_{n})=x_{m}$.
\end{proof}

When $\XX$ is a simplicial object and $n\geq 0$, we write
\[
(\del_{i}\colon\cycle(\XX,n)\to \XX_{n-1})_{i\in n+1}
\]
for the simplicial kernel of the faces $(\del_{i}\colon {\XX_{n-1}\to \XX_{n-2}})_{i\in n}$. The object $\cycle(\XX,n)$ consists of \defn{$n$-cycles} in $\XX$. For instance, the object $\cycle(\XX,2)$ of $2$-cycles in~$\XX$ contains empty triangles:
\[
\vcenter{\xymatrix@1@!0@R=2.4495em@C=1.4142em{& {\cdot} \ar[rd]^-{\beta}\\
{\cdot} \ar[ru]^-{\alpha} \ar[rr]_-{\gamma} && {\cdot}}}
\]
Note that $\cycle(\XX,n)=\L(\trunc_{n}\XX)$. Clearly~${\cycle(\XX,1)=\R{\del_0}}$; we also write $\cycle(\XX,0)$ for~$\XX_{-1}$.

As mentioned in Subsection~\ref{Truncations}, any ${(n-1)}$-truncated simplicial object $X$ in~$\X$ may be universally extended to an $n$-truncated simplicial object. Its initial object and morphisms, in (absolute!) degree~$n+1$, are given by the simplicial kernel~$(k_{i}\colon{K\to X_{n}})_{i\in n+1}$ of the initial morphisms~${(x_{i}\colon {X_{n}\to X_{n-1}})_{i\in n}}$ of $X$. The degeneracies ${(\sigma_{j}\colon X_{n}\to K)_{j\in n}}$ are induced by the simplicial identities
\[
k_{i}\comp\sigma_{j}=\begin{cases}\sigma_{j-1}\comp x_{i} & \text{if $i<j$} \\
1_{X_{n}} & \text{if $i=j$ or $i=j+1$}\\
\sigma_{j}\comp x_{i-1} & \text{if $i>j+1$}\end{cases}
\]
of $X$ and the universal property of the simplicial kernel. Repeating this construction indefinitely gives the $(n-1)$-coskeleton of $X$.

\subsection{Resolutions}\label{Resolutions}
An augmented simplicial object $\XX$ in a regular category is called \defn{acyclic} or \defn{a resolution (of $\XX_{-1}$)} when for every $n\geq 0$, the comparison morphism
\[
\lind\del_{i}\rind_{i}\colon \XX_{n}\to \cycle(\XX,n)
\]
is a regular epimorphism. (Every $n$-cycle is a boundary of an $n$-simplex.) As explained in~\cite{EGoeVdL}, in a semi-abelian category this is the case precisely when all the truncations of $\XX$, considered as higher arrows, are extensions. For this reason we may sometimes also call a truncated simplicial resolution an extension.

\subsection{The simplicial objects $\KK(A,n)$ and $\KK(Z,A,n)$}
Let $A$ be an abelian group in a Barr-exact category~$\X$ and take $n\geq 1$. The augmented simplicial object~$\KK(A,n)$ is the coskeleton of the $(n+1)$-truncated simplicial object
\[
\vcenter{\xymatrix@R=20pt@C=35pt{\scriptstyle{n+1} & \scriptstyle{n} & \scriptstyle{n-1} & \scriptstyle{n-2} \ar@{}[r]|-{\cdots} & \scriptstyle{0} & \scriptstyle{-1}\\
A^{n+1} \ar@<2.33ex>[r]^-{\del_{n+1}} \ar@<1.16ex>[r]|-{\pr_{n}} \ar@<-2.33ex>[r]_-{\pr_{0}}^-{\vdots} &
 A \ar@<1.75ex>[r]^-{!} \ar@<-1.75ex>[r]_-{!}^-{\vdots} &
 1 \ar@{=}@<1.75ex>[r] \ar@{=}@<-1.75ex>[r]^-{\vdots} & 1 \ar@{}[r]|-{\cdots} & 1 \ar@{=}[r] & 1}}
\]
with the $A$ in simplicial degree $n$ (in absolute degree $n+1$), where the degeneracies ${1\to A}$ are determined by the neutral element $0$ of $A$ and $\del_{n+1}$ is equal to
\[
(-1)^{n}\sum^{n}_{i=0}(-1)^{i}\pr_{i}.
\]
When the category is a slice $\X/Z$ over an object $Z$ in a semi-abelian category~$\X$ and $(A,\xi)$ is a $Z$-module, the simplicial object $\KK((A,\xi),n)$, considered as a diagram in $\X$, takes the following shape:
\[
\resizebox{\textwidth}{!}{
$\vcenter{\xymatrix@R=20pt@C=50pt{\scriptstyle{n+1} & \scriptstyle{n} & \scriptstyle{n-1} & \scriptstyle{n-2} \ar@{}[r]|-{\cdots} & \scriptstyle{0} & \scriptstyle{-1}\\
(A,\xi)^{n+1} \rtimes Z \ar@<2.33ex>[r]^-{\del_{n+1}\rtimes 1_Z} \ar@<1.16ex>[r]|-{\pr_{n}\rtimes 1_Z} \ar@<-2.33ex>[r]_-{\pr_{0}\rtimes 1_Z}^-{\vdots} &
 (A,\xi) \rtimes Z \ar@<1.75ex>[r]^-{f} \ar@<-1.75ex>[r]_-{f}^-{\vdots} &
 Z \ar@{=}@<1.75ex>[r] \ar@{=}@<-1.75ex>[r]^-{\vdots} & Z \ar@{}[r]|-{\cdots} & Z \ar@{=}[r] & Z}}
$}
\]
In case $\xi$ is the trivial module structure $\tau$, we obtain
\[
\resizebox{\textwidth}{!}{
$\vcenter{\xymatrix@C=50pt{\scriptstyle{n+1} & \scriptstyle{n} & \scriptstyle{n-1} & \scriptstyle{n-2} \ar@{}[r]|-{\cdots} & \scriptstyle{0} & \scriptstyle{-1}\\
A^{n+1} \times Z \ar@<2.33ex>[r]^-{\del_{n+1}\times 1_{Z}} \ar@<1.16ex>[r]|-{\pr_{n}\times 1_{Z}} \ar@<-2.33ex>[r]_-{\pr_{0}\times 1_{Z}}^-{\vdots} &
 A \times Z \ar@<1.75ex>[r]^-{\pr_{Z}} \ar@<-1.75ex>[r]_-{\pr_{Z}}^-{\vdots} &
 Z \ar@{=}@<1.75ex>[r] \ar@{=}@<-1.75ex>[r]^-{\vdots} & Z \ar@{}[r]|-{\cdots} & Z \ar@{=}[r] & Z}}
$}
\]
with $\del_{n+1}$ as above and degeneracies $\lind 0,1_{Z}\rind\colon{Z\to A\times Z}$. Given any object $Z$ and any abelian object $A$, we shall write $\KK(Z,A,n)$ for this simplicial object in $\X$. In particular, $\KK(0,A,n)=\KK(A,n)$.

\subsection{(Exact) fibrations}
Let $\XX$ be a simplicial object in a finitely complete category $\X$ and consider $n\geq 2$ and $0\leq i\leq n$. The \defn{object of $(n,i)$-horns in $\XX$} is an object $\horn^{i}(\XX,n)$ together with morphisms $x_{j}\colon{\horn^{i}(\XX,n)\to \XX_{n-1}}$ for $i\neq j\in n+1$ satisfying
\[
\text{$\del_{j}\comp x_{k}=\del_{k-1}\comp x_{j}$ for all $j<k$ with $j$, $k\neq i$}
\]
which is universal with respect to this property; also $\horn^{0}(\XX,1)=\XX_{0}=\horn^{1}(\XX,1)$.

For instance, the object $\horn^{1}(\XX,2)$ of $(2,1)$-horns in $\XX$
\[
\vcenter{\xymatrix@1@!0@R=2.4495em@C=1.4142em{& {\cdot} \ar[rd]^-{\beta}\\
{\cdot} \ar[ru]^-{\alpha} && {\cdot}}}
\]
contains ``composable pairs of arrows''.

We write
\[
\widehat f_{i}=\lind f_j\rind_{i\neq j\in n+1}\colon {W\to \horn^{i}(\XX,n)}
\]
for the morphism induced by a family $(f_{j}\colon W\to \XX_{n-1})_{i\neq j\in n+1}$ in which the morphism $f_{i}$ is missing.

Now suppose that $\X$ is a regular category. A simplicial morphism $\ff\colon{\XX\to \YY}$ satisfies the \defn{Kan condition} (respectively satisfies the Kan condition \defn{exactly}) in degree~$n$ for~$i$ when the morphism
\[
\lind \widehat\del_{i}, \ff_{n}\rind\colon \XX_{n}\to \horn^{i}(\XX,n)\times_{\smallhorn^{i}(\YY,n)}\YY_{n}
\]
universally induced by the square
\[
\vcenter{\xymatrix{\XX_{n} \ar[d]_-{\widehat\del_{i}} \ar[r]^-{\ff_{n}} & \YY_{n} \ar[d]^-{\widehat\del_{i}}\\
\horn^{i}(\XX,n) \ar[r]_{\smallhorn^{i}(\ff,n)} & \horn^{i}(\YY,n)}}
\qquad\qquad
\]
is a regular epimorphism (respectively an isomorphism). The morphism $\ff$ is called a \defn{fibration} when it satisfies the Kan condition for all $n\geq 1$ and all $i$. A~fibration is \defn{exact} in degrees larger than $n$ when the Kan condition is satisfied exactly in simplicial degrees larger than $n$ for all $i$.

A regular category is Mal'tsev if and only if every simplicial object is Kan: every morphism ${\XX\to \one}$ is a fibration~\cite[Theorem~4.2]{Carboni-Kelly-Pedicchio}. Furthermore, a regular epimorphism of simplicial objects in a regular Mal'tsev category is always a fibration~\cite[Proposition~4.4]{EverVdL2}. The Kan property for simplicial objects may also be expressed in terms of higher extensions: in a semi-abelian category, a simplicial object $\XX$ is Kan if and only if all of its truncations, considered as higher arrows in all possible directions, have a domain which is an extension~\cite{EGoeVdL}.

\begin{lemma}\label{Lemma-DeltaLambda-Square}
In a finitely complete category, given $n\geq 1$, $i\in n$, and an augmented simplicial object $\XX$, the square
\[
\vcenter{\xymatrix{\cycle(\XX,n) \ar[r]^-{\widehat\del_{i}} \ar[d]_-{\del_{i}} & \horn^{i}(\XX,n) \ar[d]^-{\del_{i-1}^{i}\times \del_{i}^{n-i}} \\
\XX_{n-1} \ar[r]_-{\lind \del_{j}\rind_{j}} & \cycle(\XX,n-1)}}
\]
is a pullback, where the arrow on the right is the restriction of 
\[
\del_{i-1}^{i}\times \del_{i}^{n-i}=\underbrace{\del_{i-1}\times \cdots\times \del_{i-1}}_{\text{$i$ times}} \times \underbrace{\del_{i}\times \cdots \times \del_{i}}_{\text{$n-i$ times}}\colon \XX_{n-1}^{n}\to \XX_{n-2}^{n} 
\]
to a morphism $\horn^{i}(\XX,n)\to \cycle(\XX,n-1)$.
\end{lemma}
Here is a picture in degree $n=2$ for $i=1$:
\[
\begin{matrix}
\vcenter{\xymatrix@1@!0@R=2.4495em@C=1.4142em{& {\cdot} \ar[rd]^-{\beta}\\
{\cdot} \ar[ru]^-{\alpha} \ar[rr]_-{\gamma} && {\cdot}}}
& \mapsto &
\vcenter{\xymatrix@1@!0@R=2.4495em@C=1.4142em{& {\cdot} \ar[rd]^-{\beta}\\
{\cdot} \ar[ru]^-{\alpha} && {\cdot}}}
\\ \downmapsto && \downmapsto \\
\vcenter{\xymatrix@1@!0@R=2.4495em@C=1.4142em{& {\hole}\\
{\cdot} \ar[rr]_-{\gamma} && {\cdot}}}
& \mapsto &
\vcenter{\xymatrix@1@!0@R=2.4495em@C=1.4142em{& {\hole}\\
{\cdot} && {\cdot}}}
\end{matrix}
\]
\begin{proof}
We again only need to give a formal proof in $\Set$, where it suffices to compare the set of couples 
\[
(x_{i},(x_{0},\dots,\widehat x_{i},\dots,x_{n}))\in \XX_{n-1}\times \XX_{n-1}^{n}
\]
which satisfy 
\[
\text{$\del_{j}(x_{k})=\del_{k-1}(x_{j})$ for all $j<k$ with $j$, $k\neq i$}
\]
and
\[
\del_{j}(x_{i})=
\begin{cases}
\del_{i-1}(x_{j}) & \text{if $j<i$}\\
\del_{i}(x_{j+1}) & \text{if $i\leq j<n$}
\end{cases}
\]
with the set
\[
\{(x_{0},\dots,x_{n})\in \XX_{n-1}^{n+1}\mid \text{$\del_{i}(x_{j})=\del_{j-1}(x_{i})$ for $0\leq i<j\leq n$}\}.
\]
These sets are clearly isomorphic, which finishes the proof.
\end{proof}

\subsection{Higher-dimensional torsors}\label{Torsors}
Let $A$ be an abelian group in a Barr-exact category~$\X$ and consider $n\geq 1$. A \defn{$\KK(A,n)$-torsor} is an augmented simplicial object~$\TT$ equipped with a simplicial morphism $\tt\colon {\TT\to \KK(A,n)}$ such that
\begin{enumerate}
\item[(T1)] $\tt$ is a fibration which is exact from degree $n$ on;
\item[(T2)] $\TT\cong \Cosk_{n-1}\TT$;
\item[(T3)] $\TT$ is a resolution.
\end{enumerate}
Let $Z$ be an object of a semi-abelian category $\X$ and $(A,\xi)$ a $Z$-module. An \defn{$n$"~torsor of $Z$ by $(A,\xi)$} is a $\KK((A,\xi),n)$-torsor in the category $\X/Z$. Morphisms of $\KK(A,n)$-torsors are defined as in the slice over $\KK(A,n)$, and thus we obtain the category $\Tors^{n}(\X,A)$ of $\KK(A,n)$-torsors in $\X$ as a full subcategory of $\S^{+}(\X)/\KK(A,n)$. When the action $\xi$ is trivial, we call $(\TT,\tt)$ an~\defn{$n$-torsor of $Z$ by $A$}, and obtain the following picture:
\[
\resizebox{\textwidth}{!}{\mbox{$
\vcenter{\xymatrix@=45pt{
\cycle(\TT,n+1) \ar[d]_{\cdots\quad\lind\lind\varsigma\circ\del_{i}\rind_{i},\del_{0}^{n+2}\rind} \ar@<2.33ex>[r] \ar@<1.16ex>[r] \ar@<-2.33ex>[r]^-{\vdots} & \cycle(\TT,n) \ar[d]^{\lind\varsigma,\del_{0}^{n+1}\rind} \ar@<1.75ex>[r] \ar@<-1.75ex>[r]^-{\vdots} & \TT_{n-1} \ar[d]^-{\del_{0}^{n}} \ar@<1.75ex>[r] \ar@<-1.75ex>[r]^-{\vdots} & \TT_{n-2} \ar[d]^-{\del_{0}^{n-1}} \ar@{}[r]|-{\cdots} & \TT_{0} \ar[d]^-{\del_{0}} \ar[r]^-{\del_{0}} & \TT_{-1} \ar@{=}[d]\\
A^{n+1} \times Z \ar@<2.33ex>[r]^-{\del_{n+1}\times 1_{Z}} \ar@<1.16ex>[r]|-{\pr_{n}\times 1_{Z}} \ar@<-2.33ex>[r]_-{\pr_{0}\times 1_{Z}}^-{\vdots} &
A \times Z \ar@<1.75ex>[r]^-{\pr_{Z}} \ar@<-1.75ex>[r]_-{\pr_{Z}}^-{\vdots} & Z \ar@{=}@<1.75ex>[r] \ar@{=}@<-1.75ex>[r]^-{\vdots} & Z \ar@{}[r]|-{\cdots} & Z \ar@{=}[r] & Z}}
$}}
\]
When $Z$ is an object of a semi-abelian category $\X$ and~$A$ is an abelian object in $\X$ considered as a trivial $Z$-module $(A,\tau)$, we write $\Tors^{n}(Z,A)$ for the category~$\Tors^{n}(\X/Z,(A,\tau))$. Taking connected components we obtain the set
\[
\Tors^{n}[Z,A]=\pi_{0}\Tors^{n}(Z,A)
\]
of equivalence classes of $n$-torsors of $Z$ by $A$ which is, in fact, an abelian group~\cite{Duskin-Torsors}.

We shall further analyse the concept of torsor in Section~\ref{Section-Torsors-and-Centrality}; for now it suffices to understand their cohomological meaning.

\subsection{The $(n+1)$-th cohomology group}
It follows from~\cite[Theorem 5.2]{Duskin-Torsors} that, when $\X$ is a Barr-exact category and
\[
\GG=(G\colon\X\to\X,\,\delta\colon G\To G^{2},\,\epsilon\colon G\To 1_{\X})
\]
is a comonad on $\X$ such that the $\GG$-projectives coincide with the regular projectives in $\X$, then
\[
\H^{n+1}(1,A)_{\GG}\cong \pi_{0}\Tors^{n}(\X,A)
\]
where $A$ is an internal abelian group in $\X$ and $1$ is the terminal object. If now~$Z$ is an object of~$\X$ then $\GG$ induces a comonad $\GG/Z=(G^{Z},\delta^{Z},\epsilon^{Z})$ on~$\X/Z$ via
\[
\delta^{Z}_{f}=\left(\vcenter{\xymatrix@!0@R=3em@C=2em{GX \ar[rr]^-{\delta_{X}} \ar[rd]_-{G^{Z}f=f\circ\epsilon_{X}} && GGX \ar[ld]^-{G^{Z}G^{Z}f=f\circ\epsilon_{X}\circ\epsilon_{GX}} \\ & Z}}\right)
\quad\text{and}\quad
\epsilon^{Z}_{f}=\left(\vcenter{\xymatrix@!0@R=3em@C=2em{GX \ar[rr]^-{\epsilon_{X}} \ar[rd]_-{G^{Z}f=f\circ\epsilon_{X}} && X \ar[ld]^-{f} \\ & Z}}\right)
\]
for all $f\colon{X\to Z}$. Hence when, in a semi-abelian category $\X$, we consider an abelian object $A$ as a trivial~$Z$-module, we see that
\[
\H^{n+1}(1_{Z},(A,\tau))_{\GG/Z}\cong \pi_{0}\Tors^{n}(\X/Z,(A,\tau))
\]
and
\[
\H^{n+1}(Z,A)_{\GG}\cong\Tors^{n}[Z,A].
\]
For instance, $\X$ may be chosen to be a variety of algebras over $\Set$, so that~$\GG$ is canonically induced by the forgetful/free adjunction. In any case, $\Tors^{n}[Z,A]$ does indeed carry an abelian group structure. Moreover, this defines an additive functor
\[
\Tors^{n}[Z,-]\colon \Ab(\X)\to \Ab.
\]

\pagebreak
\section{The groups of equivalence classes of higher central extensions}\label{Section-Central-Extensions}
We work towards a definition of the group $\Centr^{n}(Z,A)$ of equivalence classes of~$n$-fold central extensions of $Z$ by $A$, extending the definition of $\Centr^{2}(Z,A)$ given in Section~4 of~\cite{RVdL}. We start with some basic theory of (higher-dimen\-sion\-al) central extensions, first recalling known results and then proving some new ones.

\subsection{Central extensions}\label{Central extensions}
We first consider some general definitions and results valid in a homological category with a chosen strongly Birkhoff subcategory. Here we follow~\cite{EGVdL}.

A \defn{Galois structure}~\cite{Janelidze:Precategories} $\Gamma=(\X,\B,\E,\F,I,H)$ consists of categories~$\X$ and~$\B$, an adjunction
\[
\xymatrix{\X \ar@<1 ex>[r]^-{I} \ar@{}[r]|-{} & \B, \ar@<1 ex>[l]^-H \ar@{}[l]|{\perp}}
\]
and classes $\E$ and $\F$ of morphisms of $\X$ and $\B$ respectively, such that:
\begin{enumerate}
\item $\X$ has pullbacks along morphisms in $\E$;
\item
$\E$ and $\F$ contain all isomorphisms, are closed under composition and are pullback-stable;
\item
$I(\E) \subseteq \F$;
\item
$H(\F) \subseteq \E$.
\end{enumerate}
An element of $\E$ is called an \defn{$\E$-extension}.

We shall only consider Galois structures where $\X$ is (at least) a homological category, all $\E$-extensions are regular epimorphisms, and $\B$ is a full replete $\E$-reflective subcategory of $\X$. We shall never write its inclusion $H$. Such a subcategory is called \defn{strongly $\E$-Birkhoff} when for every $\E$-extension $f\colon{X\to Z}$ the induced naturality square
\begin{equation}\label{Birkhoff-Square}
\vcenter{\xymatrix{X \ar@{ >>}[r]^-{f} \ar@{ >>}[d]_-{\eta_{X}} & Z \ar@{ >>}[d]^-{\eta_{Z}} \\
IX \ar@{ >>}[r]_-{If} & IZ}}
\end{equation}
is a two-cubic $\E$-extension. (The universally induced morphism to the pullback must be in~$\E$.) From now on we shall always assume this to be the case.

If $\X$ is an exact Mal'tsev category and $\E$ consists of all regular epimorphisms, a strongly $\E$-Birkhoff subcategory of $\X$ is precisely a \defn{Birkhoff subcategory}: full, reflective and closed under subobjects and regular quotients in~$\X$, see~\cite{Janelidze-Kelly}. A Birkhoff subcategory of a variety of algebras is the same thing as a subvariety. Outside the exact Mal'tsev context, however, when $\E$ is the class of regular epimorphisms, the strong $\E$-Birkhoff property is generally stronger than the Birkhoff property, since not every pushout of extensions needs to be a two-cubic extension.

\begin{example}[Abelianisation]
It is well known that, in any semi-abelian category~$\X$, the full subcategory~$\Ab(\X)$ determined by the abelian objects is Birkhoff. This is the situation which we shall be most interested in here, in particular from Subsection~\ref{Directions} on.
\end{example}

An $\E$-extension $f\colon{X\to Z}$ in $\X$ is \defn{trivial} when the induced square~\eqref{Birkhoff-Square} is a pullback. Of course, if $X$ and $Z$ lie in $\B$ then $f$ is a trivial $\E$-extension. The $\E$-extension~$f$ is said to be \defn{normal} when both projections $\pr_{0}$, $\pr_{1}$ in the kernel pair $(\R{f},\pr_{0},\pr_{1})$ of $f$ are trivial. Finally, $f$ is \defn{central} when there exists an~$\E$"~extension $g\colon{Y\to Z}$ such that the pullback of $f$ along $g$ is trivial.

It is clear that every trivial $\E$-extension is central. Moreover, every normal $\E$"~extension is central; in the present context, also the converse holds (via Theorem~4.8 of~\cite{Janelidze-Kelly} or Proposition~2.6 in~\cite{EGVdL}). Hence the concepts of normality and centrality coincide. It follows immediately from the definition that pullbacks of $\E$-extensions along $\E$-extensions reflect centrality. Furthermore, in the present context, Proposition~4.1 and~4.3 in~\cite{Janelidze-Kelly} may be modified to prove that both the classes of trivial and of central $\E$-extensions are pullback-stable. It is also well known that a split epimorphic central $\E$-extension is always trivial.

The following important result (see~\cite{Gran-Alg-Cent,EGVdL}) will be used in Section~\ref{Section-Geometry}.

\begin{lemma}\label{Lemma-Pullbacks}
When $\X$ is a homological category and $\B$ is a strongly $\E$-Birkhoff subcategory of $\X$, the reflector $I\colon{\X\to \B}$ preserves pullbacks of $\E$-extensions along split epimorphisms.\noproof
\end{lemma}

\subsection{The tower of Galois structures for cubic central extensions}\label{Subsection-Tower}
Now we describe the Galois structures for centrality of $n$-cubic extensions introduced in~\cite{EGVdL}. We start with a semi-abelian category $\X$ and a Birkhoff subcategory~$\B$ of~$\X$. Choosing $\E$ and $\F$ to be the classes of regular epimorphisms in~$\X$ and~$\B$, we obtain a Galois structure $\Gamma$ as above---$\B$ is strongly $\E$-Birkhoff. We may now drop the prefix~$\E$; the elements of this class are the one-cubic extensions of Subsection~\ref{Extensions}.

Let us view the objects of $\X$ as zero-cubic extensions, and the objects of $\B$ as zero-cubic central extensions. With respect to the Galois structure $\Gamma_{0}=\Gamma$, there is the notion of central extension, and it is such that the full subcategory~$\CExt^{1}_{\B}(\X)$ of~$\Ext^{1}(\X)$ determined by those one-cubic central extensions is again reflective. Its reflector $I_{1}\colon{\Ext^{1}(\X)\to \CExt^{1}_{\B}(\X)}$, together with the classes $\E^{1}$ and $\F^{1}$ of one-cubic extensions in $\Ext^{1}(\X)$ and in $\CExt^{1}_{\B}(\X)$ (which we choose to be two-cubic extensions in~$\X$, and two-cubic extensions with central domain and codomain), in turn determines a Galois structure $\Gamma_{1}$. This Galois structure is again ``nice'' in that $\CExt^{1}_{\B}(\X)$ is again strongly $\E^{1}$-Birkhoff in the homological category~$\Ext^{1}(\X)$. Inductively, this defines a family of Galois structures $(\Gamma_{n})_{n\geq 0}$:
\[
\Gamma_{n}=(\Extn(\X),\CExt^{n}_{\B}(\X),\E^{n},\F^{n},I_{n},\subseteq),
\]
each of which gives rise to a notion of $(n+1)$-cubic central extension which determines the next structure~\cite[Theorem~4.6]{EGVdL}. (Here $\E^0=\E$, $\F^0=\F$ and $I_0=I$.) In particular, for every $n\geq 1$ we obtain a reflector (the centralisation functor)
\[
I_{n}\colon{\Extn(\X)\to\CExt^{n}_{\B}(\X),}
\]
left adjoint to the inclusion $\CExt^{n}_{\B}(\X)\subseteq\Extn(\X)$.

For any $n\geq 1$, the $n$-cubic extension $\lb F\rb_{\CExt^{n}_{\B}(\X)}$ in the short exact sequence
\[
\xymatrix{0 \ar[r] & \lb F\rb_{\CExt^{n}_{\B}(\X)} \ar@{{ |>}->}[r]^-{\mu^{n}_{F}} & F \ar@{-{ >>}}[r]^-{\eta^{n}_{F}} & I_{n}F \ar[r] & 0}
\]
induced by the centralisation of an $n$-cubic extension $F$ is zero everywhere except in its initial object $\lb F\rb^{n}=(\lb F\rb_{\CExt^{n}_{\B}(\X)})_{n}$, because centralisation keeps all objects in an $n$-cubic extension fixed except the initial object. So, restricting to initial objects, we obtain a short exact sequence
\begin{equation}\label{brackets}
\xymatrix{0 \ar[r] & \lb F\rb^{n} \ar@{{ |>}->}[r] & F_{n} \ar@{-{ >>}}[r] & I_{n}[F] \ar[r] & 0}
\end{equation}
in $\X$. In parallel with the case~${n=0}$ considered in Subsection~\ref{Subsection-Abelian}, this object $\lb F\rb^{n}$ acts like an $n$-dimensional commutator which may be computed as the kernel of the restriction of the kernel pair projection $(\pr_{0})_{n-1}\colon{\R{F}_{n-1}\to \dom F_{n-1}}$ to a morphism
\[
\lb \pr_{0}\rb^{n-1}\colon {\lb \R{F}\rb^{n-1}\to \lb \dom F\rb^{n-1}}
\]
in $\X$. Furthermore, an $n$-cubic extension~$F$ is central if and only if the induced morphisms
\[
\lb \pr_{0}\rb^{n-1}, \lb \pr_{1}\rb^{n-1}\colon{\lb\R{F}\rb^{n-1}\to \lb \dom F\rb^{n-1}}
\]
are isomorphisms---which happens precisely when they coincide; see~\cite{EverHopf, EGVdL} for more details. The notation $\lb F\rb^{n}$ not mentioning the Birkhoff subcategory $\B$ need not lead to confusion, because the only case which we shall use it in is $\B=\Ab(\X)$; keeping this in mind, we also write $\lb X\rb^{0}=\lb X\rb$ for the kernel of $\eta_{X}\colon {X\to \ab X}$ when $X$ is an object of $\X$.

\begin{definition}\label{def n-fold central extension}
An \defn{$n$-fold central extension} is an $n$-fold extension of which the underlying $n$-cubic extension is central.
\end{definition}

\begin{example}[The simplicial objects $\KK(Z,A,n)$]\label{Example-K(A,n)-as-Extension}
Given any integer $n\geq 1$, any object $Z$ and any abelian object~$A$ in~$\X$, the $(n+1)$-cubic extension underlying~$\KK(Z,A,n)$ is always trivial with respect to abelianisation. This follows by induction from the fact that both its domain and its codomain are $n$-cubic trivial extensions. Note, however, that the~${(n+2)}$-fold arrow underlying $\KK(Z,A,n)$ is not even a cubic extension!
\end{example}

\begin{example}[One-cubic central extensions]\label{Example-Dimension-One}
Recall that a surjective group homomorphism $f\colon{X\to Z}$ is central (with respect to $\Ab$) if and only if $[\K{f},X]=0$. This result was adapted to a semi-abelian context in~\cite{Bourn-Gran, Gran-Alg-Cent}: when $\X$ is a semi-abelian category and $\B=\Ab(\X)$ is the Birkhoff subcategory determined by all abelian objects in $\X$, the one-cubic central extensions induced by the Galois structure (the ``categorically central'' ones) are the central extensions in the algebraic sense. These may be characterised through the Smith/Pedicchio commutator of equivalence relations as those $f\colon{X\to Z}$ such that~${[\R{f},\nabla_{X}]=\Delta_{X}}$, which means that the kernel pair of the arrow~$f$ is a central equivalence relation (Subsection~\ref{Commutators}). A characterisation closer to the group case appears in~\cite{Gran-VdL} where the condition is reformulated in terms of the Huq commutator of normal subobjects so that it becomes~$[\K{f},X]=0$.
\end{example}

\begin{example}[Double central extensions]\label{Example-Dimension-Two}
One level up, the double central extensions of groups vs.\ abelian groups were first characterised in~\cite{Janelidze:Double}: a two-cubic extension such as~\eqref{Double-Extension} above is central if and only if
\[
[\K{d},\K{c}]=0=[\K{d}\cap \K{c},X].
\]
General versions of this characterisation were given in~\cite{Gran-Rossi} for Mal'tsev varieties, then in~\cite{RVdL} for semi-abelian categories and finally in~\cite{EverVdL3} for exact Mal'tsev categories: the two-cubic extension~\eqref{Double-Extension} is central (with respect to abelianisation) if and only if
\begin{equation}\label{Double-Central}
[\R{d},\R{c}]=\Delta_{X}=[\R{d}\cap \R{c},\nabla_{X}].
\end{equation}
This means that the span $(X,d,c)$ is a special kind of pregroupoid in the slice category $\X/Z$.

The main technical problem here is that later on, we will use the Huq commutator of normal monomorphisms rather than the Smith/Pedicchio commutator of equivalence relations---and the correspondence between the two which exists in level one is no longer there when we go up in degree. In fact, it is well known and easily verified that if the Smith/Pedicchio commutator of two equivalence relations is trivial, then the Huq commutator of their normalisations is also trivial~\cite{BG}. But, in general, the converse is false; in~\cite{Borceux-Bourn, Bourn2004} a counterexample is given in the category of digroups, which is a semi-abelian variety, even a variety of~$\Omega$-groups~\cite{Higgins}. The equivalence of these commutators is known as the \defn{\emph{Smith is Huq} condition (SH)} and it is shown in~\cite{MFVdL} that, for a semi-abelian category, this condition holds if and only if every star-multiplicative graph is an internal groupoid, which is important in the study of internal crossed modules~\cite{Janelidze}. Moreover, the \emph{Smith is Huq} condition is also known to hold for pointed strongly protomodular categories~\cite{BG} (in particular, for any Moore category~\cite{Rodelo:Moore}) and in action accessible categories~\cite{BJ07} (in particular, for any category of interest~\cite{Montoli, Orzech}).

The condition (SH) also implies that every action of an object on an abelian object is a module: here, the equality $[\R{f},\R{f}]=\Delta_{X}$ in Subsection~\ref{Subsection-Abelian} follows from $[\K{f},\K{f}]=[A,A]=0$.
\end{example}

\subsection{Two lemmas on higher centrality}
The centrality of a cubic extension implies that certain induced lower-dimensional cubic extensions are also central. The present proof of Lemma~\ref{Lemma-Direction-Limit} was kindly offered to us by Everaert and Gran; it is more general and more elegant than our original proof. In the case of abelianisation, it also follows easily from Theorem~\ref{Theorem-Higher-Centrality}. We first recall a well-known result (essentially Proposition~4.3 in \cite{Janelidze-Kelly}):

\begin{lemma}\label{Lemma-Kernel-Central-Extension}
In a semi-abelian category with a chosen Birkhoff subcategory, let $f\colon{X\to Y}$ be an $n$-cubic central extension considered as an arrow between $(n-1)$-cubic extensions $X$ and $Y$. Then its kernel $K$ is an $(n-1)$-cubic central extension.
\end{lemma}
\begin{proof}
The $(n-1)$-cubic extension $K$ may be obtained as the kernel of the $n$"~cubic trivial extension $f_{0}\colon{\Eq(f)\to X}$, hence also as the kernel of the $n$-cubic extension $I_{n-1}f_{0}\colon{I_{n-1}\Eq(f)\to I_{n-1}X}$ in $\CExt_{\B}^{n-1}(\X)$.
\end{proof}

\begin{lemma}\label{Lemma-Direction-Limit}
Let $F$ be an $n$-cubic central extension in a semi-abelian category with a chosen Birkhoff subcategory. Then the one-cubic extension $l_{F}\colon{F_{n}\to \L F}$ induced by $F$ is always central.
\end{lemma}
\begin{proof}
The case $n=1$ is clear (because then $F=l_{F}$), so take $n\geq 2$. We shall prove that for an $n$-cubic central extension, considered as a square~\eqref{Double-Extension} of $(n-1)$-cubic extensions, the induced comparison $\lind d,c\rind\colon{X\to D\times_{Z}C}$ is an $(n-1)$-cubic central extension $G$. Then the claim follows by induction, because $l_{F}=l_{G}$ by Lemma~\ref{Lemma-L-pullbacks}. Since $\lind d,c\rind$ is an $(n-1)$-cubic extension by definition, we just have to show its centrality.

First we may reduce the situation to trivial extensions. Indeed, taking kernel pairs to the left, we obtain the diagram
\[
\vcenter{\xymatrix{\R{c} \ar@{-{ >>}}[d]_{r} \ar@{-{ >>}}@<.5ex>[r]^-{c_{1}} \ar@{-{ >>}}@<-.5ex>[r]_-{c_{0}} & X \ar@{-{ >>}}[r]^{c} \ar@{-{ >>}}[d]_-{d} & C \ar@{-{ >>}}[d]^{g}\\
\R{f} \ar@<.5ex>@{-{ >>}}[r]^-{f_{1}} \ar@{-{ >>}}@<-.5ex>[r]_-{f_{0}} & D \ar@{-{ >>}}[r]_-{f} & Z.}}
\]
It is not hard to see that the induced comparison $\lind r,c_{0}\rind\colon{\R{c}\to \R{f}\times_{D}X}$ is a pullback of the cubic extension $\lind d,c\rind$: the diagram
\[
\vcenter{\xymatrix@1@!0@=35pt{
\Eq(c) \pullbackdots \ar[rr]^-{\lind r,c_{0}\rind} \ar[dd]_-{c_{0}} && \Eq(f)\times_{D}X \pullback \skewpullback \ar[rr]^-{d^{*}f_{1}} \ar[dd]|(.5){\hole} \ar[ld]_(.6){f_{1}^{*}d} && X \ar[dd]^-{c} \ar[ld]_-{d} \\
&\Eq(f) \pullback \ar[rr]^(.75){f_{1}} \ar[dd]_(.25){f_{0}} && D \ar[dd]^(.25){f} \\
X \ar[rr]_(.25){\lind d,c\rind}|(.5){\hole} && D\times_{Z}C \skewpullback \ar[ld]|-{f^{*}g} \ar[rr]|(.5){\hole}|(.7){g^{*}f} && C \ar[ld]^-{g} \\
&D \ar[rr]_-{f} && Z}}
\]
shows how the back pullback rectangle decomposes as a composite of pullback squares. Hence if $\lind r,c_{0}\rind$ is central then so is $\lind d,c\rind$, because pulling back reflects centrality.

Now we reduce to $n$-cubic extensions between $(n-1)$-cubic central extensions. Suppose that the square~\eqref{Double-Extension}, viewed as an arrow from $d$ to $g$, is an $n$-cubic trivial extension. Consider the following cube, which displays the centralisation of $d$ and of $g$ using the notation of~\eqref{brackets}:
\[
\vcenter{\xymatrix@!0@=2.5em{&& {I_{n-1}[d]} \ar@{~{ >>}}[rd] \ar@{-{ >>}}[rrrr] \ar@{.{ >>}}[dddd]|-{\hole} &&&& {I_{n-1}[g]} \ar@{-{ >>}}[dddd]^-{I_{n-1}g} \\
&&& {\cdot} \ar@{.{ >>}}[lddd] \ar@{.{ >>}}[rrru] \pullback \\
{X} \ar@{~{ >>}}[rd] \ar@{-{ >>}}[rrrr] \ar@{-{ >>}}[dddd]_-{d} \ar@{~{ >>}}[rruu] &&&& {C} \ar@{-{ >>}}[dddd]^(.25){g} \ar@{-{ >>}}[rruu]\\
& {\cdot} \ar@{~{ >>}}[rruu]|-{\hole} \ar@{-{ >>}}[lddd] \ar@{-{ >>}}[rrru] \pullback \\
&& {D} \ar@{.{ >>}}[rrrr] &&&& {Z} \\\\
{D} \ar@{-{ >>}}@<-.5ex>[rrrr] \ar@{:}[rruu] &&&& {Z} \ar@{=}[rruu]}}
\]
Recall from Subsection~\ref{Subsection-Tower} that the centralisation functor only changes the domain of a cubic extension, which explains the two identity morphisms in the cube.
Since the front square is a trivial extension, the top square is a pullback. By pullback cancellation, the top square of the prism between the front and back pullbacks is also a pullback, and it follows that the square of wiggly arrows is a pullback too. This completes the reduction, since cubic central extensions are pullback-stable.

Finally, for an $n$-cubic extension~\eqref{Double-Extension} between $(n-1)$-cubic central extensions~$d$ and~$g$ the claim that $\lind d,c\rind\colon{X\to D\times_{Z}C}$ is an $(n-1)$-cubic central extension holds, since $\lind d,c\rind$ is a subobject of the cubic extension $d$ in the category of $(n-1)$-cubic central extensions---indeed, a monomorphism of cubic extensions is a square of which the top map is a monomorphism---and cubic central extensions are closed under subobjects.
\end{proof}

\subsection{(Central) extensions over a fixed base object}
Let $Z$ be an object of~$\X$ and $n\geq 1$. Denote by $\Ext^{n}_{Z}(\X)$ the category of \defn{$n$-fold extensions of $Z$} or \defn{over~$Z$}, defined as the fibre over $Z$ (the pre-image of the identity $1_Z$) of the functor
\[
(-)_{0,\dots,0}\colon\Diag{n}(\X)\to \X\colon E\mapsto E_{0,\dots,0}
\]
which projects an $n$-fold extension on its terminal object---see Subsection~\ref{3^n diagrams}. Thus the objects of $\Ext^{n}_{Z}(\X)$ are $3^{n}$-diagrams with ``terminal object'' $Z$, and the morphisms are those morphisms in $\Diag{n}(\X)$ which restrict to the identity on $Z$ under the functor $(-)_{0,\dots,0}$. Similarly $\CExt^{n}_{Z}(\X)$ is the full subcategory of $\Ext^{n}_{Z}(\X)$ determined by the $n$-fold extensions of $Z$ that are central with respect to $\B$ as in Definition~\ref{def n-fold central extension}. (The index~$\B$ being dropped here is not really problematic, since we shall take~$\B$ equal to $\Ab(\X)$ anyway from Subsection~\ref{Directions} on.) Sending an $n$"~fold (central) extension to its underlying $n$-cubic (central) extension, we obtain an equivalence with the category of \defn{$n$-cubic (central) extensions over $Z$}, the fibre over $Z$ of the functor
\[
\cod^n=\underbrace{\cod\comp\cdots\comp\cod}_{\text{$n$ times}}=(-)_{0}\colon {\Ext^{n}(\X)\to \X\colon F\mapsto F_{0}}
\]
or, in the case of central extensions, its restriction to $\CExt^{n}_{\B}(\X)$.

\begin{lemma}\label{Lemma-Product-Over-Z}
Consider a semi-abelian category $\X$ with a chosen Birkhoff subcategory. Let $Z$ be an object of $\X$ and $n\geq 1$. Then $\Ext^{n}_{Z}(\X)$ and $\CExt^{n}_{Z}(\X)$ have binary products: the product of two $n$-fold (central) extensions $F$ and $G$ over $Z$ is an $n$-fold (central) extension over $Z$. Moreover, $l_{F\times G}=l_{F}\times_{Z}l_{G}$.
\end{lemma}
\begin{proof}
Given two $n$-cubic extensions $F$ and $G$ over $Z$, their product $F\times G$ in the category~$\Ext^{n}_{Z}(\X)$ is given pointwise by pullbacks in $\X$:
\[
(F\times G)_{I}=F_{I}\times_{Z}G_{I}
\]
for $I\subseteq n$. To see that it is the product as $n$-fold arrows over $Z$, it suffices to verify the universal property. This $n$-fold arrow is indeed an $n$-cubic extension by Proposition~\ref{Limit-Characterisation-Extensions}, since
\begin{align*}
{(F\times G)_{I}\to\lim_{J\subsetneq I} (F\times G)_{J}} &= {(F_{I}\times_{Z} G_{I})\to\lim_{J\subsetneq I} (F_{J}\times_{Z} G_{J})}\\
&= {(F_{I}\times_{Z} G_{I})\to(\lim_{J\subsetneq I} F_{J}\times_{Z} \lim_{J\subsetneq I} G_{J})}\\
&= {(F_{I}\to\lim_{J\subsetneq I} F_{J})\times_{Z} (G_{I}\to \lim_{J\subsetneq I} G_{J})}
\end{align*}
for all $\emptyset\neq I\subseteq n$. Note that in particular,
\[
l_{F\times G}=l_{F}\times_{Z}l_{G}\colon{(F_{n}\to\lim_{J\subsetneq n} F_{J})\times_{Z} (G_{n}\to \lim_{J\subsetneq n} G_{J}).}
\]
The $n$-cubic extension $F\times G$ is central by the Birkhoff property of $\CExt^{n}_{\B}(\X)$, it being a subobject in $\Ext^{n}(\X)$ of an $n$-cubic central extension. Indeed it is a subobject of the product $(F_{I}\times G_{I})_{I\subseteq n}$ of~$F$ and~$G$ in $\CExt^{n}_{\B}(\X)$ which, since $\CExt^{n}_{\B}(\X)$ is a reflective subcategory, is computed pointwise as in $\Ext^{n}(\X)$.
\end{proof}

\subsection{The direction of a higher (central) extension}\label{Directions}
From now on we assume that $\X$ is a semi-abelian category and $\B=\Ab(\X)$ is the Birkhoff subcategory determined by the abelian objects of $\X$. We introduce the concept of \emph{direction} for $n$-fold (central) extensions in $\X$, which is crucial in the definition and in the study of the groups $\Centr^{n}(Z,A)$. As explained in~\cite{RVdL}, this notion is based on Bourn's concept of direction for internal groupoids~\cite{Bourn2002b}.

\begin{definition}\label{Definition-Centr}
The \defn{direction} of an $n$-fold extension $E$ is its initial object $E_{2,\dots,2}$ (see Subsection~\ref{3^n diagrams}). 

From the point of view of the underlying $n$-cubic extension $F$, it is the object $\Kn{F}$, obtained by taking kernels $n$ times---each time considering a $(k+1)$-cubic extension as an arrow between $k$-cubic extensions---in the way determined by the extension $E$. If $F$ is an $n$-cubic extension, then its kernel $\K{F}$ is an $(n-1)$-cubic extension, whose kernel is an $(n-2)$-cubic extension $\Ktwo{F}$, and so on. By taking kernels $n$ times we obtain a $0$-cubic extension, so an object $\Kn{F}$.

If $E$ is central then the direction of $E$ is an abelian object of $\X$ by Lemma~\ref{Lemma-Kernel-Central-Extension} and the convention regarding zero-cubic central extensions. Given any object $Z$ of~$\X$, this defines the \defn{direction functor}
\[
\D_{(n,Z)} \colon{\CExt^n_Z(\X)\to \Ab(\X).}
\]
The fibre $\D^{-1}_{(n,Z)}A$ of this functor over an abelian object~$A$ is the category of \defn{$n$"~fold central extensions of $Z$ by $A$}, which are special $3^{n}$-diagrams under $A$ and over~$Z$. Two $n$-fold central extensions of~$Z$ by~$A$ which are connected by a zigzag in $\D^{-1}_{(n,Z)}A$ are called \defn{equivalent}. As explained in Subsection~\ref{Connected-components}, the equivalence classes, which we shall denote $[E]$ for $E$ an $n$-fold central extension of~$Z$ by~$A$, form the set
\[
\Centr^{n}(Z,A)=\pi_0(\D ^{-1}_{(n,Z)}A)
\]
of connected components of the category $\D^{-1}_{(n,Z)}A\subseteq \CExt^n_Z(\X)$. 
\end{definition}

\begin{remark}
Abusing terminology, when this does not lead to confusion, we sometimes talk about \emph{the direction of an $n$-cubic extension}---which is only determined up to isomorphism, since this $n$-cubic extension may be part of many $n$-fold extensions. 
\end{remark}

\begin{lemma}\label{Direction-as-Kernel}
For any $n$-fold central extension $E$ with underlying $n$-cubic extension $F$ we have
\[
\D_{(n,Z)} E=\K{l_{F}}=\bigcap_{i\in n} \K{f_{i}}
\]
where $l_{F}$ and the morphisms $f_{i}$ are as in~\ref{HDA}.
\end{lemma}
\begin{proof}
The chain
\[
\K{l_{F}}=\K{l_{\K{F}}}=\cdots=\K{l_{\Knminus{F}}}=\K{\Knminus{F}}
\]
gives us the first equality; the second is immediate from the definition.
\end{proof}

\begin{remark}
For an $n$-cubic extension $F$ underlying an $n$-fold extension $E$, an ``element'' $x$ of $F_{n}$ is an $n$-dimen\-sion\-al hyper-tetrahedron with faces $x_{i}=f_{i}(x)$. Such a tetrahedron is in the direction of~$E$ precisely when all its faces $x_{i}$ are zero---see Figure~\ref{Figure-Direction} on page~\pageref{Figure-Direction} for the case~${n=3}$.
\end{remark}

\subsection{The group structure on $\Centr^{n}(Z,A)$}
We are now ready to show that the set $\Centr^{n}(Z,A)$ of equivalence classes of $n$-fold central extensions of~$Z$ by~$A$ carries a canonical abelian group structure (Corollary~\ref{Corollary-Centr^{n}(Z,-)}).

\begin{lemma}\label{Lemma-Product-and-Direction}
For any object $Z$ of a semi-abelian category $\X$ and any $n\geq 1$, the direction functor~$\D_{(n,Z)} \colon{\CExt^n_Z(\X)\to \Ab(\X)}$ preserves finite products.
\end{lemma}
\begin{proof}
The terminal object $1$ of $\CExt^n_Z(\X)$ is determined by the ``constant'' $n$-cubic central extension of~$Z$ formed out of the identities~$1_Z$; it is clear that the direction of $1$ is zero.

Given two $n$-fold central extensions with respective underlying $n$-cubic extensions~$F$ and $G$ over $Z$ and directions $A$ and $B$, we have to prove that their product over $Z$ has direction~${A\times B}$. Lemma~\ref{Lemma-Product-Over-Z} tells us that the product in question does indeed exist. While lemmas~\ref{Direction-as-Kernel} and~\ref{Lemma-Product-Over-Z} give us the direction:
the kernel of~${l_{F\times G}=l_{F}\times_{Z}l_{G}}$ is $A\times_Z B=A\times B$, since the morphisms from $A$ and~$B$ to~$Z$ are null.
\end{proof}

\begin{proposition}\label{Proposition-Centr^{n}(Z,-)}
Let $Z$ be an object of a semi-abelian category~$\X$. Mapping any abelian object $A$ of~$\X$ to the set $\Centr^{n}(Z,A)$ of equivalence classes of~$n$-fold central extensions of $Z$ by $A$ gives a finite product-preserving functor
\[
\Centr^{n}(Z,-)\colon \Ab(\X)\to \Set.
\]
\end{proposition}
\begin{proof}
We explain how the functoriality of $\Centr^{n}(Z,-)$ follows from the functoriality of~$\Centr^{1}(\L F,-)$, which is an instance of Proposition~6.1 in~\cite{Gran-VdL}. Given an $n$-fold central extension $E$ of $Z$ by $A$ with underlying $n$-cubic extension $F$, we have an induced one-fold central extension
\[
\xymatrix{0 \ar[r] & A \ar@{{ |>}->}[r]^-{k_{F}} & F_{n} \ar@{ >>}[r]^{l_{F}} & \L F \ar[r] & 0}
\]
by Lemma~\ref{Lemma-Direction-Limit} and Lemma~\ref{Direction-as-Kernel}. Now let $a\colon {A\to B}$ be a morphism of abelian objects in~$\X$. Then, applying the function $\Centr^{1}(\L F,a)$ to $[l_{F}]$, we obtain an element $[l_{F'}]$ of~$\Centr^{1}(\L F,B)$ through the following construction.
\[
 \xymatrix{0 \ar[r] & A \ar@{{ |>}->}[r]^-{k_{F}} \ar[d]_-{\lind 1_{A},0\rind} & F_{n} \ar@{ >>}[r]^-{l_{F}} \ar[d]^-{\lind 1_{F_{n}},0\rind} & \L F \ar@{=}[d] \ar[r] & 0\\
0 \ar[r] & A\oplus B \ar@{{ |>}->}[r]^-{k_{F}\times 1_B} \ar@{ >>}[d]_-{\left\lind \begin{smallmatrix} a& 1_B\end{smallmatrix}\right\rind} & F_{n}\times B \ar@{ >>}[d] \ar@{ >>}[r] & \L F \ar@{=}[d] \ar[r] & 0 \\
0 \ar[r] & B \ar@{{ |>}->}[r] & F'_{n} \pushout \ar@{ >>}[r]_-{l_{F'}} & \L F \ar[r] & 0}
\]
(Here $A\oplus B$ is the biproduct of $A$ and $B$ in $\Ab(\X)$, which may be computed as their product $A\times B$ in $\X$, and $F_{n}'$ is the pushout of $\left\lind\begin{smallmatrix} a& 1_B\end{smallmatrix}\right\rind$ and $k_{F}\times 1_{B}$.) We define $\Centr^n(Z,a)[E]=[E']$, where~$E'$ is determined by the $n$-cubic extension $F'$ with initial object $F_n'$, with initial morphisms $f_i'=\pr_{i}\comp l_{F'}$ for~$i\in n$, and with $F'_I=F_I$ for all $I\subsetneq n$. The centrality of $F'$ is a consequence of~$F$ being central, since the extension $F'$ is a quotient of $F\times \KK(B,n-1)$, which is central as a product of central extensions (see~Example~\ref{Example-K(A,n)-as-Extension}). The functoriality of $\Centr^n(Z,-)$ is now an immediate consequence of the functoriality of $\Centr^1(\L F,-)$.

The functor $\Centr^{n}(Z,-)$ preserves terminal objects: indeed, $\Centr^{n}(Z,0)$ is a singleton, because the terminal object of $\CExt^{n}_{Z}(\X)$ has direction~$0$ by Lemma~\ref{Lemma-Product-and-Direction}; if $E$ is an $n$-fold central extension of~$Z$ by $0$, there is the unique morphism ${E\to 1}$ to testify that $[E]=[1]$. As for binary products, we must define an inverse to the map
\[
\xymatrix@=10em{\Centr^n(Z, A\times B) \ar@<-.5ex>[r]_-{\lind\Centr^n(Z,\pr_A), \Centr^n(Z,\pr_B)\rind} & \Centr^n(Z,A)\times \Centr^n(Z,B). \ar@<-.5ex>@{.>}[l]}
\]
This inverse takes a couple $([E],[E'])$ and sends it to $[E\times E']$, where the product is taken over $Z$: Lemma~\ref{Lemma-Product-and-Direction} insures that the direction of $E\times E'$ is $A\times B$, and the two morphisms are easily seen to compose to the respective identities. 

Indeed, for any couple $([E],[E'])$, the $n$-fold extension $E$ is an element of $\Centr^n(Z,\pr_A)[E\times E']$, while $E'$ is an element of $\Centr^n(Z,\pr_B)[E\times E']$. This proves that the dotted arrow is a section. On the other hand, by the universal property of pullbacks, any $n$-fold extension $H$ of $Z$ by $A\times B$ is connected to $E\times E'$ when $([E],[E'])=\lind\Centr^n(Z,\pr_A), \Centr^n(Z,\pr_B)\rind[H]$. Hence the dotted arrow is a retraction.
\end{proof}

\begin{corollary}\label{Corollary-Centr^{n}(Z,-)}
When $\X$ is a semi-abelian category, the functor $\Centr^{n}(Z,-)$ lifts uniquely over the forgetful functor ${\Ab\to\Set}$ to yield a functor
\[
\Centr^{n}(Z,-)\colon {\Ab(\X)\to \Ab}.
\]
In particular, any $\Centr^{n}(Z,A)$ carries a canonical abelian group structure.\noproof
\end{corollary}

\section{The geometry of higher central extensions}\label{Section-Geometry}
We give a geometrical interpretation of the concept of higher central extension, essentially a higher-dimensional version of Bourn and Gran's result~\cite{Bourn-Gran} that a one-cubic extension $f\colon {X\to Z}$ is central if and only if its kernel $A$ is abelian and its kernel pair $(\Eq(f),f_{0},f_{1})$ is the product $A\times X$ with $f_{0}=\pr_{X}$ and $f_{1}=\varphi_{\ker f,1_{X}}$ as in Subsection~\ref{Commutators}. Our Theorem~\ref{Theorem-Higher-Centrality} in essence says that an~$n$-cubic extension~$F$ is central if and only if
\begin{enumerate}
\item the direction of $F$ is abelian, and
\item any face in any $n$-fold diamond in $F$ is uniquely determined by an element of the direction of $F$.
\end{enumerate}
In the following sections this will lead to an equivalence between torsors and central extensions, Theorem~\ref{Theorem-Torsor-Equivalence}, which in turn will lead to our main result on cohomology, Theorem~\ref{Main-Theorem}.

\subsection{Higher equivalence relations}
Recall that a \defn{double equivalence relation} is an equivalence relation of equivalence relations: given two (internal) equivalence relations $R_{0}$ and $R_{1}$ on an object $X$, it is an equivalence relation~${R\rightrightarrows R_{1}}$ on the relation ${R_{0}\rightrightarrows X}$ as in the diagram below:
\[
\xymatrix@=3em{R \ar@<-.5ex>[d]_{\pr^{0}_0} \ar@<.5ex>[d]^{\pr^{0}_1} \ar@<-.5ex>[r]_{\pr^{1}_0} \ar@<.5ex>[r]^{\pr^{1}_1} & R_{1} \ar@<-.5ex>[d]_-{r^{1}_{0}} \ar@<.5ex>[d]^-{r^{1}_{1}}\\
R_{0} \ar@<-.5ex>[r]_-{r^{0}_{0}} \ar@<.5ex>[r]^-{r^{0}_{1}} & X.}
\]
That is, each of the four pairs of parallel morphisms on this diagram represents an equivalence relation, and these relations are compatible in an obvious sense. For instance, $R_{1} \square R_{0}$ denotes the largest double equivalence relation on~$R_{0}$ and~$R_{1}$, a two-dimensional version of $\nabla_{X}$; see~\cite{CPP, Smith, Borceux-Bourn, Bourn2003, Janelidze-Pedicchio}. It ``consists of'' all quadruples $(\alpha, \beta, \gamma, \delta)$ in $X^{4}$ in the configuration
\[
\vcenter{\xymatrix@1@!0@=3.5em{\gamma \ar@{.}[r] \ar@{.}[d]|-{1} & \beta \ar@{.}[d] \\
\delta \ar@{.}[r]|-{0} & \alpha,}}
\]
a $2\times 2$ matrix where $(\delta,\alpha)$, $(\gamma,\beta)\in R_{0}$ and $(\alpha,\beta)$, $(\delta,\gamma)\in R_{1}$. We shall be especially interested in the particular case where $R$ is induced by a two-cubic extension~$F$ as in Diagram~\eqref{Double-Extension}, as follows: $R_{0}=\R{c}$ is the kernel pair of $c$, the relation $R_{1}=\R{d}$ is the kernel pair of $d$ and $R=\R{d}\square \R{c}$. It is easily seen that then the rows and columns of the induced diagram
\begin{equation}\label{Blokske}
\vcenter{\xymatrix{\R{d}\square \R{c} \ar@<.5ex>[r]^-{p_1} \ar@<-.5ex>[r]_-{p_0} \ar@<.5ex>[d]^-{r_1} \ar@<-.5ex>[d]_-{r_0} & \R{d} \ar[r]^-{p} \ar@<.5ex>[d]^-{d_1} \ar@<-.5ex>[d]_-{d_0} & \R{g} \ar@<.5ex>[d]^-{g_{1}} \ar@<-.5ex>[d]_-{g_{0}} \\
\R{c} \ar@<.5ex>[r]^-{c_1} \ar@<-.5ex>[r]_-{c_0} \ar[d]_-{r} & X \ar[r]^-{c} \ar[d]_-{d} & C \ar[d]^-{g}\\
\R{f} \ar@<.5ex>[r]^-{f_{1}} \ar@<-.5ex>[r]_-{f_{0}} & D \ar[r]_-{f} & Z}}
\end{equation}
are exact forks, so consist of (effective) equivalence relations with their coequalisers; it is a denormalised $3\times 3$~diagram as studied in~\cite{Bourn2003}. Since the ``elements'' of~$X$ may now be viewed as arrows with a domain in $D$ and a codomain in $C$, any ``element'' of $\R{d}\square \R{c}$ corresponds to a \defn{(two-fold) diamond}~\cite{Janelidze-Pedicchio} in the two-cubic extension~$F$:
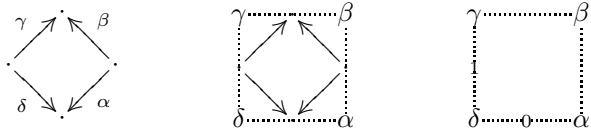
\begin{equation*}\label{Two-Diamond}
\vcenter{\xymatrix@1@!0@=2em{& {\cdot}\\
{\cdot} \ar[ru]^-{\gamma} \ar[rd]_-{\delta} && {\cdot} \ar[lu]_-{\beta} \ar[ld]^-{\alpha}\\
&{\cdot}}}\qquad\qquad
\vcenter{\xymatrix@1@!0@=2em{\gamma \ar@{.}[rr] \ar@{.}[dd] & {\cdot} & \beta \ar@{.}[dd] \\
{\cdot} \ar[ru] \ar[rd] && {\cdot} \ar[lu] \ar[ld]\\
\delta \ar@{.}[rr] & \cdot & \alpha}}\qquad\qquad
\vcenter{\xymatrix@1@!0@=4em{\gamma \ar@{.}[r] \ar@{.}[d]|-{1} & \beta \ar@{.}[d] \\
\delta \ar@{.}[r]|-{0} & \alpha}}
\end{equation*}
Note the geometrical duality here, which at this level is almost invisible since the dual of a square is a square. This will become more manifest in higher degrees. In some sense $\R{d}\square \R{c}$ is a kind of \emph{denormalised direction} of~$F$ (where the kernels are replaced by kernel pairs), also in that $\R{d}\square \R{c}$ may be considered as $\RR{F}$---see Diagram~\eqref{Blokske} and compare with Definition~\ref{Definition-Centr} for $n=2$.

\begin{figure}[b]
\resizebox{.7\textwidth}{!}{
$\xymatrix@!0@=4em{&& \bigboxvoid_{i\in 3}\R{f_{i}} \ar@<.5ex>@{.>}[ld] \ar@<-.5ex>@{.>}[ld] \ar@<.5ex>@{.>}[rrr]^-{\pr^{0}_{1}} \ar@<-.5ex>@{.>}[rrr]_-{\pr_{0}^{0}} \ar@<.5ex>@{.>}[ddd] \ar@<-.5ex>@{.>}[ddd] &&& {\R{f_{2}}\boxvoid\R{f_{1}}} \ar@<.5ex>@{.>}[ld]^(.6){\pr^{1}_{1}} \ar@<-.5ex>@{.>}[ld]_(.6){\pr_{0}^{1}} \ar@<.5ex>@{.>}[ddd]^-{\pr^{2}_{1}} \ar@<-.5ex>@{.>}[ddd]_-{\pr^{2}_{0}} \ar@{.{ >>}}[rrr] &&& \cdot \ar@<.5ex>@{.>}[ld] \ar@<-.5ex>@{.>}[ld] \ar@<.5ex>@{.>}[ddd] \ar@<-.5ex>@{.>}[ddd] \\
&{\cdot} \ar@{.{ >>}}[ld] \ar@<.5ex>@{.>}[rrr] \ar@<-.5ex>@{.>}[rrr] \ar@<.5ex>@{.>}[ddd] \ar@<-.5ex>@{.>}[ddd] &&& {\R{f_{2}}} \ar@{.{ >>}}[ld] \ar@<.5ex>@{.>}[ddd]^-{\pr^{2}_{1}} \ar@<-.5ex>@{.>}[ddd]_-{\pr^{2}_{0}} \ar@{.{ >>}}[rrr] &&& \cdot \ar@{.{ >>}}[ld] \ar@<.5ex>@{.>}[ddd] \ar@<-.5ex>@{.>}[ddd] \\
{\cdot} \ar@<-.5ex>@{.>}[rrr] \ar@<.5ex>@{.>}[rrr] \ar@<.5ex>@{.>}[ddd] \ar@<-.5ex>@{.>}[ddd] &&& \cdot \ar@<.5ex>@{.>}[ddd] \ar@<-.5ex>@{.>}[ddd] \ar@{.{ >>}}[rrr] &&& \cdot \ar@<.5ex>@{.>}[ddd] \ar@<-.5ex>@{.>}[ddd] \\
&&{\cdot} \ar@<.5ex>@{.>}[ld] \ar@<-.5ex>@{.>}[ld] \ar@<.5ex>@{.>}[rrr] \ar@<-.5ex>@{.>}[rrr] \ar@{.{ >>}}[ddd] &&& {\R{f_{1}}} \ar@<-.5ex>@{.>}[ld] \ar@<.5ex>@{.>}[ld] \ar@{.{ >>}}[ddd] \ar@{.{ >>}}[rrr] &&& \cdot \ar@<.5ex>@{.>}[ld] \ar@<-.5ex>@{.>}[ld] \ar@{.{ >>}}[ddd]\\
&{\cdot} \ar@{.{ >>}}[ld] \ar@<.5ex>@{.>}[rrr] \ar@<-.5ex>@{.>}[rrr] \ar@{.{ >>}}[ddd] &&& F_{3} \ar@{-{ >>}}[ld]|-{f_{1}} \ar@{-{ >>}}[ddd]|(.33){\hole}|-{f_{2}} \ar@{-{ >>}}[rrr]|-{f_{0}} &&& \cdot \ar@{-{ >>}}[ld] \ar@{-{ >>}}[ddd]\\
{\cdot} \ar@<.5ex>@{.>}[rrr] \ar@<-.5ex>@{.>}[rrr] \ar@{.{ >>}}[ddd] &&& \cdot \ar@{-{ >>}}[ddd] \ar@{-{ >>}}[rrr] &&& \cdot \ar@{-{ >>}}[ddd]\\
&&{\cdot} \ar@<.5ex>@{.>}[ld] \ar@<-.5ex>@{.>}[ld] \ar@<.5ex>@{.>}[rrr] \ar@<-.5ex>@{.>}[rrr] &&& \cdot \ar@<.5ex>@{.>}[ld] \ar@<-.5ex>@{.>}[ld] \ar@{.{ >>}}[rrr] &&& \cdot \ar@<.5ex>@{.>}[ld] \ar@<-.5ex>@{.>}[ld] \\
&{\cdot} \ar@{.{ >>}}[ld] \ar@<.5ex>@{.>}[rrr] \ar@<-.5ex>@{.>}[rrr] &&& \cdot \ar@{-{ >>}}[ld] \ar@{-{ >>}}[rrr]|(.66){\hole} &&& \cdot \ar@{-{ >>}}[ld]\\
{\cdot} \ar@<.5ex>@{.>}[rrr] \ar@<-.5ex>@{.>}[rrr] &&& \cdot \ar@{-{ >>}}[rrr] &&& Z}$}
\caption{$\bigboxvoid_{i\in 3}\R{f_{i}}$ for a three-cubic extension $F$}\label{Figure-Blokske}
\end{figure}

Inductively, an \defn{$n$-fold equivalence relation} may be defined as an equivalence relation of $(n-1)$-fold equivalence relations. Considered as a diagram in the base category~$\X$, it has $n$ underlying equivalence relations $R_{0}$, \dots, $R_{n-1}$ on a common object~$X$. An internal $n$-fold equivalence relation is the same thing as an internal $n$"~fold groupoid ($n$-cat-group in the case of groups~\cite{Loday}; double categories appear in~\cite{Benabou:Bicategories, Ehresmann}, for instance) in which all pairs of projections are jointly monomorphic. The \emph{largest} $n$-fold equivalence relation on $n$ given equivalence relations $R_{0}$, \dots,~$R_{n-1}$ on an object $X$---meaning that it contains all $n$-fold equivalence relations on those relations---is denoted
\begin{equation}\label{blokske Ri}
\bigboxvoid_{i\in n}R_{i}.
\end{equation}
It has projections $\pr^{i}_{0}$ and $\pr^{i}_{1}$ to $R_{i}$, for all $i\in n$, and thus consists of $2^{n}$ commutative cubes of projections, one for each choice of projection (either~$\pr^{i}_{0}$ or $\pr^{i}_{1}$) in each direction $i\in n$. This largest $n$-fold equivalence relation on $R_{0}$, \dots,~$R_{n-1}$ does indeed exist; the elements of $\bigboxvoid_{i\in n}R_{i}$ are $n$-dimensional matrices in $X$, in fact matrices of order
\[
\underbrace{2\times \cdots \times 2}_{n}.
\]
In the $i$-th direction of the matrix (counting from $0$ to $n-1$) the elements are related by the equivalence relation~$R_{i}$. In Subsection~\ref{constructing blokske} we give a two-step formal construction. 

In practice the $n$-fold equivalence relation will be induced by an $n$-cubic extension~$F$, by taking $R_{i}=\R{f_{i}}$. The induced object ${\bigboxvoid_{i\in n}\R{f_{i}}=\Rn{F}}$ may then be considered as a denormalised direction of $F$. Its elements are called \defn{\mbox{($n$-fold)} diamonds in $F$} because of their shape in the lower dimensions.

When $F$ is a three-cubic extension (see Figure~\ref{Figure-Blokske}) such a diamond is a hollow octahedron (see Figure~\ref{Figure-Diamond}) of which the faces are elements of $F_{3}$. We name the faces of the octahedron by the vertices of a cube which is formally a three-dimensional matrix where $\R{f_{0}}$ is the left-right relation, $\R{f_{1}}$ is bottom-top and $\R{f_{2}}$ is front-back. Note how the geometrical duality between the octahedron and the cube is explicit here.

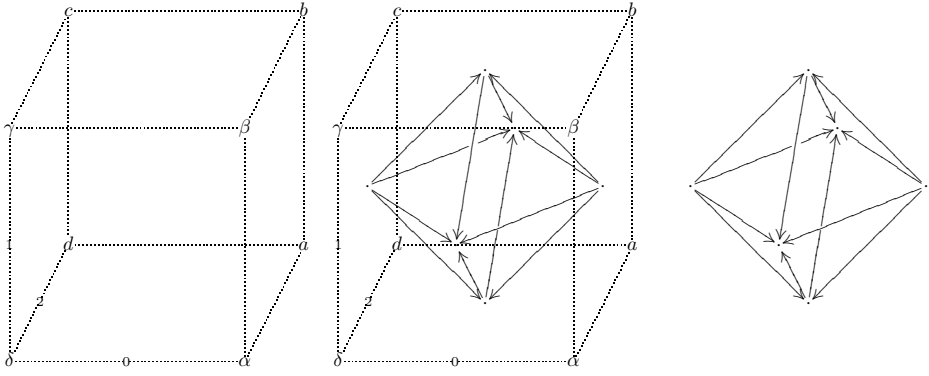
\begin{figure}
\resizebox{\textwidth}{!}{
$\xymatrix@1@!0@=3em{& c \ar@{.}[rrrr] \ar@{.}[dddd] \ar@{.}[ddl] &&&& b \ar@{.}[dddd] \ar@{.}[ddl]\\\\
\gamma \ar@{.}[rrrr] \ar@{.}[dddd]|-{1} &&&& \beta \ar@{.}[dddd] \\\\
& d \ar@{.}[rrrr] \ar@{.}[ddl]|-{2} &&&& a \ar@{.}[ddl] \\\\
\delta \ar@{.}[rrrr]|-{0} &&&& \alpha}
\quad
\xymatrix@1@!0@C=1.5em@R=3em{&& c \ar@{.}[rrrrrrrr] \ar@{.}[dddd] \ar@{.}[ddll] &&&&&&&& b \ar@{.}[dddd] \ar@{.}[ddll]\\
&&&&& {\cdot} \ar[rd] \ar[dddl] \\
\gamma \ar@{.}[rrrrrrrr] \ar@{.}[dddd]|-{1} &&&&&& {\cdot} && \beta \ar@{.}[dddd] \\
& {\cdot} \ar[urrrrr]|(.72){\hole\hole\hole} \ar[uurrrr] \ar[ddrrrr] \ar[drrr] &&&&&&&& {\cdot} \ar[ulll] \ar[dlllll] \ar[uullll] \ar[ddllll] \\
&& d \ar@{.}[rrrrrrrr] \ar@{.}[ddll]|-{2} && {\cdot} &&&&&& a \ar@{.}[ddll] \\
&&&&& {\cdot} \ar[ul] \ar[uuur]|(.43){\hole\hole} \\
\delta \ar@{.}[rrrrrrrr]|-{0} &&&&&&&& \alpha}
\quad
\xymatrix@1@!0@C=1.5em@R=3em{&&&&&&&&&&\\
&&&&& {\cdot} \ar[rd] \ar[dddl] \\
&&&&&& {\cdot} &&\\
& {\cdot} \ar[urrrrr]|(.72){\hole\hole\hole} \ar[uurrrr] \ar[ddrrrr] \ar[drrr] &&&&&&&& {\cdot} \ar[ulll] \ar[dlllll] \ar[uullll] \ar[ddllll] \\
&&&& {\cdot} &&&&&&\\
&&&&& {\cdot} \ar[ul] \ar[uuur]|(.43){\hole\hole}\\
&&&&&&&&}$}
\caption{Matrix and diamond for a three-cubic extension}\label{Figure-Diamond}
\end{figure}

\subsection{Formal construction of $\bigboxvoid_{i\in n}R_{i}$}\label{constructing blokske}
\emph{Starting with an $n$-fold arrow.} Given an $n$-fold arrow $F$, first consider it as an arrow $\dom F\to \cod F$ between $(n-1)$-fold arrows, and then take its kernel pair ${\Eq(F)\rightrightarrows \dom F}$. By Corollary~3.10 in~\cite{EGVdL}, both projections are $n$-cubic extensions in $\X$ if such is $F$. Then consider those $(n-1)$-fold arrows as (vertical) arrows between $(n-2)$-fold arrows and take kernel pairs, obtaining a double equivalence relation
\[
\xymatrix{\RR{F} \ar@<.5ex>[r] \ar@<-.5ex>[r] \ar@<.5ex>[d] \ar@<-.5ex>[d] &  \R{\dom F} \ar@<.5ex>[d] \ar@<-.5ex>[d] \\
\dom \R{F} \ar@<.5ex>[r] \ar@<-.5ex>[r] & \dom^{2}F}
\]
of $(n-2)$-fold arrows. All commutative squares in it are $n$-cubic extensions in~$\X$ if such is $F$, again by~\cite[Corollary~3.10]{EGVdL}. Repeat the process until an $n$-fold equivalence relation in~$\X$ is obtained. The object $\Rn{F}$ is $\bigboxvoid_{i\in n}\R{f_{i}}$.

\emph{Starting with equivalence relations $(R_{i})_{i\in n}$.} Take coequalisers $f_{i}$ so that each $R_{i}$ is $\R{f_{i}}$. Take pushouts of the $f_{i}$ along each other and along their pushouts until an $n$-fold arrow $F$ is obtained. Now we can apply the above construction to obtain $\bigboxvoid_{i\in n}R_{i}=\bigboxvoid_{i\in n}\R{f_{i}}$.

\subsection{Indexing the elements of $\bigboxvoid_{i\in n}\R{f_{i}}$}\label{indexing}
Consider an $n$-cubic extension~$F$. An element of $\bigboxvoid_{i\in n}\R{f_{i}}$ being an $n$-dimensional matrix, its entries are indexed by the elements of $2^{n}$, the subsets of the ordinal $n$. An entry $x_{I}$ in a matrix ${x\in\bigboxvoid_{i\in n}\R{f_{i}}}$ finds itself in the first entry of the $i$-th direction when~${i\not \in I}$ and in the second entry of the~$i$-th direction when $i\in I$. Hence the entry~$x_{I}=\pr_{I}(x)$ is
\begin{equation}\label{diagonal}
(\pr^{0}_{\delta_{I}(0)}\comp\pr^{1}_{\delta_{I}(1)}\comp\cdots\comp\pr^{n-1}_{\delta_{I}(n-1)})(x)
\end{equation}
where
\[
\delta_{I}(i)=\begin{cases} 0 & \text{if $i\not\in I$}\\
1 & \text{if $i\in I$}
\end{cases}
\]
and $\pr^{i}_{0}$ and $\pr^{i}_{1}$ are the first and second projection of $\R{f_{i}}$, extended to morphisms
\[
\bigboxvoid_{j\in k}\R{f_{j}}\to \bigboxvoid_{j\in k\setminus \{i\}}\R{f_{j}}
\]
for all $i<k\leq n$ (see Figure~\ref{Figure-Blokske}). Two entries~$x_{I}$ and~$x_{J}$ are related by $\R{f_{i}}$ when the only difference between~$I$ and $J$ is that one does, and the other does not, contain $i$. So $(x_{I},x_{J})\in \R{f_{i}}$ when $J=I\cup\{i\}$ or $I=J\cup\{i\}$.

For instance, in Figure~\ref{Figure-Diamond}, the face $\beta$ corresponds to the entry $x_2$: the set~${2\subseteq 3}$ contains $0$ and $1$ but it doesn't contain $2$.

\subsection{The induced $n$-cubes}\label{Chosen-cubes}
As explained in the paragraph following~\eqref{blokske Ri} above, given an $n$-cubic extension~$F$, any choice of a set $I\subseteq n$ corresponds to one of the commutative $n$-cubes in the $n$-fold equivalence relation $\bigboxvoid_{i\in n}\R{f_{i}}$, namely the cube whose diagonal~\eqref{diagonal} ``picks'' the $I$'th entry of any given $n$-fold diamond. We shall denote it $\kuub(F,I)$. In fact, it again forms an $n$-cubic extension in $\X$, and its initial morphisms are the
\[
\pr^{i}_{\delta_{I}(i)}\colon\bigboxvoid_{j\in n}\R{f_{j}}\to \bigboxvoid_{j\in n\setminus \{i\}}\R{f_{j}}.
\]
The property of being an extension follows, for instance, from the fact that all its morphisms are compatibly split (by the reflexivity of all the equivalence relations involved). Since no confusion with the other arrows is possible (compare with the notation introduced in Subsection~\ref{HDA}), we shall denote such a ``composed splitting'' or ``composed projection''
\[
\kuub(F,I)^{J}_{K}\colon{\kuub(F,I)_{J}\to \kuub(F,I)_{K}}
\]
when $J\subseteq K\in 2^{n}$ or $K\subseteq J\in 2^{n}$, respectively.

\subsection{The objects $\bigboxdot_{i\in n}^{I}\R{f_{i}}$}
Given an $n$-cubic extension $F$ and $I\subseteq n$, the elements of the object $\bigboxdot_{i\in n}^{I}\R{f_{i}}$ are diamonds in $F$ with the $I$-face missing, or equivalently, $n$-dimensional matrices (of order $2\times\cdots\times 2$) with the $I$-entry left out; it is the limit~$\L(\kuub(F,n\setminus I))$ from Subsection~\ref{HDA} determined by the $n$-cubic extension~$\kuub(F,n\setminus I)$. Indeed, removing an $I$'th vertex from a cube is the same thing as considering only those faces which contain the complementary $(n\setminus I)$'th vertex. Let
\[
\pi^{I}=l_{\boxvoid(F,n\setminus I)}\colon \bigboxvoid_{i\in n}\R{f_{i}}\to \bigboxdot_{i\in n}^{I}\R{f_{i}}
\]
denote the canonical projection which forgets the $I$-face, then clearly the kernel of~$\pi^{I}$ is isomorphic to the direction of $F$. (If a diamond is in the kernel of $\pi^{I}$ then all faces but one in this diamond are zero, and of course this face has its boundary zero.) In fact, this gives us a version of Lemma~\ref{Lemma-DeltaLambda-Square}, valid for higher extensions:

\begin{lemma}\label{Lemma-Diamond-Pullback}
For any $n$-cubic extension $F$, the square
\begin{equation*}\label{Diamond-Pullback}
\vcenter{\xymatrix{\bigboxvoid_{i\in n}\R{f_{i}} \ar[d]_{\pr_{I}} \ar[r]^-{\pi^{I}} & \bigboxdot_{i\in n}^{I}\R{f_{i}} \ar[d] \\
F_{n} \ar[r]_-{\lind f_{i}\rind_{i}} & \L F}}
\end{equation*}
is a pullback.
\end{lemma}
\begin{proof}
This follows from Lemma~\ref{Lemma-Iso-Pullback} since $\K{\lind f_{i}\rind_{i}} = \D_{(n,Z)} F=\K{\pi^{I}}$ as explained above.
\end{proof}

For instance, in $\bigboxdot_{i\in 3}^{2}\R{f_{i}}$ we have $3$-fold diamonds as in Figure~\ref{Figure-Diamond} in which the face $\beta=x_{2}$ is missing.

In degree two the pullback $\R{d}\times_{X} \R{c}$ (mentioned in the introduction, and computed as in Diagram~\eqref{Pullback-Smith}) that contains two-fold diamonds in which the face~$\delta$ is missing, is nothing but~${\R{d}\squaredot^{\varnothing} \R{c}}$, and the projection $\pi$ is $\pi^{\varnothing}$.

\subsection{Analysis of centrality in degree two}\label{Degree two}
As explained in~\cite{RVdL} (and, in full generality, in the proof of Theorem~\ref{Theorem-Higher-Centrality} below), the two-cubic extension~$F$ from Diagram~\eqref{Double-Extension} is central if and only if in the diagram
\[
\xymatrix{& \lb\R{d}\square \R{c}\rb \ar@{ >>}[r]^-{\lb\pi\rb} \ar@{{ |>}->}[d] & \lb\R{d}\times_{X} \R{c}\rb \ar@{{ |>}->}[d] \\
A \ar@{{ |>}->}[r] \ar@{ >>}[d] & \R{d}\square \R{c} \ar@{}[rd]|-{\texttt{(i)}} \ar@<-.5ex>@{ >>}[r]_-{\pi} \ar@{ >>}[d] & \R{d}\times_{X} \R{c} \ar@{.>}@<-.5ex>[l]_-{\iota} \ar@{ >>}[d]\\
\K{\ab\pi} \ar@{{ |>}->}[r] & \ab(\R{d}\square \R{c}) \ar@{ >>}@<-.5ex>[r]_-{\ab\pi} & \ab(\R{d}\times_{X} \R{c}) \ar@<-.5ex>[l]_-{\nu}}
\]
the morphism $\lb\pi\rb$ is an isomorphism. (Recall the bracket notation from~\eqref{Angular Bracket}.) By Lemma~\ref{Lemma-Iso-Pullback}, this occurs when the square $\texttt{(i)}$ is a pullback, which is precisely saying that~$\pi$ is a one-cubic trivial extension. (Indeed~$\pi$ is a one-cubic extension, because it is the comparison to the pullback in a two-cubic extension, in fact in a double split epimorphism.)

Note that $\ab\pi$ is a split epimorphism by Lemma~\ref{Lemma-Naturally-Maltsev}, because $\ab$ preserves the pullback $\R{d}\times_{X} \R{c}$: in fact, it preserves all pullbacks of split epimorphisms along split epimorphisms, or even all pullbacks of split epimorphisms along cubic extensions (Lemma~\ref{Lemma-Pullbacks}). This makes $\ab \pi$ a product projection. Further recall that the kernel of $\pi$ is the direction of $F$. Finally, note that the splitting $\nu$ commutes with the sections in the double equivalence relation $\R{d}\square \R{c}$ and the ones induced to $\R{d}\times_{X} \R{c}$. In fact, it is uniquely determined by this property (Lemma~\ref{Lemma-Naturally-Maltsev}). 

Hence if $F$ is central then $\pi$ is a split epimorphism, in fact a product projection (since product projections are stable under pullbacks), and thus we see that
\begin{equation}\label{Iso-Dimension-Two'}
\R{d}\square \R{c}\cong A\times (\R{d}\times_{X} \R{c})
\end{equation}
where $A$ is the direction of $F$, an abelian object. Conversely, whenever $A$ is abelian and $\pi$ is the projection in the product~\eqref{Iso-Dimension-Two'}, the extension $\pi$ is trivial, so that the square $\texttt{(i)}$ is a pullback, and $F$ is a two-cubic central extension.

The inclusion $\iota$ of $\R{d}\times_{X} \R{c}$ into $\R{d}\square \R{c}$ is \emph{compatible with degeneracies} in the following sense. The conditions which determine $\nu$ uniquely in Lemma~\ref{Lemma-Naturally-Maltsev} extend to the splitting $\iota$ of $\pi$, so that $\iota$ maps
\[
\vcenter{\xymatrix@1@!0@=2em{& {\cdot}\\
{\cdot} \ar[ru]^-{\gamma} && {\cdot} \ar[lu]_-{\beta} \ar[ld]^-{\beta}\\
&{\cdot}}}\qquad
\text{to}\qquad
\vcenter{\xymatrix@1@!0@=2em{& {\cdot}\\
{\cdot} \ar[ru]^-{\gamma} \ar[rd]_-{\gamma} && {\cdot} \ar[lu]_-{\beta} \ar[ld]^-{\beta}\\
&{\cdot}}}
\qquad\text{and}\qquad
\vcenter{\xymatrix@1@!0@=2em{& {\cdot}\\
{\cdot} \ar[ru]^-{\beta} && {\cdot} \ar[lu]_-{\beta} \ar[ld]^-{\alpha}\\
&{\cdot}}}\qquad\text{to}\qquad
\vcenter{\xymatrix@1@!0@=2em{& {\cdot}\\
{\cdot} \ar[ru]^-{\beta} \ar[rd]_-{\alpha} && {\cdot} \ar[lu]_-{\beta} \ar[ld]^-{\alpha}\\
&{\cdot}}}.
\]
Note how the conditions satisfied by a Mal'tsev operation appear here.

We may also view this slightly differently: by Lemma~3.3 in~\cite{Bourn-Gran-Maltsev}, the condition $[\R{d},\R{c}]=\Delta_{X}$ in~\eqref{Double-Central} is equivalent to the morphism
\[
\pi\colon{\R{d}\square \R{c}\to \R{d}\times_{X} \R{c}}
\]
being a split epimorphism, compatible with certain splittings as in Lemma~\ref{Lemma-Naturally-Maltsev}. Also $\Delta_{X}=[\R{d}\cap \R{c},\nabla_{X}]$ if and only if~$\pi$ is central~\cite{Gran-Alg-Cent}. Now a split epimorphism is a one-cubic central extension if and only if it is a one-cubic trivial extension, so~$\pi$ is trivial---the square~$\texttt{(i)}$ is a pullback---when the commutators $[\R{d},\R{c}]$ and $[\R{d}\cap \R{c},\nabla_{X}]$ vanish.

\subsection{Higher degrees}
This characterisation of centrality goes up to higher dimensions. The basic idea is to show by induction that an $n$-cubic extension~$F$ is central if and only if the morphisms
\[
\lb \pi^{I}\rb\colon \lb \bigboxvoid_{i\in n}\R{f_{i}}\rb \to \lb \bigboxdot_{i\in n}^{I}\R{f_{i}}\rb
\]
are isomorphisms. As we shall see, this then amounts to an isomorphism 
\begin{equation}\label{Iso-General-n}
\bigboxvoid_{i\in n}\R{f_{i}}\cong A\times \bigboxdot_{i\in n}^{I}\R{f_{i}}
\end{equation}
where $A$ is the direction of $F$: any missing face in an $n$-fold diamond is completely determined by an element in $A$.

\begin{lemma}\label{Lemma-ab-of-Square}
When $\X$ is a semi-abelian category, the functor $\ab\colon{\X\to \Ab(\X)}$ preserves any limit $\bigboxdot_{i\in n}^{I}\R{f_{i}}$ induced by any $n$-cubic extension $F$ and $I\subseteq n$. Furthermore, the comparison morphism 
\[
\ab\pi^{I}\colon\ab(\bigboxvoid_{i\in n}\R{f_{i}}) \to \bigboxdot_{i\in n}^{I}\ab(\R{f_{i}})
\]
admits a splitting. This splitting is uniquely determined by the property that it commutes with the sections in the $n$-cubic extension~$\ab(\kuub(F,n\setminus I))$ and the induced sections to the limit object $\bigboxdot_{i\in n}^{I}\ab(\R{f_{i}})$.
\end{lemma}
\begin{proof}
The first part follows from Lemma~\ref{Lemma-Pullbacks}, because by Lemma~\ref{Lemma-L-pullbacks}, the limit $\L(\kuub(F,n\setminus I))$ may be computed by repeated pullbacks of regular epimorphisms along split epimorphisms. 

Indeed, all arrows in the cube $\kuub(F,n\setminus I)$ are (compatibly) split, so that the pullback in the statement of Lemma~\ref{Lemma-L-pullbacks} is a pullback of a split epimorphism along a split epimorphism. Now Lemma~\ref{Lemma-L-pullbacks} is applied to the induced $(n-1)$-cubic extension $G$, in which all morphisms are (compatibly) split, except in the direction~$\lind d,c\rind$. We now take a pullback of an extension along a split epimorphism. This procedure is repeated until nothing but a pullback of a regular epimorphism along a split epimorphism is left. Again, by Lemma~\ref{Lemma-Pullbacks} all those pullbacks are preserved.

In the abelian category $\Ab(\Arr^{n-1}(\X))$, the reflection $\lind \ab d,\ab c\rind$ of $\lind d,c\rind$ is a split epimorphism by Lemma~\ref{Lemma-Naturally-Maltsev}. There is actually a unique splitting, compatible with the sections in the square~\eqref{Double-Extension}, coming from the fact that the $\R{f_{i}}$ are equivalence relations. Also at each further stage of the above proof we may now apply Lemma~\ref{Lemma-Naturally-Maltsev}, so that at every stage a pullback of a split epimorphism along a split epimorphism is taken, and eventually the needed morphism $\nu^{I}$ is obtained. The requirement that at each stage the chosen splitting commutes with the given sections determines $\nu^{I}$ uniquely.
\end{proof}

Recall from Subsection~\ref{Subsection-Tower} the notation $\lb F\rb^{n}$ for the initial object of the kernel of the unit of the centralisation of an $n$-cubic extension $F$, which vanishes if and only if $F$ is central. It determines a functor $\lb-\rb^{n}\colon {\Ext^{n}(\X)\to \X}$. In particular, $\lb-\rb=\lb-\rb^{0}\colon \X\to \X$ is the kernel of the unit of the abelianisation functor as in~\eqref{Angular Bracket}.

\pagebreak
\begin{theorem}\label{Theorem-Higher-Centrality}
In a semi-abelian category, let $F$ be an $n$-cubic extension with direction $A$. Then the following are equivalent:
\begin{enumerate}
\item $F$ is central;
\item the $n$-cubic extension $\lb\kuub(F,I)\rb$ is a limit $n$-cube;
\item the morphism $\lb \pi^{I}\rb\colon \lb \bigboxvoid_{i\in n}\R{f_{i}}\rb \to \lb \bigboxdot_{i\in n}^{I}\R{f_{i}}\rb$ is an isomorphism;
\item $A$ is abelian and $\bigboxvoid_{i\in n}\R{f_{i}}\cong A\times \bigboxdot_{i\in n}^{I}\R{f_{i}}$:
we have a short exact sequence
\[
\xymatrix{0 \ar[r] & A \ar@{{ |>}->}[r] & \bigboxvoid_{i\in n}\R{f_{i}} \ar@{ >>}[r]^-{\pi^{I}} & \bigboxdot_{i\in n}^{I}\R{f_{i}} \ar[r] & 0}
\]
where $\pi^{I}$ is a product projection;
\end{enumerate}
for any, hence for all,~${I\subseteq n}$. Furthermore, when these conditions are satisfied, there is a unique splitting 
\[
\iota^{I}\colon \bigboxdot_{i\in n}^{I}\R{f_{i}}\to \bigboxvoid_{i\in n}\R{f_{i}}
\]
of $\pi^{I}$ which commutes with the given sections in the $n$-cubic extension $\kuub(F,n\setminus I)$ and the induced sections to the object $\bigboxdot_{i\in n}^{I}\R{f_{i}}$.
\end{theorem}
\begin{proof}
First we show that (i) and (ii) are equivalent. The $n$-cubic extension $F$, considered as a morphism ${\dom F\to \cod F}$, is central if and only if either one of the projections ${\R{F}\to \dom F}$ is trivial. This, by definition of trivial extensions, occurs when the morphisms
\[
\xymatrix{\lb \R{F}\rb^{n-1} \ar@<.5ex>[r] \ar@<-.5ex>[r] & \lb \dom F\rb^{n-1}}
\]
are isomorphisms (see Subsection~\ref{Subsection-Tower} for more details). By Lemma~\ref{Lemma-Iso-Pullback} this happens when either one of the commutative squares in
\[
\xymatrix{\lb \RR{F}\rb^{n-2} \ar@<.5ex>[r] \ar@<-.5ex>[r] \ar@<.5ex>[d] \ar@<-.5ex>[d] & \lb \R{\dom F}\rb^{n-2} \ar@<.5ex>[d] \ar@<-.5ex>[d] \\
\lb \dom \R{F}\rb^{n-2} \ar@<.5ex>[r] \ar@<-.5ex>[r] & \lb \dom^{2}F\rb^{n-2}}
\]
is a pullback. This, in turn, is equivalent to either one of the commutative cubes in
\[
\xymatrix@!0@R=4em@C=6em{& \lb \Rthree{F}\rb^{n-2} \ar@{.>}@<.5ex>[dd] \ar@{.>}@<-.5ex>[dd] \ar@<.5ex>[rr] \ar@<-.5ex>[rr] \ar@<.5ex>[dl] \ar@<-.5ex>[dl] && \lb \RR{\dom F}\rb^{n-2} \ar@<.5ex>[dd] \ar@<-.5ex>[dd] \ar@<.5ex>[dl] \ar@<-.5ex>[dl] \\
\lb \R{\dom \R{F}}\rb^{n-2} \ar@<.5ex>[dd] \ar@<-.5ex>[dd] \ar@<.5ex>[rr] \ar@<-.5ex>[rr] && \lb \R{\dom^{2}F}\rb^{n-2} \ar@<.5ex>[dd] \ar@<-.5ex>[dd]\\
& \lb \dom \RR{F}\rb^{n-2} \ar@{.>}@<.5ex>[rr] \ar@{.>}@<-.5ex>[rr] \ar@{.>}@<.5ex>[dl] \ar@{.>}@<-.5ex>[dl] && \lb \dom \R{\dom F}\rb^{n-2} \ar@<.5ex>[dl] \ar@<-.5ex>[dl] \\
\lb \dom^{2} \R{F}\rb^{n-2} \ar@<.5ex>[rr] \ar@<-.5ex>[rr] && \lb \dom^{3}F\rb^{n-2}}
\]
being a limit cube. This process continues until we obtain a cube of dimension~$n$ whose vertices are brackets $\lb-\rb$ and whose edges are parallel pairs of arrows as in the diagrams above. This cube is precisely the $n$-fold equivalence relation $\lb \bigboxvoid_{i\in n}\R{f_{i}}\rb$, considered as a diagram in $\X$---compare with the construction in Subsection~\ref{constructing blokske}. As in Subsection~\ref{Chosen-cubes}, a choice of~${I\subseteq n}$ picks one of any two parallel arrows in this diagram in such a way that we obtain the $n$-cubic extension~$\lb\kuub(F,I)\rb$.

The equivalence between (ii) and (iii) is clear because $\lb\kuub(F,I)\rb$ is nothing but one of the cubes induced by choosing an $n$-fold arrow (making a choice of projections) in the $n$-fold equivalence relation $\lb \bigboxvoid_{i\in n}\R{f_{i}}\rb$ as in Subsection~\ref{Chosen-cubes}; so $\lb \pi^{I}\rb$ is an isomorphism if and only if this cube is a limit. The functor $\lb-\rb$ does indeed preserve the limit~$\bigboxdot_{i\in n}^{I}\R{f_{i}}$, since so does $\ab$ by Lemma~\ref{Lemma-ab-of-Square}.

Now we prove the equivalence between (iii) and (iv). Condition (iii) is equivalent to the square
\begin{equation}\label{Square-Blokskes}
\vcenter{\xymatrix{\bigboxvoid_{i\in n}\R{f_{i}} \ar@{ >>}[r]^-{\pi^{I}} \ar@{ >>}[d] & \bigboxdot^{I}_{i\in n}\R{f_{i}} \ar@{ >>}[d]\\
\ab(\bigboxvoid_{i\in n}\R{f_{i}}) \ar@{ >>}[r]_-{\ab\pi^{I}} & \ab(\bigboxdot^{I}_{i\in n}\R{f_{i}})}}
\end{equation}
being a pullback, which means that $\pi^{I}$ is a one-cubic trivial extension. Since its kernel is the abelian object $A$, the extension $\pi^{I}$ is a product projection if and only if it is a split epimorphism. By Lemma~\ref{Lemma-ab-of-Square}, the latter condition does indeed hold.

The final statement is again a consequence of Lemma~\ref{Lemma-ab-of-Square}: the needed morphism~$\iota^{I}$ is induced by the pullback \eqref{Square-Blokskes} and the splitting $\nu^{I}$ of $\ab\pi^{I}$ given by the lemma.
\end{proof}

In what follows we shall use this result to obtain one half of the equivalence between torsors and central extensions.

\begin{remark}\label{Remark-Full-3}
Note that the splitting $\iota^{I}$ of $\pi^{I}$ constructed in the proof above is natural in $F$, so that also the product decompositions (iv) are natural in the extension considered.
\end{remark}

\begin{remark}\label{Remark-Tomas-Centrality}
The proof of Theorem~\ref{Theorem-Higher-Centrality} shows that an $n$-cubic extension~$F$ is central precisely when, for any $I\subseteq n$, the induced $(n+1)$-cubic extension
\[
{\kuub(F,I)\to \ab(\kuub(F,I))}
\]
is a limit $(n+1)$-cube. In fact, these $(n+1)$-cubic extensions are part of the regular epimorphism of $n$-fold groupoids
\[
\eta_{\bigboxvoid_{i\in n}\R{f_{i}}}\colon{\bigboxvoid_{i\in n}\R{f_{i}}\to \ab(\bigboxvoid_{i\in n}\R{f_{i}})},
\]
which therefore is a discrete fibration if and only if $F$ is central. (The concept of \defn{discrete fibration} between higher-dimensional internal groupoids is the obvious extension of the one-fold groupoid case: any of its induced $n$-fold arrows must be a pullback. In the situation at hand this gives precisely the condition on the $(n+1)$-cubes ${\kuub(F,I)\to \ab(\kuub(F,I))}$ mentioned above.) In the article~\cite{Everaert-Gran-nGroupoids}, the authors study the Galois structure for $n$-fold groupoids in a semi-abelian category ($n$-cat-groups in $\Gp$, for instance~\cite{Loday}) induced by the reflection
\[
\vcenter{\xymatrix{{\Gpd^{n}(\X)} \ar@<1ex>[r]^-{\Pi^{n}_{0}} \ar@{}[r]|-{\perp} & \Dis^{n}(\X)\simeq \X \ar@<1ex>[l]^-{\supset}}}
\]
to $\X$ via the ``connected components'' functor to discrete $n$-fold groupoids. It turns out~\cite[Proposition~2.9]{Everaert-Gran-nGroupoids} that the central extensions with respect to this reflection are again the regular epimorphisms of internal $n$-fold groupoids which are discrete fibrations. Hence an $n$-cubic extension $F$ in $\X$ is central relative to~$\Ab(\X)$ if and only if the induced extension of $n$-fold groupoids~$\eta_{\bigboxvoid_{i\in n}\R{f_{i}}}$ is central relative to~$\X$.
\end{remark}

\subsection{Higher central extensions as higher-dimensional pregroupoids}\label{Maltsev}
The isomorphisms~\eqref{Iso-General-n} determine ``multiplications'' or ``compositions'' of $(n-1)$-di\-men\-sion\-al hyper-tetrahedra (or $n$-di\-men\-sion\-al hyper-triangles) in an $n$-cubic central extension, in the sense that any aggregation of hyper-tetrahedra in the shape of an $n$-fold diamond with a face missing ``composes'' to the missing face. That is to say, the composite morphism
\[
\xymatrix{p^{I}\colon\bigboxdot_{i\in n}^{I}\R{f_{i}} \ar[r]^-{\lind 0, 1\rind} & A\times \bigboxdot_{i\in n}^{I}\R{f_{i}} \ar[r]^-{\cong} & \bigboxvoid_{i\in n}\R{f_{i}} \ar[r]^-{\pr_{I}} & F_{n}}
\]
acts as a \emph{higher-dimensional Mal'tsev operation} or, more precisely, as a \emph{higher-dimensional pregroupoid structure} on $F$. Indeed, Proposition~\ref{Centrality-Sum} below implies that $p^{I}$ satisfies certain conditions which we could call \emph{higher-di\-men\-sion\-al Mal'tsev laws}. In higher degrees those algebraic properties of the $p^{I}$ still have to be further studied---for instance, what about associativity?---but we may already give a few examples.

In the two-dimensional case, $\delta=p^{\emptyset}(\alpha,\beta,\gamma)$ is the unique choice of~$\delta$ such that the projection $a=\pr_{A}(\alpha,\beta,\gamma,\delta)$ of the diamond $(\alpha,\beta,\gamma,\delta)$ on the direction~$A$ is zero. In this case we may think of $\delta$ as a composite $\gamma\beta^{-1}\alpha$. Furthermore, $p^{\emptyset}(\alpha,\alpha,\gamma)=\gamma$, since once $\alpha=\beta$ we have to take $\delta=\gamma$, as already explained  in Subsection~\ref{Degree two}. Proposition~\ref{Centrality-Sum} below gives us an alternative argument: there is no other choice possible for $\delta$ because $\pr_{A}(\alpha,\alpha,\gamma,\delta)$ has to be zero, and $\delta=\gamma$ is a \emph{valid} choice, since $\pr_{A}(\alpha,\alpha,\gamma,\gamma)=0$, so it is the \emph{uniquely valid} one.

Similarly, writing $\delta=p^{\emptyset}(a,b,c,d,\alpha,\beta,\gamma)$ for a configuration such as in Figure~\ref{Figure-Diamond}, we find
\[
\begin{cases}
d=p^{\emptyset}(a,b,c,d,a,b,c)\\
\alpha=p^{\emptyset}(a,b,b,a,\alpha,\beta,\beta)\\
\gamma=p^{\emptyset}(b,b,c,c,\beta,\beta,\gamma).
\end{cases}
\]

\begin{proposition}\label{Centrality-Sum}
In a semi-abelian category, let $F$ be an $n$-cubic central extension with direction $A$. Then in any product diagram
\[
\xymatrix{0 \ar[r] & A \ar@<-.5ex>@{{ |>}->}[r]_-{\ker\pi^{I}} & \bigboxvoid_{i\in n}\R{f_{i}} \ar@<-.5ex>@{ >>}[r]_-{\pi^{I}} \ar@{.>}@<-.5ex>[l]_-{\pr_{A}} & \bigboxdot_{i\in n}^{I}\R{f_{i}} \ar[r] \ar@{.>}@<-.5ex>[l]_-{\iota^{I}} & 0}
\]
induced by Theorem~\ref{Theorem-Higher-Centrality}, the projection $\pr_{A}$ is an alternating sum
\begin{equation}\label{Sum-Formula}
\sum_{J\subseteq n}(-1)^{|J|}\eta_{F_{n}}\comp \pr_{J}
\end{equation}
where $\pr_{J}\colon{\bigboxvoid_{i\in n}\R{f_{i}}\to F_{n}}$ sends a diamond to its $J$-face.
\end{proposition}
\begin{proof}
The idea behind the proof may be illustrated as follows in dimension two. (Here we let $F$ be the two-cubic extension from Diagram~\eqref{Double-Extension} to simplify notations.) When the calculation
\[
\vcenter{\xymatrix@1@!0@=4em{\gamma \ar@{.}[r] \ar@{.}[d]|-{1} & \beta \ar@{.}[d] \\
\delta \ar@{.}[r]|-{0} & \alpha}}
-
\vcenter{\xymatrix@1@!0@=4em{\gamma \ar@{.}[r] \ar@{.}[d]|-{1} & \beta \ar@{.}[d] \\
\gamma \ar@{.}[r]|-{0} & \beta}}
+
\vcenter{\xymatrix@1@!0@=4em{\beta \ar@{.}[r] \ar@{.}[d]|-{1} & \beta \ar@{.}[d] \\
\beta \ar@{.}[r]|-{0} & \beta}}
-
\vcenter{\xymatrix@1@!0@=4em{\beta \ar@{.}[r] \ar@{.}[d]|-{1} & \beta \ar@{.}[d] \\
\alpha \ar@{.}[r]|-{0} & \alpha}}
=
\vcenter{\xymatrix@1@!0@R=4em@C=6em{0 \ar@{.}[r] \ar@{.}[d]|-{1} & 0 \ar@{.}[d] \\
\delta-\gamma+\beta-\alpha \ar@{.}[r]|-{0} & 0,}}
\]
in which we denote the equivalence classes in the quotient by representative elements, is made in the abelian object~$\ab(\R{d}\square\R{c})$, we see that the result belongs to the kernel $A$ of the projection $\ab\pi^{\emptyset}$. Indeed, the pullback $\R{d}\times_{X}\R{c}$ is preserved by the functor $\ab$, and the projections to~$\ab \R{d}$ and~$\ab\R{c}$ send the above sum to zero. Writing $\eta_{\boxvoid}=\eta_{\R{d}\square\R{c}}$, this gives us the morphism
\[
\eta_{\boxvoid}\comp\pr_{\emptyset}\;-\;\eta_{\boxvoid}\comp\pr_{\{1\}}\;+\;\eta_{\boxvoid}\comp\pr_{2}\;-\;\eta_{\boxvoid}\comp\pr_{1}\colon{\R{d}\square\R{c}\to A},
\]
clearly a splitting for $\ker \pi^{\emptyset}$; hence by Lemma~\ref{Lemma-Split-Kernel-Product} it is the needed product projection. Note that the terms in this sum are obtained by projecting a diamond $(\alpha,\beta,\gamma,\delta)$ to a certain subdiamond, and then considering it again as a two-fold diamond via reflexivity. In the first term we  do not project at all, in the second term we project to $(\gamma,\beta)$ in $\R{f_{0}}$, in the fourth we project to $(\alpha,\beta)$ in $\R{f_{1}}$, and in the third term we project all the way to $F_{2}$.

For general $n$, let us again consider the commutative square~\eqref{Square-Blokskes}---which is a pullback by centrality of $F$---and the induced kernels:
\[
\xymatrix{A \ar@{=}[d] \ar@{{ |>}->}[r] & \bigboxvoid_{i\in n}\R{f_{i}} \ar@{ >>}[r]^-{\pi^{I}} \ar@{ >>}[d]_-{\eta_{\bigboxvoid_{i\in n}\R{f_{i}}}} \pullback & \bigboxdot^{I}_{i\in n}\R{f_{i}} \ar@{ >>}[d]^-{\eta_{\bigboxdot^{I}_{i\in n}\R{f_{i}}}}\\
A \ar@{{ |>}->}[r] & \ab(\bigboxvoid_{i\in n}\R{f_{i}}) \ar@{ >>}[r]_-{\ab\pi^{I}} & \ab(\bigboxdot^{I}_{i\in n}\R{f_{i}})}
\]
Since $\ab(\bigboxdot^{I}_{i\in n}\R{f_{i}})=\bigboxdot^{I}_{i\in n}\ab\R{f_{i}}$ by Lemma~\ref{Lemma-ab-of-Square}, the abelian object $A$ being the kernel of $\ab \pi^{I}$ implies that it is the direction of $\ab(\kuub(F,n\setminus I))$, which means~${A=\bigcap_{i\in n}\K{\ab\pr^{i}_{\delta_{n\setminus I}(i)}}}$. In order to define a morphism with codomain~$A$, we now only need to define a morphism with codomain~$\ab(\bigboxvoid_{i\in n}\R{f_{i}})$ which becomes zero when composed with the
\[
\ab\pr^{i}_{\delta_{n\setminus I}(i)}\colon\ab(\bigboxvoid_{j\in n}\R{f_{j}})\to \ab(\bigboxvoid_{j\in n\setminus \{i\}}\R{f_{j}}).
\]
We shall use this procedure to define a splitting for $\ker\pi^{I}$ as an alternating sum, which will then automatically be the needed product projection by Lemma~\ref{Lemma-Split-Kernel-Product}.

Recall the notation introduced in Subsection~\ref{Chosen-cubes}. Then, for any $J\subseteq n$, write~$I\ominus J$ for the symmetric difference $(I\cup J)\setminus (I\cap J)$ of $I$ and $J$, and put
\[
\xi_{(F,I,J)}=\kuub(F,J)^{n\setminus (I\ominus J)}_{n}\comp\kuub(F,J)^{n}_{n\setminus (I\ominus J)}\colon{\bigboxvoid_{i\in n}\R{f_{i}}\to \bigboxvoid_{i\in n}\R{f_{i}}}.
\]
This formalises the process of ``projecting to a subdiamond, then including again via reflexivity''. Now note that, given any element $x$ of $\bigboxvoid_{i\in n}\R{f_{i}}$, the $I$-entry of $\xi_{(F,I,J)}(x)$ is $x_{J}$. Furthermore, after projecting in any direction $i\in n$ onto $\R{f_{i}}$, every morphism $\pr^{i}_{\delta_{n\setminus I}(i)}\comp \xi_{(F,I,J)}$ occurs twice: indeed
\[
\pr^{i}_{\delta_{n\setminus I}(i)}\comp \xi_{(F,I,J)}=\pr^{i}_{\delta_{n\setminus I}(i)}\comp\xi_{(F,I,J\cup \{i\})}
\]
when $i\not\in J$. Two such terms will cancel each other when the alternating sum below is composed with $\ab\pi^{I}$. Hence the induced morphism
\[
\sum_{J\subseteq n}(-1)^{|J|}\eta_{(\bigboxvoid_{i\in n}\R{f_{i}})}\comp \xi_{(F,I,J)}\colon\bigboxvoid_{i\in n}\R{f_{i}}\to \ab(\bigboxvoid_{i\in n}\R{f_{i}})
\]
satisfies the conditions required to lift over $A$---that is to say, it becomes zero when we compose it with $\ab\pi^{I}$. Its $I$-entry is precisely the needed formula~\eqref{Sum-Formula}, while the other entries are zero. In particular, the morphism~\eqref{Sum-Formula} does indeed split the kernel of $\pi^{I}$.
\end{proof}

Note that the formula~\eqref{Sum-Formula} for the projection $\pr_{A}$ is independent of the chosen index~$I\subseteq n$.

When $n=1$, Proposition~\ref{Centrality-Sum} reduces to a well-known property of (one-cubic) central extensions (see~\cite{Bourn-Gran}): if $f\colon{X\to Z}$ is central and $x_{0}$, $x_{1}\colon {W\to X}$ are such that~$f\comp x_{0}=f\comp x_{1}$, then they induce a unique morphism~${x_{1}-x_{0}\colon{W\to A}}$ to the kernel $A$ of $f$ such that $x_{0}$ and ${x_{1}-x_{0}}$ together determine~$x_{1}$.

\section{Torsors and centrality}\label{Section-Torsors-and-Centrality}
We analyse the concept of torsor from the point of view of centrality of higher extensions. We prove that a truncated simplicial resolution of an object $Z$ is a torsor of $Z$ by an abelian object $A$ if and only if the underlying extension is central with direction $A$ (Theorem~\ref{Theorem-Torsor-Equivalence}; one implication is Proposition~\ref{Proposition-Central-then-Torsor}, the other Proposition~\ref{Proposition-Torsor-then-Central}).

Let $Z$ be an object and $(A,\xi)$ a $Z$-module in a semi-abelian category $\X$. Recall from Subsection~\ref{Torsors} that an $n$-torsor of $Z$ by $(A,\xi)$ is an augmented simplicial object $\TT$ together with a simplicial morphism $\tt\colon {\TT\to \KK((A,\xi),n)}$ such that
\begin{enumerate}
\item[(T1)] $\tt$ is a fibration which is exact from degree $n$ on;
\item[(T2)] $\TT\cong \Cosk_{n-1}\TT$;
\item[(T3)] $\TT$ is a resolution.
\end{enumerate}

\subsection{Why extensions?}
Condition (T2) in the definition of $n$-torsor means that (the simplicial object-part $\TT$ of) an $n$-torsor $(\TT,\tt)$ \emph{is} the ${(n-1)}$-truncated simplicial object $T=\trunc_{n-1} \TT$ (Subsection~\ref{Truncations}), in the sense that this is the only information~$\TT$ contains. Its initial object is $T_{n}=\TT(n)=\TT_{n-1}$. Condition (T3) means that the underlying $n$-fold arrow of $T$ is an extension (Subsection~\ref{Resolutions}).

\subsection{Why trivial actions?}
We shall prove that for an $n$-torsor $(\TT,\tt)$ of an object~$Z$ by a $Z$-module $(A,\xi)$ in a semi-abelian category, the action $\xi$ is trivial if and only if the induced one-cubic extension
\[
\lind \del_{i}\rind_{i}=l_{T}\colon {T_{n}=\TT_{n-1}\to \cycle(\TT,n-1)=\L T}
\]
is central with respect to abelianisation. In other words, an $n$-torsor $(\TT,\tt)$ has a trivial action if and only if
\[
\Bigl[\bigcap_{i\in n} \R{\del_{i}},\nabla_{T_{n}}\Bigr]=\Delta_{T_{n}}
\qquad
\text{or, equivalently,}
\qquad
\Bigl[\bigcap_{i\in n} \K{\del_{i}},T_{n}\Bigr]=0;
\]
see Example~\ref{Example-Dimension-One}. This extends Proposition~3.3 in~\cite{RVdL} to higher dimensions. It also explains why only cohomology \emph{with trivial coefficients} can ever classify higher central extensions: this commutator condition is part of the centrality by Lemma~\ref{Lemma-Direction-Limit}.

\begin{proposition}\label{Proposition-Trivial-Action}
In a semi-abelian category, consider an object $Z$ and a $Z$"~mod\-ule~$(A,\xi)$. For any $n$-torsor $(\TT,\tt)$ of $Z$ by $(A,\xi)$, the kernel of
\[
\lind \del_{i}\rind_{i}=l_{T}\colon {T_{n}\to \cycle(\TT,n-1)=\L T}
\]
is $A$, and the following conditions are equivalent:
\begin{enumerate}
\item the action $\xi$ is trivial;
\item the one-cubic extension $l_{T}$ is central;
\item for all $i\in n$ we have $\cycle(\TT,n)\cong A\times \horn^{i}(\TT,n)$; more precisely,
\[
\widehat\del_{i}\colon \cycle(\TT,n)\to \horn^{i}(\TT,n)
\]
is a product projection with kernel $A$.
\end{enumerate}
\end{proposition}
\begin{proof}
For any $i\in n$, Lemma~\ref{Lemma-DeltaLambda-Square} tells us that the square
\[
\xymatrix{\cycle(\TT,n) \ar[r]^-{\widehat\del_{i}} \ar[d]_-{\del_{i}} & \horn^{i}(\TT,n) \ar[d] \\
T_{n} \ar[r]_-{\lind \del_{i}\rind_{i}} & \cycle(\TT,n-1)}
\]
is a pullback. Note that all its arrows are regular epimorphisms: the morphism~$\del_{i}$ as any split epimorphism; $\lind\del_{i}\rind_{i}$ since $\TT$ is a resolution; and $\widehat\del_{i}$ either by the Kan property, which all simplicial objects in a semi-abelian category have, or as a pullback of~$\lind\del_{i}\rind_{i}$. We see that the kernel of $\lind \del_{i}\rind_{i}$ is isomorphic to the kernel of~$\widehat\del_{i}$ (Lemma~\ref{Lemma-Iso-Pullback}), and furthermore $\lind \del_{i}\rind_{i}$ is central if and only if so is $\widehat\del_{i}$---indeed, central extensions are preserved and reflected by pullbacks of extensions along extensions. Since~$(\TT,\tt)$ is an $n$-torsor, also the square
\[
\xymatrix{\cycle(\TT,n) \ar@<-.5ex>[r]_-{\widehat\del_{i}} \ar[d]_-{\lind \varsigma,\del_{0}^{n+1}\rind} & \horn^{i}(\TT,n) \ar@{.>}@<-.5ex>[l] \ar[d]^{\del_{0}^{n}} \\
(A,\xi) \rtimes Z \ar@<-.5ex>[r]_-{p} & Z \ar@<-.5ex>[l]_-{s}}
\]
is a pullback, by the exact fibration property. This already proves that the kernel of $\lind \del_{i}\rind_{i}$ is $A$ (again Lemma~\ref{Lemma-Iso-Pullback}). Note that a split epimorphism with abelian kernel represents a trivial action if and only if it is a product projection, if and only if it is a trivial extension, if and only if it is a central extension. Again using that central extensions are preserved and reflected by pullbacks of extensions along extensions we obtain the claimed result.
\end{proof}

Hence, from now on, we shall only have to consider torsors of $Z$ by a trivial module $(A,\tau)$---we called them \emph{$n$-torsors of $Z$ by $A$} in Subsection~\ref{Torsors}---and restrict our cohomology theory accordingly. In this case, a torsor ``looks'' as follows:
\begin{equation}\label{torsordiagram}
\resizebox{.9\textwidth}{!}{\mbox{$
\vcenter{\xymatrix@=45pt{
\cycle(\TT,n+1) \ar[d]_{\cdots\quad\lind\lind\varsigma\circ\del_{i}\rind_{i},\del_{0}^{n+2}\rind} \ar@<2.33ex>[r] \ar@<1.16ex>[r] \ar@<-2.33ex>[r]^-{\vdots} & \cycle(\TT,n) \ar[d]^{\lind\varsigma,\del_{0}^{n+1}\rind} \ar@<1.75ex>[r] \ar@<-1.75ex>[r]^-{\vdots} & \TT_{n-1} \ar[d]^-{\del_{0}^{n}} \ar@<1.75ex>[r] \ar@<-1.75ex>[r]^-{\vdots} & \TT_{n-2} \ar[d]^-{\del_{0}^{n-1}} \ar@{}[r]|-{\cdots} & \TT_{0} \ar[d]^-{\del_{0}} \ar[r]^-{\del_{0}} & \TT_{-1} \ar@{=}[d]\\
A^{n+1} \times Z \ar@<2.33ex>[r]^-{\del_{n+1}\times 1_{Z}} \ar@<1.16ex>[r]|-{\pr_{n}\times 1_{Z}} \ar@<-2.33ex>[r]_-{\pr_{0}\times 1_{Z}}^-{\vdots} & A \times Z \ar@<1.75ex>[r]^-{\pr_{Z}} \ar@<-1.75ex>[r]_-{\pr_{Z}}^-{\vdots} & Z \ar@{=}@<1.75ex>[r] \ar@{=}@<-1.75ex>[r]^-{\vdots} & Z \ar@{}[r]|-{\cdots} & Z \ar@{=}[r] & Z}}
$}}
\end{equation}

\begin{remark}\label{Naturality-Product-Decomposition}
It is clear from the proof that the product decomposition (iii) is natural in $(\TT,\tt)$, so that any morphism of torsors is compatible with the induced product decompositions.
\end{remark}

\begin{lemma}\label{Lemma-Direction-Cycle}
For any simplicial resolution $\XX$, the kernel of any induced regular epimorphism $\widehat\del_{i}\colon{\cycle(\XX,n)\to \horn^{i}(\XX,n)}$ is the direction $A$ of the underlying $n$-cubic extension $X$.
\end{lemma}
\begin{proof}
If an $(n,i)$-horn $\widehat x_{i}$ of an $n$-cycle $x$ in $\XX$ is zero, then the $i$-face~$x_{i}$ which is missing in the horn must have boundary zero, so that~$x_{i}$ belongs to~$A$. Conversely, the inclusion of $A$ into $\cycle(\XX,n)$ takes an element $a$ of $A$ and sends it to the $n$-cycle in $\XX$ which is zero everywhere---except in its $i$-entry, where it is $a$. This $n$-cycle is sent to zero by $\widehat\del_{i}$.

More formally, this also follows from Lemma~\ref{Direction-as-Kernel} combined with Lemma~\ref{Lemma-DeltaLambda-Square}, since ${\lind\del_i\rind_i = l_X\colon \XX_{n-1}\to \cycle(\XX,n-1)}$: the kernel of $\widehat\del_{i}$ coincides with the kernel of~$l_{X}$ since the square in Lemma~\ref{Lemma-DeltaLambda-Square} is a pullback, and the latter kernel is $A$ by Lemma~\ref{Direction-as-Kernel}.
\end{proof}

\subsection{Multiplying simplices in a torsor}\label{Multiplication}
As explained in~\cite{Duskin}, given an~$n$-torsor $(\TT,\tt)$ of $Z$ by $A$ and an integer $i\in n$, the isomorphism
\[
\cycle(\TT,n)\cong A\times \horn^{i}(\TT,n)
\]
induces a multiplication or composition of the simplices in a horn to the ``missing face'' such that the thus completed $n$-cycle ``commutes'', in the sense that its projection on $A$ is zero. So a horn may be considered as a \emph{composable aggregation of simplices}---compare with the higher Mal'tsev structures $p^{I}$ from Subsection~\ref{Maltsev}. Indeed, we may simply use the morphism
\[
m^{i}\colon\xymatrix@1{\horn^{i}(\TT,n) \ar[r]^-{\lind 0,1\rind} & A\times \horn^{i}(\TT,n) \ar[r]^-{\cong} & \cycle(\TT,n) \ar[r]^-{\del_{i}} & T_{n}.}
\]
This composition of $(n,i)$-horns satisfies certain additional properties~\cite{Duskin}, of which for us the most important one is compatibility with degeneracies. From the axioms of torsor (the requirement that $\tt\colon{\TT\to \KK(Z,A,n)}$ be a simplicial morphism) it follows that a degenerate $n$-cycle commutes. Hence any $(n,i)$-horn in $\TT$ which \emph{may be} completed to a degenerate $n$-cycle \emph{has to be} completed this way, and hence composes to the $i$-face of this degenerate $n$-cycle.

For instance, in degree two, the left hand side $(2,1)$-horn
\[
\vcenter{\xymatrix@1@!0@R=2.4495em@C=1.4142em{& {\cdot} \ar[rd]^-{\sigma_{0}\del_{0}\alpha}\\
{\cdot} \ar[ru]^-{\alpha} && {\cdot}}}
\qquad\qquad
\vcenter{\xymatrix@1@!0@R=2.4495em@C=1.4142em{& {\cdot} \ar@{}[d]|(.7){\sigma_{1}\alpha} \ar[rd]^-{\sigma_{0}\del_{0}\alpha}\\
{\cdot} \ar[ru]^-{\alpha} \ar@{.>}[rr]_-{\alpha} && {\cdot}}}
\]
fits into the right hand side degenerate $2$-simplex $\sigma_{1}\alpha$. It follows by uniqueness that $m^{1}(\sigma_{0}\del_{0}\alpha,\alpha)=\alpha$. Likewise, $m^{0}(\alpha,\alpha)=\sigma_{0}\del_{0}\alpha$, etc.

\subsection{The exact fibration property}
Most of the fibration property (T1) of a torsor comes for free, since a regular epimorphism of simplicial objects in a regular Mal'tsev category is always a fibration~\cite[Proposition~4.4]{EverVdL2}. Given a simplicial morphism $\tt\colon{\TT\to \KK(Z,A,n)}$ satisfying (T2) and (T3), already the~${\tt_{i}=\del_0^{i+1}}$ are regular epimorphisms for all $i\in n$, so it suffices to check the regularity of $\tt_{n}$ and~$\tt_{n+1}$. Then there is the exactness, but this reduces to one square being a pullback---Diagram~\eqref{Fundamental-Square} for any $i\in n$---which in turn corresponds to a direction property.

\begin{proposition}\label{Isos-are-Free}
Suppose that $Z$ is an object and $A$ is an abelian object in a semi-abelian category. Let $\tt\colon{\TT\to \KK(Z,A,n)}$ be as in the definition of torsors, satisfying conditions {\rm (T2)} and {\rm (T3)}. Then for every $i$ the square
\begin{equation}\label{Fundamental-Square}
\vcenter{\xymatrix{\cycle(\TT,n) \ar[r]^-{\widehat\del_{i}} \ar[d]_-{\lind \varsigma,\del_{0}^{n+1}\rind} & \horn^{i}(\TT,n) \ar[d]^{\del_{0}^{n}} \\
A \times Z \ar[r]_-{\pr_{Z}} & Z}}
\end{equation}
is a pullback if and only if the induced morphism ${\bigcap_{i}\K{\del_{i}}\to A}$ is an isomorphism. When this is the case, the simplicial morphism $\tt$ is a fibration, exact from degree $n$ on, so that $(\TT,\tt)$ is an $n$-torsor of $Z$ by~$A$.
\end{proposition}
\begin{proof}
Again, $\widehat\del_{i}$ is a regular epimorphism by the Kan property. As in the proof of Lemma~\ref{Lemma-Direction-Cycle}, the kernel of $\widehat\del_{i}$ is $\K{\lind\del_{i}\rind_{i}}=\bigcap_{i}\K{\del_{i}}$. Via Lemma~\ref{Lemma-Iso-Pullback} this already proves the equivalence.

Recall that every regular epimorphism of simplicial objects in a semi-abelian category is a fibration. When the above square~\eqref{Fundamental-Square} is a pullback (for any~${i\in n}$), the morphism ${\del_{0}^{n}}$ being regular epimorphic implies that also $\lind \varsigma,\del_{0}^{n+1}\rind$ is a regular epimorphism.

One degree up, the corresponding squares are automatically pullbacks: indeed, any comparison ${\cycle(\TT,n+1)\to \horn^{i}(\TT,n+1)}$ is an isomorphism by the axiom~(T2) which tells us that every $n$-simplex in $\TT$ is an $n$-cycle, as is any morphism
\[
\widehat\del_{i}\colon A^{n+1}\times Z \to \horn^{i}(\KK(Z,A,n),n+1)=A^{n+1}\times Z.
\]
In higher degrees there is nothing to be checked because $\tt\colon{\TT\to \KK(Z,A,n)}$ is completely determined by the coskeleton construction. This implies that $\tt$ is a regular epimorphism in all degrees, hence it is a fibration. This fibration is exact in degree~$n$ since \eqref{Fundamental-Square} is a pullback for every $i$, and in higher degrees since both its domain and its codomain are constructed as a coskeleton, so that we can apply Lemma~\ref{Lemma coskeleton pullback}.
\end{proof}

Thus we see that an $n$-torsor $(\TT,\tt)$ of $Z$ by $A$ has an underlying $n$-cubic extension of~$Z$ of which the direction is $A$. Furthermore, the squares~\eqref{Fundamental-Square} are pullbacks, which means that $\cycle(\TT,n)\cong A\times \horn^{i}(\TT,n)$. Note that the projection on $\horn^{i}(\TT,n)$ is $\widehat\del_{i}$ and the projection on $A$ is $\varsigma$.

In what follows we shall prove that this condition is equivalent to the centrality of the underlying $n$-cubic extension. Given an $n$-cubic central extension~$T$ of~$Z$ by~$A$, we construct a simplicial morphism $\tt\colon{\TT=\cosk_{n-1}T\to \KK(Z,A,n)}$ such that the squares~\eqref{Fundamental-Square} are all pullbacks. As explained above, this is enough for~$(\TT,\tt)$ to be an $n$-torsor. Furthermore, Proposition~\ref{Proposition-Full} tells us that such a simplicial morphism~$\tt$ is uniquely determined, so that its existence is a \emph{property} of $T$, not additional structure---as it should be, because centrality is also a property.

The other implication (which says that the underlying $n$-cubic extension of an~$n$-torsor is always central) will be treated in the following section.

\subsection{Embedding cycles into diamonds}
Up to symmetry of the diamond, there is a unique way a cycle may be embedded into a diamond using degeneracies to fill up missing faces. In degree two there is the morphism
\[
s_{2}(\XX)\colon\cycle(\XX,2)\to \R{\del_{1}}\square \R{\del_{0}}\colon \lind x_{0},x_{1},x_{2}\rind \mapsto\vcenter{\xymatrix@1@!0@=3.5em{\sigma_{0}\del_{1}x_{0} \ar@{.}[r] \ar@{.}[d]|-{1} & x_{2} \ar@{.}[d] \\
x_{0} \ar@{.}[r]|-{0} & x_{1}}}
\]
which sends the left hand side (empty) triangle
\[
\vcenter{\xymatrix@1@!0@R=2.4495em@C=1.4142em{& {\cdot} \ar[rd]^-{x_{0}}\\
{\cdot} \ar[ru]^-{x_{2}} \ar[rr]_-{x_{1}} && {\cdot}}}
\qquad\qquad
\vcenter{\xymatrix@1@!0@=2em{& {\cdot}\\
{\cdot} \ar[ru]^-{\sigma_{0}\del_{1}x_{0}} \ar[rd]_-{x_{0}} && {\cdot} \ar[lu]_-{x_{2}} \ar[ld]^-{x_{1}}\\
&{\cdot}}}
\]
to the right hand side diamond. In degree three we have
\[
s_{3}(\XX)\colon\cycle(\XX,3)\to \bigboxvoid_{i\in 3}\R{\del_{i}}\colon \lind x_{0},x_{1},x_{2},x_{3}\rind\mapsto\vcenter{\xymatrix@1@!0@=2.5em{& \sigma_{0}\del_{2}x_{0} \ar@{.}[rr] \ar@{.}[dd] \ar@{.}[dl] && x_{3} \ar@{.}[dd] \ar@{.}[dl]\\
\sigma_{0}\del_{1}x_{0} \ar@{.}[rr] \ar@{.}[dd]|-{1} && x_{2} \ar@{.}[dd] \\
& \sigma_{1}\del_{2}x_{0} \ar@{.}[rr] \ar@{.}[dl]|-{2} && \sigma_{1}\del_{2}x_{1} \ar@{.}[dl] \\
x_{0} \ar@{.}[rr]|-{0} && x_{1}}}
\]
and in general we have an inductive formula, as follows.

\begin{notation}[D\'ecalage]
Let ${}^{-}\XX$ denote the \defn{d\'e\-ca\-lage} of $\XX$, the augmented simplicial object constructed out of $\XX$ by forgetting the lowest degree $\XX_{-1}$ and the last face operators $\del_{n}\colon{\XX_{n}\to \XX_{n-1}}$, so that ${}^{-}\XX_{n}=\XX_{n+1}$. We obtain a morphism of simplicial objects $\dd\colon{{}^{-}\XX\to \XX}$ by $\dd_{n}=\del_{n+1}\colon {{}^{-}\XX_{n}=\XX_{n+1}\to \XX_{n}}$.
\end{notation}

\begin{proposition}\label{Proposition-Inclusion}
For any simplicial object $\XX$ in a semi-abelian category and any~${n\geq 2}$ there is a canonical natural inclusion
\[
s_{n}(\XX)\colon\cycle(\XX,n)\to \bigboxvoid_{i\in n}\R{\del_{i}}.
\]
\end{proposition}
\begin{proof}
We give a proof by induction; the base step is explained above. Suppose $s_{n}(\XX)$ is defined for every $\XX$ and natural in $\XX$; we then construct a morphism $s_{n+1}(\XX)$, natural in $\XX$. Given an $(n+1)$-cycle
\[
x=\lind x_{0},\dots,x_{n},x_{n+1}\rind \in\cycle(\XX,n+1),
\]
note that both $\widehat x_{n+1}=\lind x_{0},\dots,x_{n}\rind$ and
\[
\widehat y_{n+1}=\lind\sigma_{n-1}\del_{0}x_{n+1},\dots,\sigma_{n-1}\del_{n-1}x_{n+1},x_{n+1}\rind,
\]
where
\[
y=\sigma_{n}x_{n+1}=\lind\sigma_{n-1}\del_{0}x_{n+1},\dots,\sigma_{n-1}\del_{n-1}x_{n+1},x_{n+1},x_{n+1}\rind,
\]
are in $\cycle({}^{-}\XX,n)$. The induction hypothesis gives us a pair of diamonds, and we define
\[
s_{n+1}(\XX)(x)=\lind s_{n}({}^{-}\XX)(\widehat x_{n+1}),s_{n}({}^{-}\XX)(\widehat y_{n+1})\rind\in\bigboxvoid_{i\in n}\R{{}^{-}\del_{i}}\times \bigboxvoid_{i\in n}\R{{}^{-}\del_{i}}.
\]
Now we only have to show that this pair does belong to $\bigboxvoid_{i\in n+1}\R{\del_{i}}$, which means that $\del_{n}(s_{n}({}^{-}\XX)(\widehat x_{n+1}))=\del_{n}(s_{n}({}^{-}\XX)(\widehat y_{n+1}))$. This equality follows from the naturality of $s_{n}$, which makes the square
\[
\xymatrix{\cycle({}^{-}\XX,n) \ar[r]^-{\cycle(\dd,n)} \ar[d]_-{s_{n}({}^{-}\XX)} & \cycle(\XX,n) \ar[d]^{s_{n}(\XX)} \\
\bigboxvoid_{i\in n}\R{{}^{-}\del_{i}} \ar[r]_-{\dd} & \bigboxvoid_{i\in n}\R{\del_{i}}}
\]
commute, and the fact that $\cycle(\dd,n)(\widehat x_{n+1})$ is equal to $\cycle(\dd,n)(\widehat y_{n+1})$. Indeed, we have $\del_{n}x_{n}=\del_{n}x_{n+1}$ and
\[
\del_{n}x_{i}=\del_{i}x_{n+1}=\del_{n}\sigma_{n-1}\del_{i}x_{n+1}
\]
for every $i\in n$, so that the latter equality holds. This completes the construction of~$s_{n+1}(\XX)$, which is evidently natural in $\XX$.
\end{proof}

The morphism $s_{n}(\XX)$ constructed above takes an ``element'' $x=\lind x_{0},\dots,x_{n}\rind$ of the object~$\cycle(\XX,n)$ and maps it to the diamond $s_{n}(\XX)(x)$ which has $x_{i}$ on its $i$"~entry and degeneracies elsewhere (see Subsection~\ref{indexing}). Clearly, $s_{n}(\XX)$ restricts to morphisms
\[
\dot s_{n}^{i}(\XX)\colon\horn^{i}(\XX,n)\to \bigboxdot^{i}_{j\in n}\R{\del_{j}},
\]
natural in $\XX$.

When we say that an ${(n-1)}$-truncated simplicial resolution is \defn{central}, we mean that such is the underlying $n$-cubic extension. We write $\SimpCExt^{n}_{Z}(\X)$ for the (non-full) subcategory of $\Diag{n}(\X)$ consisting of those $3^{n}$-diagrams with an underlying $n$-cubic extension which is a central ${(n-1)}$-truncated simplicial resolution, with morphisms between such which restrict to simplicial morphisms. We write  
\begin{equation}\label{Dd}
\Dd_{(n,Z)} \colon{\SimpCExt^n_Z(\X)\to \Ab(\X)}
\end{equation}
for the restriction of $\D_{(n,Z)}$ to this category.

\begin{proposition}\label{Proposition-Central-then-Torsor}
If, in a semi-abelian category, an $(n-1)$-truncated simplicial resolution is central, then it is an $n$-torsor.
\end{proposition}
\begin{proof}
Let $\TT$ be a simplicial resolution and let $A$ be the direction of $T=\trunc_{n-1}\TT$, considered as a trivial $Z$-module. We have to define a morphism of augmented simplicial objects $\tt\colon{\TT\to \KK(Z, A, n)}$ as in~\eqref{torsordiagram}. Such a simplicial morphism is completely determined by the choice of a suitable morphism $\varsigma\colon{\cycle(\TT,n)\to A}$.

Consider, for $i\in n+1$, the commutative square of solid arrows
\[
\xymatrix{0 \ar@{.>}[r] & A \ar@{:}[d] \ar@{.>}@<-.5ex>[r] & \cycle(\TT,n) \pullbackdots \ar@{.>}@<-.5ex>[l]_-{\varsigma} \ar@<-.5ex>[r]_-{\widehat\del_{i}} \ar[d]_-{s_{n}(\TT)} & \horn^{i}(\TT,n) \ar@{.>}@<-.5ex>[l] \ar[d]^{\dot s_{n}^{i}(\TT)} \ar@{.>}[r] & 0 \\
0 \ar@{.>}[r] & A \ar@{.>}@<-.5ex>[r] & \bigboxvoid_{j\in n}\R{\del_{j}} \ar@<-.5ex>[r]_{\pi^{i}} \ar@{.>}@<-.5ex>[l]_-{\pr_{A}} & \bigboxdot_{j\in n}^{i}\R{\del_{j}} \ar@{.>}@<-.5ex>[l] \ar@{.>}[r] & 0}
\]
which embeds cycles into diamonds. By assumption, the kernel of $\pi^{i}$ is~$A$; moreover, by Theorem~\ref{Theorem-Higher-Centrality},
\[
\bigboxvoid_{j\in n}\R{\del_{j}}\cong A\times \bigboxdot_{j\in n}^{i}\R{\del_{j}}
\]
with $\pi^{i}$ the projection on $\bigboxdot_{j\in n}^{i}\R{\del_{j}}$. The square above is a pullback as a consequence of Lemma~\ref{Lemma-Iso-Pullback}, since $\widehat\del_{i}$ is a regular epimorphism by the extension property of $T$, and since the kernel of~$\widehat\del_{i}$ is~$A$ by Lemma~\ref{Lemma-Direction-Cycle}. This implies that
\[
\cycle(\TT,n)\cong A\times \horn^{i}(\TT,n)
\]
with $\widehat\del_{i}$ the projection on $\horn^{i}(\TT,n)$. We may now complete the square with the dotted arrows.

We choose $\varsigma$ to be $\pr_{A}\comp s_{n}(\TT)\colon{\cycle(\TT,n)\to A}$, the projection of $\cycle(\TT,n)$ on~$A$. We must prove that this does indeed give us a genuine morphism~$\tt\colon{\TT\to \KK(Z, A, n)}$; then the exact fibration property holds by Proposition~\ref{Isos-are-Free}, so that $(\TT,\tt)$ is an $n$-torsor.

For this, we only need to check that all the squares in the diagram
\[
\xymatrix@=45pt{
\cycle(\TT,n+1) \ar[d]_{\lind\lind\varsigma\circ\del_{i}\rind_{i},\del_{0}^{n+2}\rind} \ar@<2.33ex>[r]^-{\del_{n+1}} \ar@<1.16ex>[r] \ar@<-2.33ex>[r]^-{\vdots} & \cycle(\TT,n) \ar[d]^{\lind\varsigma,\del_{0}^{n+1}\rind}\\
A^{n+1} \times Z \ar@<2.33ex>[r]^-{\del_{n+1}\times 1_{Z}} \ar@<1.16ex>[r]|-{\pr_{n}\times 1_{Z}} \ar@<-2.33ex>[r]_-{\pr_{0}\times 1_{Z}}^-{\vdots} &
 A \times Z}
\]
commute. This condition reduces to the commutativity of just one square, the one ``on top'':
\begin{equation}\label{Sum-condition}
\varsigma\comp \del_{n+1}=(-1)^{n}\sum^{n}_{i=0}(-1)^{i}\varsigma\comp\del_{i}.
\end{equation}
In fact the morphism $\lind\lind\varsigma\comp\del_{i}\rind_{i},\del_{0}^{n+2}\rind$ is already the unique one that makes all the other squares commute. But this equality follows from Proposition~\ref{Centrality-Sum}, which tells us that the morphism $\varsigma$ itself may be considered as an alternating sum,
\[
\varsigma=\sum_{J\subseteq n}(-1)^{|J|}\eta_{\TT_{n-1}}\comp\pr_{J}\comp s_{n}(\TT).
\]
Using that the alternating sum
\[
\sum_{i=0}^{n+1}(-1)^{i}\eta_{\cycle(\TT,n)}\comp \del_{i}
\]
is zero by the simplicial identities, the equality~\eqref{Sum-condition} may now be obtained via a direct calculation in the abelian object~$A$.
\end{proof}

\begin{proposition}\label{Proposition-Full}
Given $f\colon X\to Y$ in $\SimpCExt^{n}_{Z}(\X)$, let $(\XX,\xx)$ and~$(\YY,\yy)$ be the $n$-torsors corresponding to $X$ and $Y$ and $\ff\colon{\XX\to \YY}$ the simplicial morphism corresponding to $f$. If $f$ keeps the direction fixed, that is to say, if
\[
\Dd_{(n,Z)}f=1\colon \Dd_{(n,Z)}X\to \Dd_{(n,Z)}Y, 
\]
then~$\yy\comp \ff=\xx$ so that $\ff$ is a morphism of torsors. In other words, we have a functor
\[
\Dd_{(n,Z)}^{-1}A\to \Tors^{n}(Z,A).
\]
Furthermore, this functor is fully faithful.
\end{proposition}
\begin{proof}
For any morphism of central truncated simplicial objects which keeps the terminal object and the direction fixed, the projections to the directions are compatible with it. To see this we only need to consider the diagram
\[
\xymatrix{A \ar@{=}[d] \ar@<-.5ex>@{{ |>}->}[r] & \cycle(\XX,n) \ar[d]^-{\cycle(\ff,n)} \ar@{ >>}@<-.5ex>[l]_-{\varsigma_{\XX}} \ar@{ >>}[r] & \horn^{i}(\XX,n) \ar[d]^-{\smallhorn^{i}(\ff,n)}\\
A \ar@<-.5ex>@{{ |>}->}[r] & \cycle(\YY,n) \ar@{ >>}@<-.5ex>[l]_-{\varsigma_{\YY}} \ar@{ >>}[r] & \horn^{i}(\YY,n)}
\]
and note that $\varsigma_{\YY}\comp \cycle(\ff,n)=\varsigma_{\XX}$ by naturality of the product decompositions \emph{induced by centrality}---see Remark~\ref{Remark-Full-3} or the above proof. This shows that $\ff$ is a morphism of torsors.

This immediately shows that the functor $\Dd_{(n,Z)}^{-1}A\to \Tors^{n}(Z,A)$ is fully faithful as claimed: after all, a morphism in $\Tors^{n}(Z,A)$ is nothing but a morphism in $\Dd_{(n,Z)}^{-1}A$ satisfying an additional condition---but we just proved that this condition always holds.
\end{proof}

\section{The commutator condition}\label{Section-Commutator-Assumption}
In general it is not clear how an isomorphism on the simplicial level may be extended to an isomorphism on the level of higher-dimensional diamonds. Therefore, to prove that every $n$-torsor is an $n$-cubic central extension, we shall add an assumption on the base category: we ask that higher central extensions may be characterised in terms of binary Huq commutators. This happens in many cases, but thus far we have no precise characterisation of the categories which satisfy this condition.

It is proved in Section~9.1 of~\cite{EGVdL} that an $n$-cubic extension of groups $F$ is central with respect to $\Ab$ if and only if $\bigl[\bigcap_{i\in I}\K{f_{i}},\bigcap_{i\in n\setminus I}\K{f_{i}}\bigr]=0$ for all~$I\subseteq n$. The theory which we develop depends crucially on a similar characterisation of higher central extensions, valid in a sufficiently general context.

\begin{definition}\label{Commutator-assumption}~\cite{RVdL3}
We say that an $n$-cubic extension $F$ in a semi-abelian category~$\X$ is \defn{H-central} when
\[
\Bigl[\bigcap_{i\in I}\K{f_{i}},\bigcap_{i\in n\setminus I}\K{f_{i}}\Bigr]=0
\]
for all $I\subseteq n$. The category $\X$ satisfies the \defn{commutator condition on $n$-cubic central extensions} when the H-central $n$-cubic extensions in~$\X$ coincide with the \defn{categorically central} ones, namely those which are central with respect to~$\Ab(\X)$ in the Galois-theoretic sense used throughout the rest of the paper. We say that~$\X$ satisfies the \defn{commutator condition~(CC)} when it satisfies the commutator condition on $n$-cubic central extensions for all~$n$.
\end{definition}

\subsection{The cases $n=1$ and $n=2$}
As explained in the Introduction and in Example~\ref{Example-Dimension-One}, every semi-abelian category satisfies the commutator condition for one-cubic central extensions. From the Introduction and Example~\ref{Example-Dimension-Two}, it follows that in a semi-abelian category, the commutator condition on two-cubic central extensions is weaker than the \emph{Smith is Huq} condition~\cite{MFVdL}.

\subsection{Some examples}
It is shown in~\cite{EGVdL} that, next to the category of groups, also the categories Lie algebras and non-unitary rings have (CC). The examples of Leibniz and Lie $n$-algebras were treated in~\cite{CKLVdL}. Moreover, from~\cite{RVdL3} we know that any semi-abelian category with the \emph{Smith is Huq} condition has (CC), while the categories of loops and of commutative loops do not satisfy this condition. As examples we have \emph{action representative} semi-abelian categories~\cite{BJK2,Borceux-Bourn-SEC}, \emph{action accessible} categories~\cite{BJ07}---in particular all \emph{categories of interest}~\cite{Orzech,Montoli}, so also all \emph{varieties of groups}~\cite{Neumann}---next to all \emph{strongly semi-abelian}~\cite{Bourn2004} and \emph{Moore} categories~\cite{Gerstenhaber, Rodelo:Moore}. For instance, the categories of associative and non-associative algebras and of (pre)crossed modules satisfy (CC).

A general context where many examples may be found is given by those semi-abelian categories for which the abelianisation functor is protoadditive~\cite{Everaert-Gran-nGroupoids, Everaert-Gran-TT}, as considered below. This example gives two extreme special cases: semi-abelian arithmetical categories recalled in Example~\ref{Arithmetical categories}, such as the categories of von Neumann regular rings, Boolean rings and Heyting semilattices (where the cohomology theory becomes trivial) on the one hand, and abelian categories recalled in Example~\ref{Abelian categories} (where, via a version of the Dold--Kan correspondence~\cite{Dold-Puppe}, the theory gives us the Yoneda $\ext$ groups) on the other.

\begin{example}[Protoadditive abelianisation]
Recall from~\cite{Everaert-Gran-nGroupoids} that a functor between semi-abelian categories is \defn{protoadditive} when it preserves split short exact sequences
\[
\vcenter{\xymatrix{0 \ar[r] & K \ar@{{ |>}->}[r]^-{k} & X \ar@<-.5ex>@{-{ >>}}[r]_-{f} & Y \ar[r] \ar@<-.5ex>[l]_-{s} & 0}}
\]
(the cokernel $f$ is split by some morphism $s$). It is explained in~\cite{Everaert-Gran-TT} that, when~$\X$ is semi-abelian and the abelianisation functor $\ab\colon{\X\to \Ab(\X)}$ is protoadditive, the Huq commutator $[K,L]$ of two normal subobjects $K$, $L$ of an object~$X$ is $\lb K\cap L\rb=[K\cap L,K\cap L]$. This gives us
\[
\Bigl[\bigcap_{i\in I}\K{f_{i}},\bigcap_{i\in n\setminus I}\K{f_{i}}\Bigr]=\Bigl[\bigcap_{i\in n}\K{f_{i}},\bigcap_{i\in n}\K{f_{i}}\Bigr]=\lb\D_{(n,Z)}F\rb
\]
for any $n$-cubic extension $F$ of $Z$ and any $I\subseteq n$. Furthermore, by another result in~\cite{Everaert-Gran-TT}, an $n$-cubic extension is categorically central if and only if its direction is abelian; hence the commutator condition~(CC) holds. In fact, this argument extends easily to a proof that all semi-abelian arithmetical categories satisfy (SH).
\end{example}

A non-trivial instance of this situation, mentioned in~\cite{Everaert-Gran-TT}, is the variety of non-unitary rings that satisfy the law $abab=ab$. We now explain another special case, one which is less interesting from a cohomological point of view, but which does give a class of extreme examples.

\begin{example}[Arithmetical categories]
\label{Arithmetical categories}
Recall from \cite{Pedicchio2} that an exact Mal'\-tsev category is \defn{arithmetical} when every internal groupoid is an equivalence relation. We restrict ourselves to semi-abelian arithmetical categories, examples of which are the dual of the category of pointed sets, more generally, the dual of the category of pointed objects in any topos, and also the categories of von Neumann regular rings, Boolean rings and Heyting semi-lattices~\cite{Borceux-Bourn}. Since in such a category all abelian objects are trivial, the abelianisation functor is protoadditive, so that the commutator condition (CC) holds. (By the above, as shown in~\cite{MFVdL2}, every arithmetical category moreover satisfies (SH).) Here an $n$-cubic extension is categorically central if and only if its direction is zero, which means that the extension is a limit $n$-cube (or an isomorphism, when $n=1$). Hence the interpretation of cohomology in terms of higher central extensions (Theorem~\ref{Main-Theorem}) just means that any two $n$-cubic central extensions of an object~$Z$, so limit $n$-cubes over $Z$, are connected, because $\Centr^{n}(Z,0)\cong \H^{n+1}(Z,0)$ is trivial---which is, however, not difficult to prove directly.
\end{example}

At the other end of the spectrum we find the context of abelian categories where~(CC) also holds, and where the cohomology theory reduces to Yoneda's interpretation of $\ext^{n}(Z,A)$.

\begin{example}[Abelian categories]
\label{Abelian categories}
In an abelian category all Huq commutators are zero while all extensions are categorically central, which already gives us~(CC). On top of that the abelianisation functor is an identity, so that it is trivially protoadditive. Via the Dold--Kan theorem~\cite{Dold-Puppe}, an~$(n-1)$-truncated simplicial resolution, considered as an $n$-fold extension of $Z$ by $A$, corresponds to an exact sequence
\[
\xymatrix{0 \ar[r] & A \ar@{{ |>}->}[r] & C_{n-1} \ar[r] & C_{n-2} \ar[r] & {\cdots} \ar[r] & C_{0} \ar@{ >>}[r] & Z \ar[r] & 0}
\]
of length $n$. (See Figure~\ref{3x3 diag} on page~\pageref{3x3 diag} for a picture when $n=2$.) As a consequence of the results in Section~\ref{Section-Main-Theorem} we get
\[
\ext^{n}(Z,A)\cong \Centr^{n}(Z,A)\cong \H^{n+1}(Z,A),
\]
the equivalence between Yoneda's cohomology via $\ext$ groups~\cite{Yoneda-Exact-Sequences, MacLane:Homology} and comonadic cohomology which was first established in~\cite{Beck}. (See also~\cite{Barr-Beck}. A proof of the same result via torsor theory is given in~\cite{Glenn}.) The dimension shift is there because our numbering of the cohomology objects agrees with the classical non-abelian examples (groups, Lie algebras, etc.) rather than with the abelian case.
\end{example}

\subsection{From torsors to central extensions}
We are now ready to prove the equivalence between torsors and central extensions we need for our cohomological interpretation of higher central extensions.

\begin{proposition}\label{Proposition-Torsor-then-Central}
In a semi-abelian category, the underlying $n$-cubic extension of an $n$-torsor is H-central.
\end{proposition}
\begin{proof}
Let $(\TT,\tt)$ be an $n$-torsor of $Z$ by $A$ with underlying $n$-cubic extension~$T$. Then already the commutator $[T_{n},A]$ is zero: by Proposition~\ref{Proposition-Trivial-Action}, since $A$ is a trivial $Z$-module, and by Example~\ref{Example-Dimension-One}.

Now suppose that $\emptyset\neq I\subsetneq n$ and consider $x\colon {X=\bigcap_{j\in I}\K{\del_{j}}\to T_{n}}$ and $y\colon {Y=\bigcap_{j\in n\setminus I}\K{\del_{j}}\to T_{n}}$. We are to show that $x$ and $y$ Huq-commute (see Subsection~\ref{Commutators}), so that~$T$ is H-central.

Without any loss of generality we may assume that $\del_{0}y=0$. (If not, reverse the roles of $x$ and~$y$.) Let $i$ be the smallest element of $I$. Then~$\del_{i}x=0$ and~$\del_{j}y=0$ for all $j<i$, and the boundaries
\[
\del\sigma_{i-1}x=\lind\sigma_{i-2}\del_{0}x,\dots,\sigma_{i-2}\del_{i-2}x,x,x,0,\sigma_{i-1}\del_{i+1}x,\dots,\sigma_{i-1}\del_{n-1}x\rind
\]
and
\[
\del\sigma_{i}y=\lind 0,\dots,0,0,y,y,\sigma_{i}\del_{i+1}y,\dots,\sigma_{i}\del_{n-1}y\rind
\]
of $\sigma_{i-1}x$ and $\sigma_{i}y$ determine $(n,i)$-horns
\[
\overline{x}=\widehat{(\del\sigma_{i-1}x)}_{i}=\lind\sigma_{i-2}\del_{0}x,\dots,\sigma_{i-2}\del_{i-2}x,x,0,\sigma_{i-1}\del_{i+1}x,\dots,\sigma_{i-1}\del_{n-1}x\rind
\]
and
\[
\overline{y}=\widehat{(\del\sigma_{i}y)}_{i}=\lind0,\dots,0,0,y,\sigma_{i}\del_{i+1}y,\dots,\sigma_{i}\del_{n-1}y\rind
\]
in $\TT$ which Huq-commute with each other. The first $i+1$ components do so because anything commutes with zero. The others commute for the same reason, one of $\del_{j}x$ or $\del_{j}y$ being trivial by assumption on $x$ and $y$. In other words, there is a morphism~$\widehat\varphi_{i}$ such that the diagram
\[
\xymatrix{X \ar[rd]|-{\lind1_{X},0\rind} \ar@/^/[rrd]^-{\overline{x}} \\
& X\times Y \ar@{.>}[r]|-{\widehat\varphi_{i}} & \horn^{i}(\TT,n)\\
Y \ar[ru]|-{\lind0,1_{Y}\rind} \ar@/_/[rru]_-{\overline{y}}}
\]
is commutative, namely, the morphism determined by the family
\[
\varphi_{j}=\begin{cases}
\overline{x}_{j}\comp \pr_{X} & \text{if $j<i$ (so that $j\not\in I$)}\\
\overline{x}_{j}\comp \pr_{X} & \text{if $j>i$ and $j-1\not\in I$}\\
\overline{y}_{j}\comp \pr_{Y} & \text{if $j>i$ and $j-1\in I$;}
\end{cases}
\]
note that indeed $\del_{k}\comp \varphi_{j}=\del_{j-1}\comp\varphi_{k}$ for all $k<j$ such that $i\not\in\{j,k\}$. Furthermore, being induced by degeneracies, $\overline{x}$ and $\overline{y}$ compose to the face left out---see Subsection~\ref{Multiplication}---so that the diagram
\[
\xymatrix@C=3em{X \ar[rd]|-{\lind1_{X},0\rind} \ar@/^/[rrd]_-{\overline{x}} \ar@/^/[rrrd]^-{x} \\
& X\times Y \ar[r]|-{\widehat\varphi_{i}} & \horn^{i}(\TT,n) \ar[r]|(.6){m^{i}} & T_{n}\\
Y \ar[ru]|-{\lind0,1_{Y}\rind} \ar@/_/[rru]^-{\overline{y}} \ar@/_/[rrru]_-{y}}
\]
is commutative, and $x$ and $y$ Huq-commute.
\end{proof}

Thus we proved that, in a semi-abelian category, any categorically central truncated simplicial resolution gives a torsor (Proposition~\ref{Proposition-Central-then-Torsor}) and any torsor gives an H-central truncated simplicial resolution (Proposition~\ref{Proposition-Torsor-then-Central}). To complete the circle, what we need is precisely the commutator condition~(CC).

\begin{theorem}\label{Theorem-Torsor-Equivalence}
In a semi-abelian category which satisfies the commutator condition~{\rm (CC)}, an augmented simplicial object $\TT$ is part of an $n$-torsor $(\TT,\tt)$ if and only if its underlying $n$-fold arrow is an $n$-cubic central extension.\noproof
\end{theorem}

\begin{corollary}\label{Corollary-Torsor-Equivalence}
Under~{\rm (CC)}, the functor $\Dd_{(n,Z)}^{-1}A\to \Tors^{n}(Z,A)$ described in Proposition~\ref{Proposition-Full} is an equivalence of categories.
\end{corollary}
\begin{proof}
Theorem~\ref{Theorem-Torsor-Equivalence} tells us that this functor is essentially surjective, while it is fully faithful by Proposition~\ref{Proposition-Full}.
\end{proof}

\section{Cohomology classifies higher central extensions}\label{Section-Main-Theorem}
In this last section we prove our main result, Theorem~\ref{Main-Theorem}, which states that, for any object $Z$, any abelian object $A$, and any $n\geq 1$, we have a natural group isomorphism
\[
\H^{n+1}(Z,A)\cong\Centr^{n}(Z,A).
\]
To do so, we only need to use the results of the previous sections and establish a natural bijection between the underlying sets.

We already know that, for truncated simplicial resolutions, being a torsor is equivalent to being central. Now we have to explain how any (central) extension may be approximated by a truncated augmented simplicial object so that the two types of objects may be compared. In fact, any equivalence class of central extensions of $Z$ by $A$ contains a truncated simplicial object.

\subsection{Simplicial approximation of higher-dimensional arrows}
Using a classical Kan extension argument, every $n$-dimensional arrow may be universally approximated by an ${(n-1)}$-truncated simplicial object. Indeed, the functor from Subsection~\ref{Truncations}
\[
\arr_{n}=\Fun(-,\a_{n})\colon{\SimpArr^{n}(\X)\to \Arr^{n}(\X)}
\]
has a right adjoint
\[
\simparr_{n}=\Ran_{\a_{n}}(-)\colon{\Arr^{n}(\X)\to \SimpArr^{n}(\X)}
\]
which takes an $n$-fold arrow $F\colon{(2^{n})^{\op}\to \X}$ and maps it to the right Kan extension
\[
\Ran_{\a_{n}}F\colon{(\Delta^{+}_{n})^{\op}\to \X}
\]
of $F$ along the functor $\a_{n}\colon 2^{n}\to \Delta^{+}_{n}$.

\begin{proposition}\label{Proposition-Simplification-of-Extension}
Let $\X$ be a regular category with enough projectives. Then for all $n\geq 1$, the functors~$\arr_{n}$ and $\simparr_{n}$ preserve $n$-cubic extensions.
\end{proposition}
\begin{proof}
Since an ${(n-1)}$-truncated simplicial object is by definition an $n$-cubic extension if and only if so is its underlying $n$-fold arrow, the functor $\arr_{n}$ preserves and reflects $n$-cubic extensions. The case of $\simparr_{n}$ for $n\geq 2$ is more complicated: given an~$n$-fold arrow~$F$, the Kan extension $\simparr_{n}F=\Ran_{\a_{n}}F$ is computed pointwise as a limit (see for instance~\cite{MacLane}). For example, a two-cubic extension such as~\eqref{Double-Extension} has
\[
\xymatrix{\R{d}\times_{X}\R{c} \ar@<-1ex>[r] \ar@<1ex>[r] & X \ar[l] \ar[r] & Z}
\]
as its simplicial approximation. It is easily seen that also in general, $\simparr_{n}F$ has the morphism ${F_{n}\to F_{0}}$ as its augmentation. In order to obtain a quick formal proof in higher degrees we assume that $\X$ has enough projectives and use Proposition~\ref{Projective-Characterisation-Extensions}. Let $\XX$ be an $(n-1)$-truncated degreewise projective simplicial object and consider a collection of arrows ${(X_{J}\to (\simparr_{n}F)_{J})_{|J|\leq i}}$ as in the proposition. Composing with the counit $u\colon{\arr_{n}\simparr_{n}F\to F}$ of the adjunction at $F$ we obtain a collection of arrows to the $n$-cubic extension~$F$, which extends to a morphism of $n$-fold arrows ${X\to F}$. By adjointness we now obtain the needed morphism of $n$-fold arrows ${X\to \simparr_{n}F}$, extending the given collection of arrows. This proves that $\simparr_{n}F$ is an $n$-cubic extension.
\end{proof}

We now return to the semi-abelian context and prove that then these adjunctions also preserve centrality. These two results together extend Proposition~5.1 in~\cite{RVdL} to higher degrees and beyond the case of central extensions.

\begin{lemma}\label{Lemma Counit-square}
Suppose $F$ is an $n$-cubic extension, $G=\arr_{n}\simparr_{n}F$ and $u\colon{G\to F}$ is the counit of the adjunction at $F$. Let $\L F$ (respectively $\L G$) be the limit described in Subsection~\ref{HDA} and $l_{F}$ (respectively $l_{G}$) the comparison morphism. Then the square 
\begin{equation}\label{Counit-square}
\vcenter{\xymatrix{G_{n} \ar[r]^-{u_{n}} \ar@{ >>}[d]_-{l_{G}} & F_{n} \ar@{ >>}[d]^-{l_{F}}\\
\L G \ar[r]_-{\L u} & \L F}}
\end{equation}
is a pullback.
\end{lemma}
\begin{proof}
We prove that the pullback $\L G\times_{\L F}F_{n}$ is isomorphic to $G_{n}$ by showing that the $(n-1)$-truncated simplicial object $H$ which is equal to $G$ everywhere, except in level $n-1$ where it is $\L G\times_{\L F}F_{n}$, is actually isomorphic to the simplicial approximation of $F$.

The $n$-cubic extension $H$ is indeed an ${(n-1)}$-truncated simplicial object: the degeneracies are induced by composition of the degeneracies of $G$ with the comparison morphism ${G_{n}\to \L G\times_{\L F}F_{n}}$. This comparison morphism is part of a morphism of $(n-1)$-truncated simplicial objects ${G\to H}$. Its inverse ${H\to G}$ is now induced by the universal property of $G$.
\end{proof}

\pagebreak
\begin{proposition}\label{Proposition-Simplification-of-Central-Extension}
Let $\X$ be a semi-abelian category with enough projectives. For all $n\geq 1$, the functors $\arr_{n}$ and $\simparr_{n}$ preserve centrality. Furthermore, both functors preserve the direction of a central extension.
\end{proposition}
\begin{proof}
Let $F$ be an $n$-cubic central extension. Then the direction $A=\bigcap_{i\in n}\K{f_{i}}$ of~$F$ (Lemma~\ref{Direction-as-Kernel}) is part of the short exact sequence
\[
\xymatrix{0 \ar[r] & A \ar@{{ |>}->}[r] & F_{n} \ar@{ >>}[r]^-{l_{F}} & \L F \ar[r] & 0.}
\]
By Lemma~\ref{Lemma Counit-square} we have that the square~\eqref{Counit-square} is a pullback. Via Lemma~\ref{Lemma-Direction-Limit}, this already implies that $l_{G}$ is a central extension when $F$ is central; moreover, we have $A=\bigcap_{i\in n}\K{g_{i}}$ by Lemma~\ref{Lemma-Iso-Pullback}, so that the functor~$\simparr_{n}$ preserves directions. Since pullbacks preserve product projections, by Theorem~\ref{Theorem-Higher-Centrality} we only need to prove that any square
\[
\xymatrix{\bigboxvoid_{i\in n}\R{g_{i}} \ar[r] \ar[d]_-{\pi^{I}_{G}} & \bigboxvoid_{i\in n}\R{f_{i}} \ar[d]^-{\pi^{I}_{F}} \\
\bigboxdot^{I}_{i\in n}\R{g_{i}} \ar[r] & \bigboxdot^{I}_{i\in n}\R{f_{i}}}
\]
is a pullback. To see this, consider the commutative cube
\[
\xymatrix@!0@R=3em@C=4em{& \bigboxvoid_{i\in n}\R{g_{i}} \ar[ld] \ar[rr] \ar@{.>}[dd]^(.25){\pi^{I}_{G}} && \bigboxvoid_{i\in n}\R{f_{i}} \ar[dd]^-{\pi^{I}_{F}} \ar[ld]\\
G_{n} \ar[rr]_(.25){u_{n}} \ar[dd]_-{l_{G}} && F_{n} \ar[dd]^(.25){l_{F}}\\
& \bigboxdot^{I}_{i\in n}\R{g_{i}} \ar@{.>}[rr] \ar@{.>}[ld] && \bigboxdot^{I}_{i\in n}\R{f_{i}} \ar[ld]\\
\L G \ar[rr]_-{\L u} && \L F}
\]
in which the left and right hand side faces are the ones of Lemma~\ref{Lemma-Diamond-Pullback}. These squares are pullbacks, and since we already proved that the front face~\eqref{Counit-square} is a pullback as well, the claim follows.
\end{proof}

\subsection{The equivalence between central extensions and torsors}
By Proposition~\ref{Proposition-Simplification-of-Central-Extension} the functors $\arr_{n}$ and $\simparr_{n}$ not only preserve central extensions, but also the directions of those central extensions. Hence for any object~$Z$ and any abelian object $A$, these functors (co)restrict to an adjunction
\[
\xymatrix@1@=4em{{\Dd^{-1}_{(n,Z)}A} \ar@<1ex>[r]^-{\arr_{n}} \ar@{}[r]|-{\perp} & {\D^{-1}_{(n,Z)}A} \ar@<1ex>[l]^-{\simparr_{n}}}
\]
where the functor~\eqref{Dd}
\[
\Dd_{(n,Z)} \colon{\SimpCExt^n_Z(\X)\to \Ab(\X)}
\]
is the restriction of $\D_{(n,Z)}$ to those $3^{n}$-diagrams which, as an $n$-cubic central extension, are an $(n-1)$-truncated simplicial object. 
Taking connected components gives a bijection of sets (see Remark~5.2 in~\cite{RVdL}). By Corollary~\ref{Corollary-Torsor-Equivalence} this bijection takes the shape
\begin{equation}\label{bijection}
\pi_{0}\Tors^{n}(Z,A)\cong\pi_{0}(\Dd^{-1}_{(n,Z)}A)\cong\pi_{0}(\D^{-1}_{(n,Z)}A)
\end{equation}
when also the commutator condition~(CC) holds.

\pagebreak
\begin{proposition}
The bijection~\eqref{bijection} is natural in $A$.
\end{proposition}
\begin{proof}
In~\cite[Section 4]{Duskin-Torsors} the functor $\pi_{0}\Tors^{n}(Z,-)$ is defined as follows: given $f\colon{A\to B}$ in $\Ab(\X)$, an $n$-torsor $(\XX,\xx)$ of $Z$ by $A$ universally induces an $n$-torsor $f_{*}(\XX,\xx)$ of~$Z$ by~$B$. The construction in Proposition~\ref{Proposition-Centr^{n}(Z,-)} gives us another $n$-torsor $(\YY,\yy)$ of $Z$ by $B$ together with a simplicial morphism $\ff\colon {\XX\to\YY}$ over $\KK(f,n)$ which, by the universal property defining $f_{*}(\XX,\xx)$, yields a morphism $f_{*}(\XX,\xx)\to (\YY,\yy)$ of $n$-torsors of $Z$ by $B$. Thus $f_{*}(\XX,\xx)$ and $(\YY,\yy)$ end up in the same equivalence class, so that the bijection~\eqref{bijection} is natural in~$A$.
\end{proof}

Thus we see that the underlying sets of the abelian groups
\[
\H^{n+1}(Z,A) = \Tors^{n}[Z,A] = \pi_{0}\Tors^{n}(Z,A)
\]
and $\Centr^{n}(Z,A)=\pi_{0}(\D^{-1}_{(n,Z)}A)$ are naturally isomorphic. Since both
\[
\H^{n+1}(Z,-)\quad\text{and}\quad \Centr^{n}(Z,-)\colon \Ab(\X)\to \Set
\]
are product-preserving functors (Proposition~\ref{Proposition-Centr^{n}(Z,-)}), they lift to naturally isomorphic functors~${\Ab(\X)\to\Ab}$, which gives us

\begin{theorem}\label{Main-Theorem}
Let $Z$ be an object and $A$ an abelian object in a semi-abelian category with enough projectives satisfying the commutator condition {\rm (CC)}. Then for every $n\geq 1$ we have an isomorphism $\H^{n+1}(Z,A)\cong\Centr^{n}(Z,A)$, natural in $A$.\noproof
\end{theorem}

Thus we obtain the claimed ``duality'' between internal homology and external cohomology.

\begin{theorem}\label{Duality-Theorem}
Consider $n\geq 1$ and let $Z$ be an object in a semi-abelian category with enough projectives~$\X$ which satisfies the commutator condition {\rm (CC)}. Then
\[
\H_{n+1}(Z,\Ab(\X))=\lim \D_{(n,Z)}\qquad\text{and}\qquad \H^{n+1}(Z,A)\cong\pi_{0}(\D^{-1}_{(n,Z)}A),
\]
where $A$ is any abelian object in $\X$. When, in particular, $\X$ is monadic over~$\Set$, the homology and the cohomology are comonadic Barr--Beck (co)ho\-mo\-lo\-gy with respect to the canonical comonad on $\X$.
\end{theorem}
\begin{proof}
The equality holds by definition, while the isomorphism is Theorem~\ref{Main-Theorem}. The interpretation in terms of comonadic (co)homology is the main result of~\cite{GVdL2} in the case of homology, and of~\cite{Duskin, Duskin-Torsors} in the case of cohomology.
\end{proof}

\subsection*{Acknowledgement}
We would like to thank Tomas Everaert and George Peschke for interesting comments and suggestions, and the referee for careful guidance leading to this revised version of the paper.


\providecommand{\noopsort}[1]{}
\providecommand{\bysame}{\leavevmode\hbox to3em{\hrulefill}\thinspace}
\providecommand{\MR}{\relax\ifhmode\unskip\space\fi MR }
\providecommand{\MRhref}[2]{%
  \href{http://www.ams.org/mathscinet-getitem?mr=#1}{#2}
}
\providecommand{\href}[2]{#2}

\end{document}